\def\longversion{0} 
\def\diffcolor{black} 

\documentclass[10pt]{article}
\usepackage[utf8]{inputenc}
\usepackage[T1]{fontenc}
\usepackage{comment}
\usepackage{pgfplots} 
\usepackage{adjustbox}
\pgfplotsset{compat=1.15}
\usepackage{mathrsfs}
\usetikzlibrary{arrows}

\usepackage{marvosym}
\usepackage{dsfont}
\usepackage{rotating}
\usepackage{amssymb}
\usepackage{amsmath}
\usepackage{amsthm}
\usepackage{tabularx}
\usepackage{verbatim}
\usepackage[linesnumbered, ruled, boxed, vlined]{algorithm2e}
\usepackage{enumitem}
\usepackage{framed}
\usepackage{url}
\usepackage{float}
\usepackage{mathtools}
\usepackage{array}
\usepackage[colorlinks=true, linkcolor=blue, citecolor=blue, urlcolor=blue]{hyperref}
\usepackage{cleveref}

\DeclareMathOperator{\theend}{end}

\newcommand{\first}{\deg^+}
\newcommand{\second}{\deg^-}
\DeclareMathOperator{\Cov}{Cov}
\DeclareMathOperator{\Var}{Var}

\DeclareMathOperator{\Seq}{Seq}

\DeclareMathOperator{\myskew}{skew}
\DeclareMathOperator{\Ima}{Im}
\DeclareMathOperator{\BAD}{BAD}
\DeclareMathOperator{\troot}{root}
\DeclareMathOperator{\Deg}{Deg}

\newcommand{\Tend}{T_{\theend}}

\newcommand{\towriteornottowrite}[2]{\ifthenelse{\equal{\longversion}{1}}{{\color{\diffcolor}#1}}{{\color{\diffcolor}#2}}}

\newcommand{\Fast}{F^\ast}

\usepackage{xcolor}
\definecolor{mygreen}{rgb}{0.01, 0.75, 0.24}
\usepackage{tikz}
\usepackage{tikz-3dplot}

\def\mycolor{black}

\newcommand{\newdot}[2]{
	\node[shape=circle, minimum size = 5pt, inner sep=0pt, outer sep=0pt, fill=\mycolor, color=\mycolor] (#1) at (#2){};
}

\newcommand{\codegree}{\Delta}
\newcommand{\onedegree}{\Delta_1}
\newcommand{\twodegree}{\Delta_2}

\usepackage[a4paper, margin=2.5cm,top=3cm, bottom=3cm]{geometry}

\setlist[enumerate,1]{label={\textnormal{(\roman*)}}}

\theoremstyle{plain}

\usepackage{aliascnt}

\newtheorem{theorem}{Theorem}[section]

\newaliascnt{lemma}{theorem}
\newtheorem{lemma}[lemma]{Lemma}
\aliascntresetthe{lemma}
\crefname{lemma}{lemma}{lemmas}
\Crefname{lemma}{Lemma}{Lemmas}

\newaliascnt{proposition}{theorem}
\newtheorem{proposition}[proposition]{Proposition}
\aliascntresetthe{proposition}
\crefname{proposition}{proposition}{propositions}
\Crefname{proposition}{Proposition}{Propositions}

\newaliascnt{property}{theorem}

\aliascntresetthe{property}
\crefname{property}{property}{properties}
\Crefname{property}{Property}{Properties}

\newaliascnt{claim}{theorem}
\newtheorem{claim}[claim]{Claim}
\aliascntresetthe{claim}
\crefname{claim}{claim}{claims}
\Crefname{claim}{Claim}{Claims}

\newaliascnt{conjecture}{theorem}

\aliascntresetthe{conjecture}
\crefname{conjecture}{conjecture}{conjectures}
\Crefname{conjecture}{Conjecture}{Conjectures}

\newaliascnt{observation}{theorem}
\newtheorem{observation}[observation]{Observation}
\aliascntresetthe{observation}
\crefname{observation}{observation}{observations}
\Crefname{observation}{Observation}{Observations}

\newaliascnt{question}{theorem}

\aliascntresetthe{question}
\crefname{question}{question}{questions}
\Crefname{question}{Question}{Questions}

\newaliascnt{corollary}{theorem}
\newtheorem{corollary}[corollary]{Corollary}
\aliascntresetthe{corollary}
\crefname{corollary}{corollary}{corollaries}
\Crefname{corollary}{Corollary}{Corollaries}

\theoremstyle{definition}

\newaliascnt{definition}{theorem}
\newtheorem{definition}[definition]{Definition}
\aliascntresetthe{definition}
\crefname{definition}{definition}{definitions}
\Crefname{definition}{Definition}{Definitions}

\newaliascnt{remark}{theorem}
\newtheorem{remark}[remark]{Remark}
\aliascntresetthe{remark}
\crefname{remark}{remark}{remarks}
\Crefname{remark}{Remark}{Remarks}

\newaliascnt{fact}{theorem}
\newtheorem{fact}[fact]{Fact}
\aliascntresetthe{fact}
\crefname{fact}{fact}{facts}
\Crefname{fact}{Fact}{Facts}

\crefname{exercise}{exercise}{exercises}
\Crefname{exercise}{Exercise}{Exercises}

\AtBeginDocument{%
	\def\MR#1{}
}

\newcommand{\oldqed}{}
\def\endofClaim{\hfill\scalebox{.6}{$\Box$}}

\newenvironment{claimproof}[1][Proof]{
	\renewcommand{\oldqed}{\qedsymbol}
	\renewcommand{\qedsymbol}{\endofClaim}
	\begin{proof}[#1]
	}{
	\end{proof}
	\renewcommand{\qedsymbol}{\oldqed}
}

\def\textcol{black}
\newcommand{\textbe}[2]{
	\node[label=below:{\textcolor{\textcol}{#2}}] at (#1) {};
}
\newcommand{\textab}[2]{
	\node[label=above:{\textcolor{\textcol}{#2}}] at (#1) {};
}
\newcommand{\textat}[2]{
	\node at (#1) {\textcolor{\textcol}{#2}};
}
\usepackage{savesym}
\savesymbol{textle}
\newcommand{\textle}[2]{
	\node[label=left:{\textcolor{\textcol}{#2}}] at (#1) {};
}
\newcommand{\textri}[2]{
	\node[label=right:{\textcolor{\textcol}{#2}}] at (#1) {};
}

\newcommand{\floor}[1]{\left\lfloor#1\right\rfloor}
\newcommand{\ceil}[1]{\left\lceil#1\right\rceil}
\newcommand\eps{\varepsilon}
\newcommand\N{\mathbb{N}}
\newcommand{\E}{\mathbb{E}}
\newcommand{\coleq}{\coloneqq}
\renewcommand\P{\mathbb{P}}
\newcommand\cP{\mathcal{P}}
\newcommand\cF{\mathcal{F}}
\newcommand\cC{\mathcal{C}}

\newcommand\cS{\mathcal{S}}
\renewcommand\O{\mathcal{O}}
\newcommand\cG{\mathcal{G}}

\newcommand\cK{\mathcal{K}}

\newcommand{\calZ}{\mathcal Z}
\newcommand{\calE}{\mathcal E}
\newcommand{\F}{\mathcal F}
\newcommand{\Z}{\mathbb Z}
\newcommand{\C}{\mathcal C}

\newcommand{\B}{\mathcal B}
\newcommand{\Y}{\mathcal Y}

\newcommand{\subs}{\subseteq}
\newcommand{\abs}[1]{\mathopen{}\left\lvert#1\right\rvert\mathclose{}}

\newcommand{\extratightcycle}[2]{EC_{#1}^{(#2)}}
\newcommand{\extratightpath}[2]{EP_{#1}^{(#2)}}

\newcommand{\set}[2]{\{#1\,:\;#2\}}

\DeclareMathOperator{\super}{super}
\def\claimproof{\removelastskip\penalty55\medskip\noindent{\em Proof of claim: }}
\def\noclaimproof{{\unskip\nobreak\hfill\penalty50\hskip2em\hbox{}\nobreak\hfill%
		$-$\parfillskip=0pt\finalhyphendemerits=0\par}\goodbreak}

\def\endclaimproof{\noclaimproof\medskip}
\newcommand{\Gto}[1]{G^{(#1)}}
\newcommand{\Gskew}{G_{\myskew}}
\newcommand{\Gsuper}{G_{\super}}

\usetikzlibrary{arrows.meta, positioning,decorations.markings}
\tikzset{
	theorem/.style={
		rectangle,
		rounded corners,
		draw=black,
		fill=blue!10,
		very thick,
		minimum width=3cm,
		text width=3.5cm,
		align=center,
		inner sep=5pt
	},
	arrow/.style={
		-{Stealth[length=3mm]},
		thick
	},
	edgelabel/.style={
		text width=2.5cm,
		align=center,
		fill=white,  
		draw=black!30,   
		rounded corners,
		inner sep=2pt
	}
}
\newlist{proofoutline}{enumerate}{1}
\setlist[proofoutline,1]{label=\textbf{Step \arabic*}, leftmargin=*}

\author{Stefan Glock\thanks{Fakultät für Informatik und Mathematik, Universität Passau, Innstraße 41, 94032 Passau, Germany.\\ \textit{Email:} \href{mailto:stefan.glock@uni-passau.de}{\texttt{stefan.glock@uni-passau.de}}.}
	\and Olaf Parczyk\thanks{Fachbereich Mathematik und Informatik, Freie Universität Berlin, Arnimallee 3, 14195 Berlin, Germany.\\ 
		\textit{Emails:} \href{mailto:parczyk@mi.fu-berlin.de}{\texttt{parczyk@mi.fu-berlin.de}}, 
		\href{mailto:szabo@mi.fu-berlin.de}{\texttt{szabo@mi.fu-berlin.de}}, 
		\href{mailto:s.rathke@fu-berlin.de}{\texttt{s.rathke@fu-berlin.de}}.}
	\thanks{Department AI in Society, Science, and Technology, Zuse Institute Berlin, Takustraße 7, 14195 Berlin, Germany.}
	\and Silas Rathke\footnotemark[2] 
	\and Tibor Szabó\footnotemark[2]}

\date{}

\title{The maximum diameter of $d$-dimensional simplicial complexes}

\begin{document}
	
	\maketitle
	
	\begin{abstract}
		For every fixed dimension $d$ and sufficiently large $n$, we 
		determine the maximum possible diameter of a strongly connected $d$-dimensional simplicial complex on $n$ vertices. 
		This improves on a sequence of previous results and settles a problem of Santos from 2013. On the way, as a special case, we also characterise the existence of an extra-tight Euler tour 
		in the complete $d$-uniform hypergraph on $n$ vertices.
	\end{abstract}
	
	\section{Introduction}
	
	The Polynomial Hirsch Conjecture is a central problem of Discrete Geometry, stating that the diameter of the vertex/edge graph of a polytope is at most a polynomial in the number of its facets and the dimension. While this conjecture is certainly captivating for its own sake, its broader importance stems from its direct connection to one of the most fascinating major open problems of Discrete Optimization: the polynomiality of the Simplex Method. Should the conjecture turn out to be false, then no pivot rule can be polynomial in the worst case. (The original conjecture of Hirsch from 1957, stating the specific upper bound of $n-d$, was famously disproved by Santos~\cite{HirschCounterexample} in 2010.) 
	Due to the apparent difficulty of understanding the fine geometry of polytopes, most approaches towards bounding the diameter rely on combinatorial abstractions. 
	These strip away much of the geometry and focus solely on well-chosen combinatorial properties of the set of facets. The classic quasi-polynomial upper bound of Kalai and Kleitman~\cite{KalaiKleitman}, which is the basis of the subsequent best known bounds (Todd~\cite{Todd2024}, Sukegawa~\cite{Sukegawa2019}), also turns out to be of this type.  The combinatorial {\em base abstractions} of Eisenbrand, H\"ahnle, Razborov, and Rothvo\ss~\cite{Eisenbrand2010} generalize many of the known abstractions into the same simple framework.  The Polymath3 project \cite{polymath3} initiated by Gil Kalai in 2009 was specifically targeted towards identifying and making progress on various combinatorial abstractions.
	
	In this paper, we consider the diameter problem for abstract simplicial complexes, a natural abstraction studied by Santos, and provide the precise answer. Our investigations lead us to define a certain highly regular, global combinatorial structure and we characterise its existence in the complete uniform hypergraph. 
	
	A \emph{simplicial complex} on $n$ vertices is a family $\C$ of subsets of $[n]$ which is closed under taking subsets. The maximal elements of $\C$ are called \emph{facets}. We call $\C$ a \emph{(simplicial) $d$-complex}, if each of its facets has size $d+1$.
	The \emph{dual graph} $G(\C)$ of a simplicial $d$-complex $\C$ is defined on the set ${\cal F}(\cC)$ of facets as its vertex set, with two facets forming an edge in $G(\C)$ if their intersection has size~$d$. (Sets of size $d$ are also called \emph{ridges} and the dual graph is sometimes referred to as the facet-ridge graph.)
	We say $\C$ is \emph{strongly connected} if $G(\C)$ is connected.
    The \emph{diameter} of $\C$ is the diameter of its dual graph $G(\C)$ (i.e., the maximum, taken over all pairs $u,v\in V(G(\C))$ of vertices, of the length of a shortest $u$-$v$-path).
    Santos~\cite{santosdefinesH} defined $H_s(n,d)$ to be the maximum diameter of a strongly connected simplicial $d$-complex on~$[n]$ (note that in his notation, $d$ is shifted by one). We refer the reader to \Cref{sec: Conclusion} for a description of the connection to the diameter of polytopes.  
	Santos~\cite{santosdefinesH} proved for fixed $d\ge 2$ that
	\begin{equation*}\label{eq:santosbounds} \Omega\!\left(n^{\frac{2d+2}{3}}\right)\le H_s(n,d) \le \frac{1}{d}\binom nd.
	\end{equation*} 
	For the upper bound, his simple volume argument in fact implies
	\begin{equation}\label{eq:upper bound}
		H_s(n,d)\le \floor{\frac{1}{d}\binom{n}{d}-\frac{d+1}{d}}.
	\end{equation}
	For the lower bound, Santos~\cite{santosdefinesH} gives an explicit product construction.
    An alternative lower bound of the order $n^{d/4}$ is contained in a paper by Kim~\cite{MR3152073} (cf.~\cite{criado2017maximum}). 
    Subsequently, a series of authors used a variety of different tools to construct complexes of larger and larger diameter. 
	Criado and Santos~\cite{criado2017maximum} gave an explicit algebraic construction of simplicial $d$-complexes using finite fields, whose diameter, $\Theta(n^d)$, matched the order of magnitude of the upper bound for every fixed~$d$. Criado and Newman~\cite{criado2021randomized} introduced probabilistic constructions to the problem and reduced the gap between the upper and lower bounds, from a factor exponential in $d$ to $\mathcal O(d^2)$.
	Most recently, Bohman and Newman~\cite{bohman2022complexes} managed to pin down the precise asymptotics for every fixed $d\geq 2$ applying the differential equations method to track the evolution of a random greedy algorithm to construct the desired $d$-complex with large diameter:  
	\begin{equation}\label{eq:bohmanbound}  
		\left( \frac{1}{d} - (\log n)^{-\eps} \right) \binom{n}{d} \le H_s(n,d) \, ,
	\end{equation}
	where $\eps < 1/d^2$ and $n$ is sufficiently large.
    We note (cf.~\cite{2-dimensional}) here that an earlier result of D{\k{e}}bski, Lonc, and Rz{\k{a}}{\.z}ewski~\cite{dkebski2017harmonious} on harmonious colorings of fragmentable $k$-uniform hypergraphs also can be used to derive an asymptotically precise lower bound of $\left(\frac{1}{d}-o(1)\right) {\binom n d}$ on $H_s(n,d)$, without an explicit error term. Just very recently, Gould and Kelly~\cite{gould2025advancing} improved the error term in \eqref{eq:bohmanbound} from logarithmic to $\mathcal O(n^{-1/d})$ (as an application of a general theorem on certain hypergraph matchings).
	
	In \cite{2-dimensional}, the last three authors gave explicit constructions of simplicial $2$-complexes whose diameter attains the upper bound \eqref{eq:upper bound} for all $n$, except for $6$:
	\begin{equation}\label{eq: H_s(n,2)}H_s(n,2) = \begin{cases}
			\floor{\frac{1}{2}\binom{n}{2}-\frac{3}{2}} & n\ne 6\\
			5=\floor{\frac{1}{2}\binom{6}{2}-\frac{3}{2}}-1 & n=6.
	\end{cases}\end{equation}
	
	In \cite{2-dimensional} it was also conjectured that the simple upper bound \eqref{eq:upper bound} of Santos can in fact be achieved for all fixed $d$, as long as $n$ is large enough. Here we prove this conjecture.
	\begin{theorem}\label{theo: precise value of H(n d)}
		For every positive integer $d\ge 2$, there exists a positive integer $n_0$ such that for all $n>n_0$, \[H_s(n,d)=\floor{\frac{1}{d}\binom{n}{d}-\frac{d+1}{d}}.\]
	\end{theorem}

	Our construction of a $d$-dimensional simplicial complex for Theorem~\ref{theo: precise value of H(n d)} starts out with a reduction to a design theoretic problem on almost-complete $d$-uniform hypergraphs. 
	For this, the concept of ``turns'', which was used in \cite{2-dimensional} to handle the ad hoc analysis of certain cases of the proof, will be instrumental and employed in a far more abstract setting. With them, we will be able to construct a relatively short simplicial complex whose dual graph is a path and has a given modular degree sequence. The rest of the simplicial complex can then be obtained through the existence of a structure we call extra-tight Euler trails in almost-complete hypergraphs. The proof of this existence theorem (Theorem~\ref{theo: large min degree implies extra-tight trail}) takes up a significant portion of our paper.
	
	The proliferation of ad hoc complications that arise in some of the $2$-dimensional constructions of \cite{2-dimensional} suggests that an explicit approach in higher dimension is not likely to be workable. 
	And indeed, our construction for Theorem~\ref{theo: large min degree implies extra-tight trail} employs the probabilistic technique of absorption. For its proof we take inspiration from several earlier papers \cite{gishboliner2023tight, glock2020euler, glock2023existence}, and combine concepts and ideas developed there with new ingredients. 
	
	{\bf Remark.} As it was noted by Santos~\cite{santosdefinesH}, the maximum diameter $H_s(n,d)$ is equal to the length of the longest induced path in the Johnson graph $J(n,d+1)$, which is thus also determined in Theorem~\ref{theo: precise value of H(n d)}.

	\subsection{Hypergraphs}
	
	Simplicial $d$-complexes $\cC$ are in one-to-one correspondence with $(d+1)$-graphs ${\cal F}$ where ${\cal F}$ is the set of facets of $\cC$ and $\cC$ is the simplicial $d$-complex $\langle {\cal F} \rangle : = \{ F' : \exists F \in \cF, F' \subseteq F \}$ generated by $\cF$. 
	We define the dual graph $G(\mathcal F)$ and diameter of a $(d+1)$-graph ${\cal F}$ to be those of 
	the $d$-complex $\langle {\cal F}\rangle$.
	The {\em shadow} $\partial \cF$ of $\cF$ is the $d$-graph containing all those $d$-sets that are subsets of a member of $\cF$. 
	
	Since for any $(d+1)$-graph $\cG$, the dual graph of a subfamily $\cF \subseteq \cG$ corresponding to the vertices of a shortest path in $G(\cG )$ is a path itself, the maximum diameter $H_s(n,d)$ can always be attained by a $(d+1)$-family whose dual graph is a path.  
	The first part of the next observation gives the simple argument for \eqref{eq:upper bound} when the dual graph is a path, which thus implies the upper bound in \Cref{theo: precise value of H(n d)}. Moreover, we also obtain simple conditions, in terms of the shadow, for attaining this bound. 
	
	\begin{observation}\label{obs: size of F and its shadow}
		Let $\cF$ be a $(d+1)$-graph on vertex set $[n]$. 
		\begin{enumerate}
			\item\label{enum: path vs shadow} If $G(\cF)$ is a path, then $|\partial \cF | = |\cF| \cdot d +1$. In particular, the length of the path is at most  $\floor{\frac{1}{d}\binom{n}{d}-\frac{d+1}{d}}$ and equality holds if and only if ${\binom n d} - |\partial \cF| \leq d-1$.
			\item\label{enum: cycle vs shadow} If $G(\cF)$ is a cycle, then $|\partial \cF | = |\cF| \cdot d$. In particular, if $\partial \cF = {\binom{[n]}d}$ then for any $F\in \cF$ the $(d+1)$-graph $\cF'= \cF \setminus \{ F\}$ has diameter $\floor{\frac{1}{d}\binom{n}{d}-\frac{d+1}{d}}$. 
		\end{enumerate}
	\end{observation}
	
	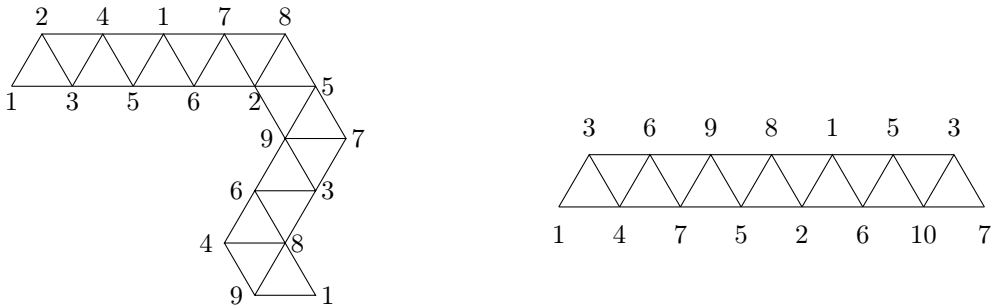
\begin{figure}[hbt]
		\centering
		\begin{tikzpicture}[scale=0.8]
			\foreach \x in {0,...,6} {
				\begin{scope}[shift={(\x,0)}]
					\draw (0,0) -- (0:1);
					\draw[] (0:1) -- (60:1) -- (0,0);
				\end{scope}
			}
			\draw (0.5,0.86603) -- (6.5,0.86603);
			\textbe{0,0}{1}
			\textbe{1,0}{4}
			\textbe{2,0}{7}
			\textbe{3,0}{5}
			\textbe{4,0}{2}
			\textbe{5,0}{6}
			\textbe{6,0}{10}
			\textbe{7,0}{7}
			\textab{0.5,0.86603}{3}
			\textab{1.5,0.86603}{6}
			\textab{2.5,0.86603}{9}
			\textab{3.5,0.86603}{8}
			\textab{4.5,0.86603}{1}
			\textab{5.5,0.86603}{5}
			\textab{6.5,0.86603}{3}
			\begin{scope}[shift={(-9,2)}]
				\foreach \x in {0,...,4} {
					\begin{scope}[shift={(\x,0)}]
						\draw (0,0) -- (0:1);
						\draw[] (0:1) -- (60:1) -- (0,0);
					\end{scope}
				}
				\draw (0.5,0.86603) -- (4.5,0.86603) (4,0) -- (4.5,-0.866) -- (5,0) -- (5.5,-.866) -- (4.5,-0.866) -- (5,-1.732) -- (5.5,-.866) (4.5,-.866) -- (4,-1.732) -- (5, -1.732) -- (4.5,-2.598) -- (3.5,-2.598) -- (4,-1.732) -- (4.5,-2.598) -- (4,-3.464)  (3.5,-2.598) -- (4,-3.464) -- (5,-3.464) -- (4.5,-2.598);
				\textbe{0,.2}{1}
				\textbe{1,.2}{3}
				\textbe{2,.2}{5}
				\textbe{3,.2}{6}
				\textbe{4,.2}{2}
				\textat{5.2,0}{5}
				\textat{4.2,-0.866}{9}
				\textat{3.7,-1.732}{6}
				\textat{3.2,-2.598}{4}
				\textat{5.7,-0.866}{7}
				\textat{5.2,-1.73}{3}
				\textat{4.7,-2.598}{8}
				\textat{5.2,-3.464}{1}
				\textat{3.7,-3.464}{9}
				\textab{.5,.7}{2}
				\textab{1.5,.7}{4}
				\textab{2.5,0.7}{1}
				\textab{3.5,0.7}{7}
				\textab{4.5,0.7}{8}
			\end{scope}
		\end{tikzpicture}
		\caption{On the left: A strongly connected simplicial $2$-complex on $[9]$ with maximum diameter. The facets are given by the triangles. To ensure that the dual graph is a path, each pair in $\binom{[9]}{2}$ appears at most once as an edge. The diameter is maximum by \Cref{obs: size of F and its shadow} because only the pair $\{4,7\}$ does not appear as an edge.
			On the right: A (non-optimal) straight simplicial $2$-complex on $[10]$.}
		\label{fig:straight vs. non-straight}
	\end{figure}
	
	\begin{proof}
		For part \ref{enum: path vs shadow}, we consider the members $F_1, \ldots , F_{|\cF|}$ of $\cF$ in the order they appear on the path of the dual graph. The set $F_1$ contains $(d+1)$ $d$-subsets and each subsequent $F_i$ contains $d$ such $d$-subsets that were not contained in any of the previous $F_j$s (otherwise $G(\cF)$ contains a cycle).  The bound on the length $|\cF| -1$ of the path then follows since $\partial \cF \subseteq {\binom{[n]}d}$.  Equality holds if and only if $\frac1 d(|\partial \cF| -1) =  |\cF| =\floor{\frac{1}{d}\left({\binom n d} - 1\right)}$, which is equivalent to ${\binom n d}-(d-1) \leq |\partial \cF|$. \\
		For part \ref{enum: cycle vs shadow}, we consider the members of $\cF$ in an order as they appear on the cycle of the dual graph. The counting of $d$-subsets is the same as in part \ref{enum: path vs shadow} until the very last vertex, which has two neighbours the $d$-subsets of which were already accounted for. So for the last vertex of the cycle we need to count only $(d-1)$ new $d$-subsets, hence the formula for the shadow follows. Removing $F$ from $\cF$  makes the dual graph into a path and from the shadow $(d-1)$ of the $d$-subsets get removed. Now if $\partial \cF = {\binom{[n]} d}$, then $|\partial (\cF \setminus \{ F \} )| = {\binom n d} -d+1$ so the diameter of $\cF \setminus \{ F \}$ is $\floor{\frac{1}{d}\binom{n}{d}-\frac{d+1}{d}}$ by part \ref{enum: path vs shadow}. 
	\end{proof}
	
	A special, vertex-sequential way of constructing a $(d+1)$-family with a path dual graph is via an appropriate sequence $T=(v_1, \dots , v_k) \in [n]^{k}$ of (not necessarily distinct) vertices, where one takes the sets formed by $d+1$ consecutive entries of $T$. 
	To have a simplicial $d$-complex, one should of course assume that no element of $[n]$ repeats within $d+1$ consecutive entries of $T$, but then consecutive $(d+1)$-sets automatically have a $d$-element intersection, so the dual graph 
	{\em contains} a spanning path (of length $k-d-1$). We are interested in the case when there are no more edges in it.
	
	\begin{definition}
		\label{def: straight}
		For a sequence $T=(v_1, \dots , v_k) \in [n]^{k}$, let $\cS_T = \cS_T^{(d+1)} := \{\{ v_i, v_{i+1}, \ldots , v_{i+d}\} : i=1, \ldots, k-d\}$.
		If the hypergraph $\cS_T$ is $(d+1)$-uniform and $G(\cS_T)$ is a path, then $\cS_{T}$, as well as the $d$-complex it generates,  is called {\em straight}.
	\end{definition}
	
	The $d$-complexes of \cite{criado2021randomized,criado2017maximum, bohman2022complexes}, achieving maximum diameter in the order of magnitude, are all of this special form. 
	The extremal $2$-complexes of \cite{2-dimensional}  also contain significant straight portions, but, for most values of $n$, are not entirely of that type. 
	The situation will be similar for our $d$-complexes establishing the precise \Cref{theo: precise value of H(n d)}. 
	This is because $(d+1)$-graphs achieving maximum diameter can only be straight for certain values of $n$, those satisfying some divisibility conditions.  
	Yet, as we will see, a straight $(d+1)$-graph of a maximum diameter {\em is} possible for an infinite sequence of $n$, leading us to the definition and construction of certain highly regular combinatorial structures we call extra-tight Euler tours. 
	
	The key property of straight simplicial complexes, that 
	their dual graph $G(\cS_T)$ does not contain {\em more} edges than its canonical spanning path, means that the non-neighbouring $(d+1)$-intervals of $T$ must have at most $d-1$ common elements.
	This is most conveniently formulated in the realm of $d$-graphs: 
	every $d$-subset of $[n]$ should occur at most once among $d+1$ consecutive entries of $T$.

	\begin{definition}\label{def: extra-tight trail tour}
		Given a $d$-graph $H$, an \emph{extra-tight trail} $T$ in $H$ is a sequence $(v_1,\dots,v_k)$ of vertices of $H$ such that the sets $e_{\iota,\sigma}\coleq \{v_\iota,\dots,v_{\iota+d}\}\backslash\{v_{\iota+\sigma}\}$ with $\iota\in[k-d]$, $\sigma\in[d]$ and the set $e_{k-d+1,d}\coleq \{v_{k-d+1},\dots,v_k\}$ are all pairwise distinct and in $H$.
        The \emph{ends} of $T$ are $(v_1,\dots,v_d)$ and $(v_{k-d+1},\dots,v_k)$.
		An \emph{extra-tight tour} $C$ in $H$ is a (cyclic) vertex sequence $(v_1,\dots,v_k)$ such that the edges $e_{\iota,\sigma}\coleq \{v_\iota,\dots,v_{\iota+d}\}\backslash\{v_{\iota+\sigma}\}$ with $\iota\in[k]$, $\sigma\in[d]$ are all pairwise distinct and all in $H$ (where $v_{k+i}\coleq v_i$ for $i\in[d]$). 
        We say that the above edges $e_{\iota, \sigma}$ are {\em covered} by the extra-tight trail (respectively, tour).
        \end{definition}
	
	Observe that a sequence $T=(v_1,\dots,v_k) \in [n]^k$ forms an extra-tight trail in $K_n^{(d)}$ if and only if $\cS_T$ is a straight $(d+1)$-graph. 
	Analogously, a (cyclic) sequence $T=(v_1,\dots,v_k)$ forms an extra-tight tour in $K_n^{(d)}$ if and only if the hypergraph $\cC_T := \{\{ v_i, v_{i+1}, \ldots , v_{i+d}\} : i=1, \ldots, k\}$ (where $v_{k+j}\coleq v_j$ for $j\in[d]$) is a $(d+1)$-graph whose dual graph is a cycle.   

    \begin{remark}
        Observe that for the sequence $T$ of an extra-tight trail (respectively, tour) in $K_n^{(d)}$, the edges covered by $T$ are exactly the ones contained in the $d$-graph $\partial\cS_T$ (respectively, $\partial\cC_T$).
        For ease of notation, whenever it does not cause confusion, we will also use the letter $T$ to denote the $d$-graph $\partial\cS_T$ (respectively, $\partial\cC_T$) itself.
    \end{remark}

    \begin{definition}
    An extra-tight trail (respectively, tour) $T$ in a $d$-graph $H$ which covers all edges of $H$
    is called an \emph{extra-tight Euler trail} (respectively, \emph{extra-tight Euler tour}) of $H$.
    \end{definition}

	If $T$ is an extra-tight Euler tour of $K_n^{(d)}$ then $\partial \cC_T = {\binom{[n]} d}$, so
	by \Cref{obs: size of F and its shadow}\ref{enum: cycle vs shadow} deleting any member of $\cC_T$ provides a construction establishing \Cref{theo: precise value of H(n d)}. 
    
	The name of the concept in the previous definition is inspired by the generalisation of usual Euler tours in graphs to $d$-graphs. A (cyclic) sequence 
    $T$ of vertices in a $d$-graph $\cG$ is called an {\em Euler tour} if 
    every $d$ cyclically consecutive entries of the sequence forms an edge of $\cG$ and every edge of $\cG$ appears exactly once as such. An Euler tour of the complete $d$-graph $\cK_n^{(d)}$ means a (cyclic) sequence of ${\binom n d}$ vertices from $[n]$, such that each $d$-subset of $[n]$ appears uniquely as $d$ cyclically consecutive entries of the sequence. The existence of this beautiful object of coding theory was famously conjectured by Chung, Diaconis, and Graham~\cite{CHUNG199243} (they called them universal cycles), whenever $d$ divides ${\binom{n-1}{d-1}}$. 
	Glock, Joos, K\"uhn, and Osthus~\cite{glock2020euler} proved this in much greater generality: for $d$-graphs, that besides satisfying the immediate divisibility conditions on vertex degrees, might have a  
	minimum $(d-1)$-degree which is a linear fraction less than that of the complete graph. 
	
	For the existence of an extra-tight Euler tour in $K_n^{(d)}$ certain degree conditions are also necessary. Namely, if $T$ is an extra-tight Euler tour of a $d$-graph $G$, then every occurrence of an element $a\in [n]$ as an entry of $T$ is contained in $d^2$ edges of $T$ (cf.\@ \Cref{obs: degree in trails}). Since the definition of an extra-tight Euler tour requires these $d$-subsets to be distinct for different occurrences of  $a$ in $T$ and together they should include all $d$-subsets  of $G$ containing $a$, an extra-tight Euler tour can only exist if the degree of $a$ in $G$ is divisible by $d^2$. In particular, if an extra-tight Euler tour exists in the clique $K_n^{(d)}$, then $d^2$ divides~${\binom{n-1}{d-1}}$.
	
	Our next theorem states that these divisibility conditions are also sufficient for large enough $n$, even if the minimum $(d-1)$-degree is a small linear fraction smaller than that of the clique.  
	
	\begin{theorem}\label{cor: extra-tight tour covering everything}
		For every integer $d\ge 2$, there are an $\alpha>0$ and a positive integer $n_0$ such that the following holds for every integer $n\ge n_0$. If $G$ is an $n$-vertex $d$-graph with $\delta(G)\ge(1-\alpha)n$ for which all vertex degrees are divisible by $d^2$, then it has an extra-tight Euler tour.
		In particular, $K_n^{(d)}$ has an extra-tight Euler tour if and only if $d^2$ divides $\binom{n-1}{d-1}$. 
	\end{theorem}
	
	As we argued above, this theorem together with \Cref{obs: size of F and its shadow} implies \Cref{theo: precise value of H(n d)} whenever $d^2$ divides $\binom{n-1}{d-1}$.  
	This means a positive fraction of the integers $n$ for any $d$, including for instance every $n\equiv 1 \pmod{d^2}$.
	The correct asymptotics is also implied for every $n$, with an error term of $O(n^{-1})$, which is better than that of any of the previous asymptotic results \cite{bohman2022complexes,dkebski2017harmonious, gould2025advancing}. Yet, from the point of view of the precise value, this residue class looks like the ``easiest'' of the cases, as the optimal  $(d+1)$-graph must miss covering $d-1$ of the  $d$-subsets. According to \Cref{obs: size of F and its shadow}, this is the most  ``wiggle room'' we can possibly have.
	And indeed, the resolution of the $2$-dimensional case in \cite{2-dimensional} should be a warning sign. 
	There nice, explicit cyclic constructions were found whenever $n-1$ is divisible by $2^2=4$. For the other residue classes, however, more and more intricate ad hoc adjustments had to be introduced to achieve the optimal diameter. So even if we had access to relatively simple explicit extra-tight Euler tours for some residue class (which we do not have, not even for $d=3$), the complications to extend this to all $n$ would quite likely be horrendous/infeasible.  
	
	Indeed our approach for the proof of \Cref{theo: precise value of H(n d)} is very different. It avoids the separate treatment of different residue classes, and rather relies on the existence of very long extra-tight Euler trails in hypergraphs that are allowed to be considerably sparser than complete. Our proof of \Cref{cor: extra-tight tour covering everything} will also go through this theorem.
	For extra-tight trails, the divisibility constraints on the degrees are a bit more complicated for the vertices at the ends, but again we can show that these, together with a high enough minimum $(d-1)$-degree are sufficient.

	\begin{theorem}\label{theo: large min degree implies extra-tight trail}
		For every integer $d\ge 2$, there are an $\alpha>0$ and a positive integer $n_0$ such that the following holds for every integer $n\ge n_0$. Let $G$ be an $n$-vertex $d$-graph with $\delta(G)\ge(1-\alpha)n$ and $\{v_1,\dots,v_d\},\{v_1',\dots,v_d'\}$ disjoint edges in $G$ such that
		\begin{equation}\label{eq: d-trail divisible}
			\deg_G(v)\equiv\begin{cases}
				i(d-1)+1\pmod{d^2}&\exists \, i\in[d]:v\in\lbrace v_i,v'_i\rbrace\\
				0\pmod{d^2}&\text{else.}
			\end{cases}
		\end{equation}
		Then there exists an extra-tight Euler trail with the ends $(v_1,\dots,v_d)$ and $(v'_d,\dots,v'_1)$.
	\end{theorem}
	
	The straight simplicial complex of Bohman and Newman~\cite{bohman2022complexes} is constructed through a random sequence of vertices forming an extra-tight trail and thus provides an asymptotically optimal solution for Theorem~\ref{theo: large min degree implies extra-tight trail} in $K_n^{(d)}$.
	Their random process however gets stuck at some point, in the sense that the addition of any vertex as the next entry of the sequence would violate the extra-tight trail conditions. 
	
	To avoid getting stuck, it is natural to employ the method of absorption. 
	We start with constructing a special extra-tight trail $A$ called ``absorber'', which has a lot of built-in flexibility, so that it is able to incorporate into itself the leftover edges $L$ of {\em any} asymptotically optimal extra-tight trail our process ends up building. By this we mean that $A\cup L$ can be covered by an extra-tight trail whose ends are the same as those of $A$. 
	We make sure that the asymptotically optimal extra-tight trail our process builds contains $A$ as a subsequence and hence when the process gets stuck, the leftover can be absorbed into it, i.e. the subsequence of $A$ is replaced by that of the extra-tight Euler trail of $A\cup L$.  
	
	For our proof we adapt some of the heavy machinery developed by the first author together with Kühn, Lo, and Osthus~\cite{glock2023existence} in their work on decompositions of hypergraphs and combine it with several novel ideas. One of these contributions includes intricate switcher mechanisms, built into the absorber trail, the role of which is to correct the ``mistakes'' which happen during the absorption process. Another novelty is a significant strengthening of the existing asymptotically optimal constructions \cite{bohman2022complexes,dkebski2017harmonious,gould2025advancing}, which also ensures that the maximum degree of the leftover hypergraph can be arbitrarily small. 
	
	\subsection{Proof overview and organization of the paper.}\label{sec: Proof overview}
	
	\def\withalltheos{1}
	\begin{figure}[htpb]
		\centering
		\resizebox{\linewidth}{!}{
			\begin{tikzpicture}[node distance=4cm and 2.5cm]
				\node[theorem] (Approx) {Approximate Decomposition \Cref{lem: approximate decomposition}};
				\node[theorem, right=of Approx] (Supercomplex) {Supercomplex Lemma};
				
				\node[theorem, below=of Supercomplex] (Cover Down) {Cover Down \Cref{lem: Cover Down}};
				\node[theorem, below=of Cover Down] (extra-tight) {\Cref{theo: large min degree implies extra-tight trail}};
				\node[theorem, right=of Cover Down] (Absorber) {Absorber \Cref{lem: Absorber Lemma}};
				\node[theorem, left= of Cover Down] (Vortex) {Vortex \Cref{lem:There is a vortex}};
				\node[theorem, below=of extra-tight](diameter) {\Cref{theo: precise value of H(n d)}};
				\node[theorem, left=of extra-tight] (Turn) {\Cref{obs: C vs C'}};
				\ifthenelse{\equal{\withalltheos}{1}}{\node[theorem, right=of diameter] (closed) {\Cref{cor: extra-tight tour covering everything}};
					\draw[arrow] (extra-tight) --node[edgelabel]{quickly implies} (closed);
				}{}
				
				\draw[arrow] (Approx) -- node[edgelabel] {covers most of the edges} (Cover Down);
				\draw[arrow] (Supercomplex) -- node[edgelabel] {sets aside edges before to absorb leftover} (Cover Down);
				\draw[arrow] (Vortex) -- node[edgelabel] {gives vortex} (extra-tight);
				\draw[arrow] (Cover Down) -- node[edgelabel] {iteratively covers everything of a level in the vortex} (extra-tight);
				\draw[arrow] (Absorber) --node[edgelabel]{sets aside edges before to absorb leftover of last level} (extra-tight);
				\draw[arrow] (extra-tight) --node[edgelabel]{implies once degrees are fixed} (diameter);
				\draw[arrow] (Turn) --node[edgelabel]{fixes degrees using turns} (diameter);
				
		\end{tikzpicture}}
		\caption{Overall proof structure}
		\label{fig: Float chart}
	\end{figure}
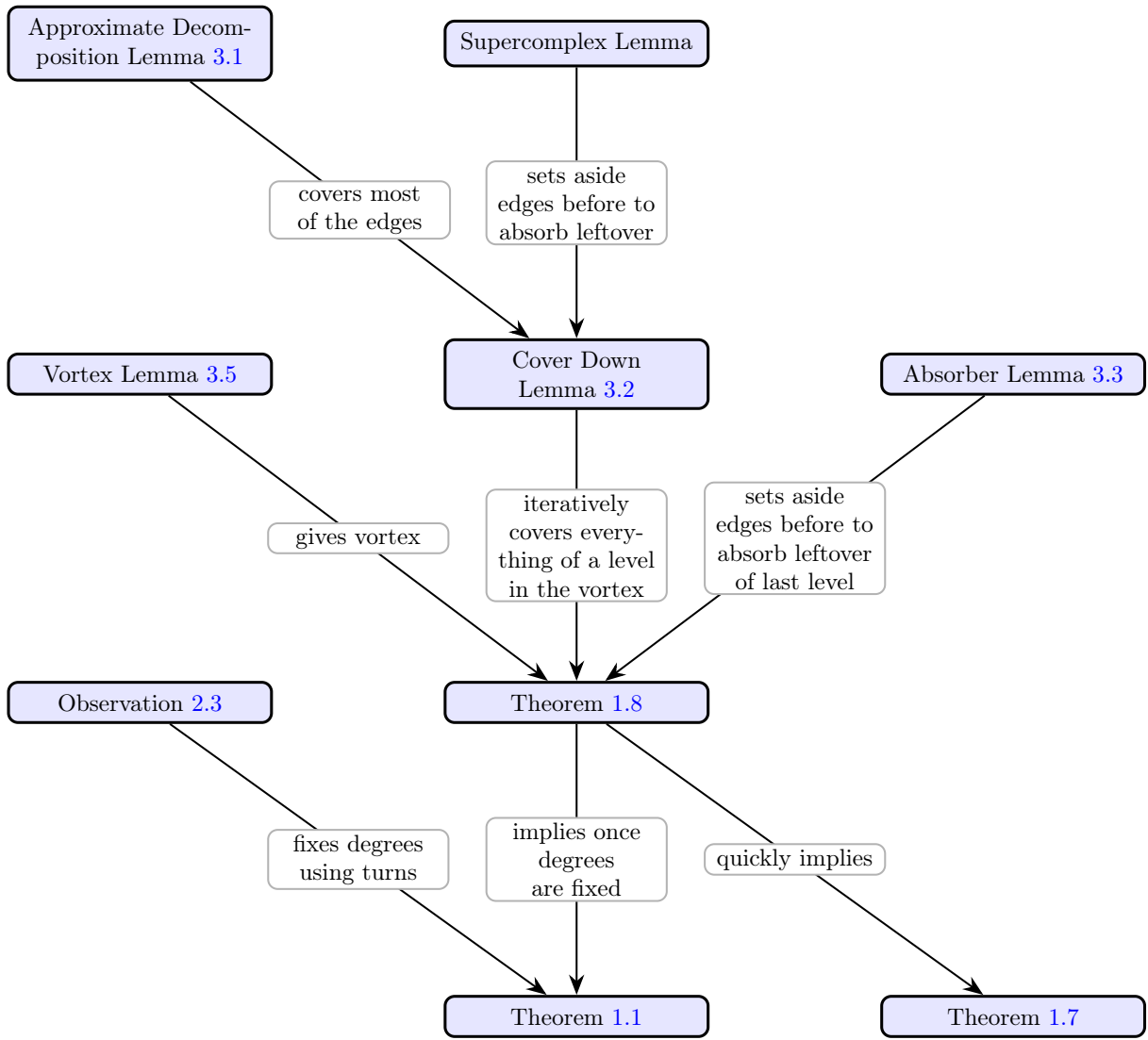
	
	We now outline the overall proof structure which is also illustrated in Figure~\ref{fig: Float chart}. 
	In Section~\ref{sec:proof H(n d)}, we derive  
	Theorem~\ref{cor: extra-tight tour covering everything} 
	and \Cref{theo: precise value of H(n d)}
	from \Cref{theo: large min degree implies extra-tight trail}. 
	The former is quite straightforward; for the latter, one needs to resolve the issue that the divisibility condition on the vertex degrees might not be satisfied. 
	We overcome this obstacle by using the fact that our simplicial $d$-complex is not required to be straight, but can make ``turns'' (cf.~\Cref{fig:straight vs. non-straight}). 
	This will allow us to build an initial segment of our simplicial complex such that removing the $d$-sets they cover (together with up to $d-1$ additional $d$-sets), we are left with a $d$-graph satisfying the divisibility and minimum degree conditions of \Cref{theo: large min degree implies extra-tight trail}. The application of \Cref{theo: large min degree implies extra-tight trail} then gives us a straight simplicial complex, which can be concatenated with the initial segment, resulting in a simplicial complex covering all but at most $d-1$ of the $d$-sets of $[n]$, which is optimal by Observation~\ref{obs: size of F and its shadow}. 
	This reduction is an essential step in our translation of the problem on diameter of simplicial complexes to design theory.  
	
	The proof of \Cref{theo: large min degree implies extra-tight trail} is inspired by the approach of Glock, Joos, Kühn, and Osthus~\cite{glock2020euler} for proving the conjecture of Chung, Diaconis, and Graham~\cite{CHUNG199243} that $K_n^{(d)}$ has a tight $d$-uniform Euler tour when $\binom{n-1}{d-1}$ is divisible by~$d$ and $n$ is large enough. Very roughly speaking, their approach works as follows. First, they construct a relatively short closed tight walk in which every ordered $(d-1)$-set of vertices appears at least once consecutively. This walk serves as a ``backbone'' of the final tour. 
	Afterwards, they remove the edges of this tour and decompose the remaining edges into tight cycles. This second step is achieved by applying a deep result from design theory~\cite{glock2023existence}, a generalization of Keevash's existence of designs \cite{Keevash} to decompositions into arbitrary hypergraphs~$F$. The simple but crucial point is then that these cycles can be merged into the initial tour to form one tour.
	
	The analogue of this argument does not work in the extra-tight setting. While one could still find a ``backbone'' extra-tight trail and could also apply the $F$-decomposition theorem to decompose into extra-tight cycles, the problem is that those extra-tight cycles cannot be simply merged into the backbone trail. Indeed, due to the edges that skip a vertex in the sequence, some edges that were covered before by either the backbone trail or the extra-tight cycle will not be used anymore, while new edges are needed that might not be available for this. 
	
	One of the main novel ingredients in our approach is the construction of \emph{switchers}. These are flexible structures, akin to absorbers, that allow us to reverse the side effects of the merging operation. The construction of these switchers is intricate, and again showcases the increase in difficulty when transitioning from the $2$-dimensional case to higher dimensions: For $d=2$, one can construct the necessary switchers explicitly, but for $d\ge 3$, this seems infeasible. 
    
	A new challenge that arises with this is that we cannot build these switchers for all possible locations because of space barriers. To overcome this, we will employ the iterative absorption method. This will enable us to ensure that the final merging step of inserting extra-tight cycles into a backbone trail is only needed on a set $U_\ell$ of constant size. For this, in a first step, a vortex is constructed, which is a sequence of nested vertex subsets, the final one being~$U_\ell$. In each step, the leftover is pushed into the next subset. This is achieved by the Cover Down Lemma, which in turn relies crucially on an Approximate Decomposition Lemma similar to the asymptotic solutions in \cite{bohman2022complexes,dkebski2017harmonious}. However, for the iterative procedure to not get out of control, we need an Approximate Decomposition Lemma with ``boosted'' error parameters. In spirit, this is similar to the Boost Lemma from~\cite{glock2023existence}, and the implementation is inspired by the recent work in~\cite{gishboliner2023tight}.
	
	The proof of Theorem~\ref{theo: large min degree implies extra-tight trail} stretches over several sections.
	In \Cref{sec:proof extra-tight trail}, we first state the three key lemmas (the Approximate Decomposition Lemma, the Cover Down Lemma, and the Absorber Lemma) and use them to derive \Cref{theo: large min degree implies extra-tight trail}. The proofs of these lemmas are then done in the subsequent sections. The Absorber Lemma is proved in \Cref{sec: Absorber}, the Approximate Decomposition Lemma in \Cref{sec: Approximate}, and the Cover Down Lemma in \Cref{sec: Cover Down}. The latter two use the concept of rooted embeddings from \cite{glock2023existence}. In \Cref{sec: Rooted Embeddings}, we quickly repeat their definition and state the lemmas that will be used.	
		
		\subsection{Notation and Probabilistic Tools}
		We will follow the same notation as \cite{glock2023existence}. For completeness, we state them again here. 
		
		For a positive integer, let $[n]\coleq\{1,\dots,n\}$ and $[n]_{0}\coleq \{0,1,\dots,n\}$. For a set $X$ and $i\in\N_0$, let $\binom{X}{i}$ be the set of all $i$-subsets of~$X$. 
		
		By the hierarchy $x\ll y$ we mean that for every $y\in(0,1)$, there is an $x_0\in (0,1)$ such that for all $0<x<x_0$, the subsequent statement holds. If a hierarchy contains several $\ll$, they are to be read from right to left. Furthermore, if a hierarchy contains $1/x$, $x$ is assumed to be a positive integer.
		
		A \emph{$d$-graph} $H$ is a $d$-uniform hypergraph, where we often identify $H$ with its edge set $E(H)$.
        If $G$ is a $d$-graph and $S\subs V(G)$ a set with $\abs{S}\in[d]_0$, the \emph{link graph} $G(S)$ is the $(d-\abs{S})$-graph that has vertex set $V(G)\backslash S$ and contains all subsets $S'$ such that $S\cup S'\in G$. We write $\deg(S)\coleq\abs{G(S)}$. For $i\in[d-1]_0$, the \emph{minimum} resp.\@ \emph{maximum $i$-degree} $\delta_i(G)$ resp.\@ $\Delta_i(G)$ is the minimum resp.\@ maximum value of $\deg(S)$ over all $S\in \binom{V(G)}{i}$. Moreover, $\delta(G)\coleq \delta_{d-1}(G)$ and $\Delta(G)\coleq \Delta_{d-1}(G)$. Finally, if $S\in \binom{V(G)}{d-1}$ and $U\subs V(G)$, $\deg_G(S,U)$ is the number of $u\in U$ such that $S\cup\{u\}\in G$.
		
		The hypergraphs of extra-tight trails and extra-tight tours with no repeated vertex will be especially important in our constructions. 
		
		\begin{definition} 
			The $d$-graphs $\cS_T^{(d)}$ and $\C_T^{(d)}$ created from a sequence $T=(v_1,\dots,v_k)$ of pairwise {\em distinct} vertices is called the ($d$-uniform) \emph{tight path}, denoted by $P_{k}^{(d)}$, and the ($d$-uniform) \emph{tight cycle}, denoted by $C_k^{(d)}$, respectively. The shadows $\partial P_{k}^{(d+1)}$ and $\partial C_k^{(d+1)}$ are called the ($d$-uniform) {\em extra-tight path} and {\em extra-tight cycle} and are denoted by $\extratightpath{k}{d}$ and $\extratightcycle{k}{d}$, respectively.
		\end{definition}

        We will need to pay particular attention to the ends of paths and extra-tight trails.

		\begin{definition}
			If $T$ is an extra-tight trail, then a set $S\subs V(H)$ \emph{appears at an end} of $T$ if $S\subs\{v_1,\dots,v_d\}$ or $S\subs\{v_{k-d+1},\dots,v_k\}$ and analogously for an extra-tight path if $T$ does not repeat vertices.
		\end{definition}
		
		We will use the following version of a Chernoff bound several times.
		
		\begin{theorem}[\cite{janson2011random}, Corollary 2.3]\label{theo: Chernoff}
			Let $X\coleq\sum_{i\in [n]} X_i$ where the $X_i$ are independently distributed in $[0,1]$ for $i\in[n]$. Then, for every $0<\beta\le 3/2$, we have
			\[\P\big[\abs{X-\E[X]}\ge\beta\E[X]\big]\le 2\exp\!\left(-\frac{\beta^2}{3}\E[X]\right).\]
		\end{theorem}
		
		A very similar result holds for hypergeometric distributions as well:
		
		\begin{theorem}[\cite{janson2011random}, Theorem 2.10, Equation (2.6)]\label{theo: hypergeometric}
			Let $X$ have a hypergeometric distribution\footnote{Recall that a random variable $X$ has hypergeometric distribution with parameters $N$, $n$, $m$ if $X\coleq\abs{S\cap [m]}$ where $S$ is a subset chosen uniformly at random among all subsets of $[N]$ of size $n$.} with parameters $N$, $n$, $m$. Then for all $0\le t$, we have
			\begin{align}
				\P\big[X\le \E[X]-t\big]&\le \exp\!\left(-\frac{t^2}{2\E[X]}\right).\label{eq: absolute hypergeometric chernoff}
			\end{align}
		\end{theorem}
		
		At one point, we will need the following variant which is called McDiarmid's Inequality and can also be derived from the Azuma-Hoeffding Inequality.
		
		\begin{theorem}[\cite{janson2011random}, Remark 2.28]\label{lem: McDiarmid}
			Let $X_1,\dots,X_n$ be independent random variables with $X_i$ taking values in $A_i$ for each $i\in [n]$. Furthermore, let $c_1,\dots,c_n$ be positive numbers and let $f\colon A_1\times \dots\times A_n\to \mathbb R$ be a function satisfying $\abs{f(x)-f(\bar x)}\le c_i$ whenever the vectors $x,\bar x\in A_1\times\dots\times A_n$ differ only in the $i$-th coordinate. Then, the random variable $Y\coleq f(X_1,\dots,X_n)$ satisfies
			\[\P\big[\abs{Y-\E[Y]}>t\big]\le 2\exp\!\left(-\frac{2t^2}{\sum_{i=1}^nc_i^2}\right)\]
			for all $t>0$.
		\end{theorem}
		
		\section{Proof of Theorem~\ref{theo: precise value of H(n d)} and \ref{cor: extra-tight tour covering everything}
		}\label{sec:proof H(n d)}
		
		In this section, we show how the main \Cref{theo: precise value of H(n d)} and \Cref{cor: extra-tight tour covering everything} can be derived from \Cref{theo: large min degree implies extra-tight trail}. The remainder of the paper will then concern the proof of \Cref{theo: large min degree implies extra-tight trail}.
		
		First, we need an observation that justifies the degree requirement \eqref{eq: d-trail divisible}.
		
		\begin{observation}\label{obs: degree in trails}
			If $C=(v_1,\dots,v_k)$ is an extra-tight tour in a $d$-graph with $k\ge d+3$, then $\deg_C(v)\equiv 0\pmod{d^2}$ for all $v\in V(C)$. If, moreover, $C$ is an extra-tight cycle, then $\deg_C(v)=d^2$. If $T=(v_1,\dots,v_k)$ is an extra-tight trail in a $d$-graph with $k\ge 2d$ where the $2d$ vertices at the ends of $T$ are pairwise distinct, then
			\begin{align*}\deg_T(v)&\equiv\begin{cases}
					i(d-1)+1&\exists i\in[d]:v\in\lbrace v_i,v_{k+1-i}\rbrace\\
					0&\text{else}
				\end{cases}\pmod{d^2}.
				\intertext{If, moreover, $T$ is an extra-tight path, then }
				\deg_T(v)&=\begin{cases}
					i(d-1)+1&\exists i\in[d]:v\in\lbrace v_i,v_{k+1-i}\rbrace\\
					d^2&\text{else.}
			\end{cases}\end{align*}
		\end{observation}
		\begin{proof}
			Given an edge $e$ in $C$, $k\ge d+3$ implies that at least 3 vertices are not in~$e$. By the definition of extra-tightness, each edge can skip at most one vertex. Therefore, all vertices covered by $e$ are consecutive vertices on $C$ except for at most one. Thus, there are unique $\iota\in[k]$, $\sigma\in[d]$ such that $e=e_{\iota,\sigma}$ (cf.~\Cref{def: extra-tight trail tour}). Hence, $k\ge d+3$ implies that the $e_{\iota,\sigma}$ are pairwise distinct.
			
			The label $v_\ell$ is exactly in the $e_{\iota,\sigma}$ with $\iota\in\{\ell-d,\ell-d+1,\dots,\ell\}$ and $\sigma\in[d]\backslash\{\ell-\iota\}$ (where $e_{\iota,\sigma}\coleq e_{\iota+k,\sigma}$ for $-d<\iota<0$). These are exactly $(d+1)(d-1)+1=d^2$ edges (the $+1$ comes from $\iota=\ell$ where $\sigma\in[d]$). Thus, $\deg_C(v)\equiv 0\pmod{d^2}$ for all $v\in V(C)$.
			
			Next, we consider the extra-tight trail~$T$. If $v_\ell$ is a label that is not at the end of $T$, then $v_\ell$ is in the same $d^2$ edges as above. Since the extra-tight trails are symmetric, we only have to consider $v_\ell$ with $\ell\in[d]$. There, the label $v_\ell$ is exactly in the $e_{\iota,\sigma}$ with $\iota\in[\ell]$ and $\sigma\in[d]\backslash\{\ell-\iota\}$ (since $k\ge 2d$, $[\ell]\subs [k-d]$ whence all these $e_{\iota,\sigma}$ are indeed defined). Therefore, $v_\ell$ is in $\ell(d-1)+1$ edges. Since the vertices at the end are pairwise distinct, the result follows.
		\end{proof}
	
	Frequently, we will have to connect several extra-tight trails into one long extra-tight trail. The following lemma allows us to do this.
	
	\begin{lemma}\label{lem: glue two extra-tight trails}
		Let $n$, $d$, $k$, $\ell$ be non-negative integers with $d\ge 2$. Let $A=(a_1,\dots,a_k)$ and $B=(b_1,\dots,b_\ell)$ be the (possibly empty) sequences of two extra-tight trails in an $n$-vertex graph $G$ such that $A\cap B=\emptyset$. Furthermore, assume that $U\subs V(G)$ such that \[\abs U> d^2\big(n-\delta(G)+\Delta(A)+\Delta(B)+2d^3\big)+4d.\] 
		Then there are $2d$ distinct vertices $v_1,\dots,v_{2d}\in U$ such that
		\[(a_1,\dots,a_k,v_1,\dots,v_{2d},b_1,\dots,b_\ell)\]
		is the sequence of an extra-tight trail in $G$.
	\end{lemma}
    \begin{proof}
    We select the vertices $v_1,\dots,v_{2d}$ one after the other.
    For $i=1,\dots,2d$, let $E_i$ be the link of $v_i$ in the subgraph of the extra-tight trail $(a_1,\dots,a_k,v_1,\dots,v_{2d},b_1,\dots,b_\ell)$ induced by the vertex set $\{a_1,\dots,a_k, v_1,\dots,v_{i}, b_1,\dots,b_\ell\}$.
    Note that, by \Cref{obs: degree in trails}, $|E_i| \le d^2$.
    
    Now suppose, for $i\in [2d-1]$, we have found vertices $v_1,\dots,v_{i-1}\in U\backslash\{a_{k-d+1},\dots,a_k,b_1,\dots,b_d\}$ such that $A_i:=(a_1,\dots,a_k,v_1,\dots,v_{i-1})$ is an extra-tight trail that is disjoint from $B$.
    This clearly holds for $i=0$, where $A_0=A$.
    We pick $v_{i} \in U\backslash\{a_{k-d+1},\dots,a_k,v_1,\dots,v_{i-1},b_1,\dots,b_d\}$ such that $\{v_i\} \cup e$ is an edge of $G_i:=G \setminus (A_i \cup B)$ for each $e \in E_i$.
    With $\delta(G_i)\ge \delta(G)-\Delta(A_i)-\Delta(B)\ge \delta(G)-\Delta(A)-\Delta(B)-2d \cdot d^2$ this is possible, because
    \begin{equation*}
		\begin{adjustbox}{max width=\linewidth}
		$\begin{aligned}
			\abs{\left(\bigcap_{S\in E_i}G_i(S)\right)\backslash\{a_{k-d+1},\dots, a_k,v_1,\dots,v_{i-1},b_1,\dots,b_d\}}&\ge \abs{\bigcap_{S\in E_i}G_i(S)}-4d\\
			&\ge n-\abs{E_i}\cdot \big(n-\delta(G')\big)-4d\\
			&\ge n-d^2\cdot \big(n-\delta(G)+\Delta(A)+\Delta(B)+2d^3\big)-4d\\
			&>n-\abs{U}.
		\end{aligned}$
        \end{adjustbox}
        \end{equation*}
    For $i \le 2d-2$, as all new edges are contained in $\{a_{k-d+1},\dots, a_k,v_1,\dots,v_{i-1}\}$ and these vertices are all different, it follows that $A_{i+1}:=(a_1,\dots,a_k,v_1,\dots,v_i)$ is an extra-tight trail that is disjoint from $B$ and we can continue to the next $i$.
    
    For $i=2d-1$, besides the edges of the extra-tight trail $(a_1,\dots,a_k,v_1,\dots,v_{2d-1})$, we now also identified in $G_{2d-1}$ the edges $\{v_{d-i},\dots,v_{2d-1},b_1,\dots,b_{d-i}\}$ for $i=1,\dots,d-1$, which are pairwise different, because they are contained in $\{v_{d+1},\dots,v_{2d-1},b_1,\dots,b_{d}\}$ and these are all different.
    It only remains to pick $v_{2d} \in U\backslash\{a_{k-d+1},\dots,a_k,v_1,\dots,v_{2d-1},b_1,\dots,b_d\}$ such that $\{v_{2d}\} \cup e$ is an edge of $G_{2d}:=G\setminus(A_{2d} \cup B)$ for each $e \in E_{2d}$.
    This is possible by the same calculation.
    Moreover, all the edges  $\{v_{2d}\} \cup e$ for $e \in E_{2d}$ are pairwise different because they are contained in $\{v_{d+1},\dots,v_{2d},b_1,\dots,b_d\}$  and these are all different.
    Hence we indeed get an extra-tight trail $(a_1,\dots,a_k,v_1,\dots,v_{2d},b_1,\dots,b_\ell)$.
    \end{proof}
	
    We start by proving \Cref{cor: extra-tight tour covering everything} before proving \Cref{theo: precise value of H(n d)} at the end of this section.
	\begin{proof}[Proof of \Cref{cor: extra-tight tour covering everything}]
	Let $G$ be an $n$-vertex $d$-graph with $\delta(G)\ge (1-\alpha)n$ where every vertex degree is divisible by $d^2$. First, we want to find an extra-tight path $P=(v_1,\dots,v_{2d})$ in~$G$. We can do this by applying \Cref{lem: glue two extra-tight trails} with $A_{\ref{lem: glue two extra-tight trails}}\coleq B_{\ref{lem: glue two extra-tight trails}}\coleq ()$ and $U_{\ref{lem: glue two extra-tight trails}}=V(G)$.
			
			Now consider the graph $G'\coleq \big(G\backslash P\big)\cup\big\{\{v_1,\dots,v_d\},\{v_{d+1},\dots,v_{2d}\}\big\}$. Since we assumed that $\deg_G(v)\equiv 0\pmod {d^2}$ holds for all $v\in V(G)$ and by \Cref{obs: degree in trails}, we get, for $i\in[d]$
			\[\deg_{G'}(v_i)\equiv -\big(i(d-1)+1\big)+1=-i(d-1)\equiv(d+1-i)(d-1)+1\pmod{d^2}\]
			and similarly, for $i\in\{d+1,\dots,2d\}$. We can conclude that $G'$ satisfies \eqref{eq: d-trail divisible} with the ends $(v_d,v_{d-1},\dots,v_1)$ and $(v_{2d},\dots,v_{d+1})$. Furthermore, $\delta(G')\ge (1-2\alpha)n$ if $n$ is large enough. Thus, we can apply \Cref{theo: large min degree implies extra-tight trail} with $\alpha_{\ref{theo: large min degree implies extra-tight trail}}\coleq 2\alpha$ and get an extra-tight trail $T=(v_d,\dots,v_1,\dots,v_{2d},\dots,v_{d+1})$ with ends $(v_d,\dots,v_1)$ and $(v_{2d},\dots,v_{d+1})$ covering all edges of~$G'$. But together with the edges of $P$, the vertex sequence of $T$ forms an extra-tight tour that covers all edges of~$G$.
		\end{proof}
		
		Next, we will prove \Cref{theo: precise value of H(n d)}. For this, we will find a simplicial $d$-complex $\C$ on $[n]$ whose dual graph $G(\C)$ is a path and where all but at most $d-1$ of the sets in $\binom{[n]}{d}$ appear in~$\C$. 
		
		The simplest way to do this would be to define a straight simplicial $d$-complex with the required properties. However, this is not always possible: by \Cref{obs: degree in trails}, in a straight simplicial $d$-complex $\C$, the number of $d$-sets a fixed element $i\in[n]$ appears in is divisible by $d^2$ unless $i$ is among the $2d$ elements that occur at one of the ends of~$\C$. Therefore, if the vertex degree $\binom{n-1}{d-1}$ in $K_n^{(d)}$ is not divisible by $d^2$, then it is impossible that a straight simplicial complex has the maximum possible diameter. 
		
		Fortunately, we can adjust this degree issue by first defining a relatively short simplicial $d$-complex $\C_1$ that uses ``turns''. Afterwards, the graph of the $d$-sets that are not elements in $\C_1$ (minus at most $d-1$ edges) will satisfy \eqref{eq: d-trail divisible} such that we can cover it with a straight simplicial $d$-complex $\C_2$. We will make sure that $\C_1$ and $\C_2$ have the same ends so that they can be combined into one long simplicial $d$-complex of maximum diameter. 
		
		To get $\C_1$, we will start with a straight simplicial $d$-complex and insert turns one after the other. The following observation specifies how we insert turns and how this influences the number of $d$-sets of the simplicial complex a fixed element is contained in. 

        \tdplotsetmaincoords{60}{345}  
		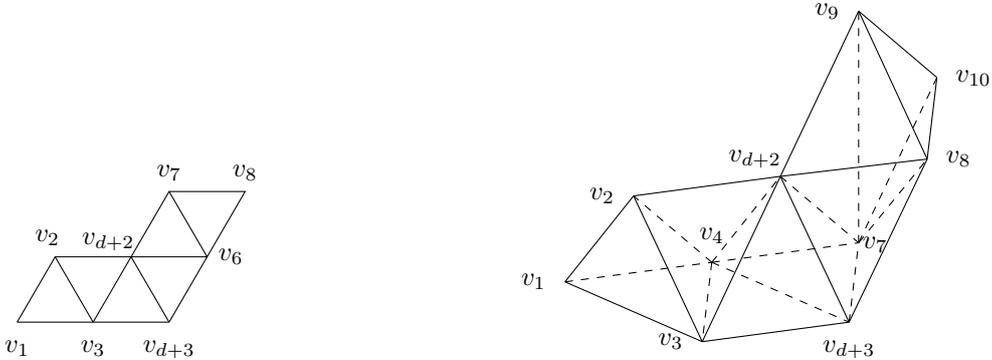
\begin{figure}[H]
			\centering
			\begin{tikzpicture}
				\foreach \x in {0,1} {
					\begin{scope}[shift={(\x,0)}]
						\draw (0,0) -- (0:1);
						\draw[] (0:1) -- (60:1) -- (0,0);
					\end{scope}
				}
				\draw (0.5,0.86603) -- (2.5,0.86603) (2,0) -- (3,1.73206) -- (2,1.73206) -- (2.5,0.86603) (2,1.73206) -- (1.5,0.86603);
				\textbe{0,0}{$v_1$}
				\textab{0.4,0.76603}{$v_2$}
				\textbe{1,0}{$v_3$}
				\textat{1.2,1.06603}{$v_{d+2}$}
				\textbe{2,0}{$v_{d+3}$}
				\textri{2.4,0.86603}{$v_{6}$}
				\textab{2,1.63206}{$v_7$}
				\textab{3,1.63206}{$v_8$}
			\end{tikzpicture}
			\hspace{3cm}
			\begin{tikzpicture}[tdplot_main_coords, scale=2]
				\coordinate (B) at (0.333,1,0);
				\coordinate (C) at (0.666,0.5,0.887);
				\coordinate (D) at (1,0,0);
				\coordinate (E) at (1.333,1,0);
				\coordinate (F) at (1.666,0.5,0.887);
				\coordinate (G) at (2,0,0);
				\coordinate (H) at (2.333,1,0);
				\coordinate (I) at (2.666,0.5,0.87);
				\coordinate (J) at (2.333,1,1.774);
				\coordinate (K) at (3,1.5,0.887);
				\foreach \x/\y in {D/B,B/C,C/D,D/F,D/G,G/I,I/F,F/C,F/J,G/F,J/I,K/I,K/J%
				} {
					\draw (\x) -- (\y);
				}
				\foreach \x/\y in {B/E,E/C,E/D,E/G,E/F,E/H,H/I,H/J,G/H,F/H,H/K%
				} {
					\draw[dashed] (\x) -- (\y);
				}
				\textle{B}{$v_1$}
				\textle{C}{$v_2$}
				\textle{D}{$v_3$}
				\textab{E}{$v_4$}
				\textle{1.8,0.5,1}{$v_{d+2}$}
				\textbe{G}{$v_{d+3}$}
				\textri{2.233,1,0}{$v_7$}
				\textri{I}{$v_8$}
				\textle{J}{$v_9$}
				\textri{K}{$v_{10}$}
			\end{tikzpicture}
			\caption{The simplicial $d$-complex $\C'$ for $d=2$ and $d=3$.}
			\label{fig: turns for d=2 and d=3}
		\end{figure}
		
		\begin{observation}\label{obs: C vs C'}
			Let $\C$ be a simplicial $d$-complex on $[n]$ with facet set $\mathcal F$ whose dual graph $G(\C)$ is a path~$P$. Let $d+4$ consecutive vertices of $P$ be the facets of the straight simplicial $d$-complex $(v_1,\dots,v_{2d+4})$. Furthermore, assume that $\C$ does not contain any of the $d$-sets $\{v_{d+2},v_{d+4},v_{d+5},\dots,v_{2d+3}\}\backslash\{v_i\}$ with $i\in \{d+4,\dots, 2d+2\}$. Let $\C'$ be the simplicial $d$-complex generated by $\mathcal F'\coleq \mathcal F\backslash\big\{\{v_{d+3},\dots,v_{2d+3}\}\big\}\cup\big\{\{v_{d+2},v_{d+4},\dots,v_{2d+3}\}\big\}$. Then $G(\C')$ is also a path and 
			\[\abs{(\partial\F)(v)}-\abs{(\partial \F')(v)}=\begin{cases}-(d-1)&v=v_{d+2}\\
				d-1&v=v_{d+3}\\
				0&\text{else.}
			\end{cases}\]
		\end{observation}
		
		\Cref{fig: turns for d=2 and d=3} shows $\C'$ in the case $d=2$ and $d=3$ where $\C$ is just the straight simplicial $d$-complex $(v_1,\dots,v_{2d+4})$. Notice that, in the left figure, the vertex $v_{d+2}$ is contained in $5=4+(d-1)$ edges, whereas $v_{d+3}$ is contained in $3=4-(d-1)$. Similarly, in the right figure, $v_{d+2}$ is contained in $11=9+(d-1)$ triangles and $v_{d+3}$ is contained in $7=9-(d-1)$ triangles.
		
		\begin{proof}[Proof of \Cref{obs: C vs C'}]
			Since $\C$ does not contain any of the $d$-sets $\{v_{d+2},v_{d+4},v_{d+5},\dots,v_{2d+3}\}\backslash\{v_i\}$ with $i\in\{d+4,\dots,2d+2\}$, the only facets in $\C'$ that have an intersection of size $d-1$ with the facet $\{v_{d+2},v_{d+4},v_{d+5},\dots,v_{2d+3}\}$ are $\{v_{d+2},\dots,v_{2d+2}\}$ and $\{v_{d+4},\dots,v_{2d+4}\}$, i.e.\@ the same facets that had an intersection of size $(d-1)$ with $\{ v_{d+3},\dots,v_{2d+3}\}$. Therefore, $G(\C')$ is still a path.
			
			The $d$-sets that $\C$ contains but not $\C'$ are $\{v_{d+3},v_{d+4},v_{d+5},\dots,v_{2d+3}\}\backslash\{v_i\}$ with $i\in\{d+4,\dots,2d+2\}$. The $d$-sets that $\C'$ contains but not $\C$ are $\{v_{d+2},v_{d+4},v_{d+5},\dots,v_{2d+3}\}\backslash\{v_i\}$ with $i\in\{d+4,\dots,2d+2\}$.
		\end{proof}
		
		\begin{proof}[Proof of \Cref{theo: precise value of H(n d)}] Let $G$ be the complete $d$-graph on the vertex set $[n]$. We will successively build a simplicial $d$-complex $\C$ on $[n]$. If a $d$-set $\{i_1,\dots,i_d\}$ appears in $\C$, we will delete it from the edge set of~$G$. That way, $G$ will always encode the set of $d$-sets that still need to appear in~$\C$.
			
			Our final simplicial $d$-complex will mostly be a straight simplicial $d$-complex but with a few turns throughout. It will start with a straight simplicial $d$-complex starting with $(1,\dots,d)$ and end with a straight simplicial $d$-complex ending with $(n-d+1,\dots,n)$. To be able to do that, we must make sure that the $d$-sets $\{1,\dots,d\}$ and $\{n-d+1,\dots,n\}$ are not used anywhere else.
			
			Note that the $d$-uniform Handshake Lemma implies that the sum of vertex degrees $n\binom{n-1}{d-1}$ in $G$ is divisible by~$d$. Let $s\in [d-1]_0$ be such that $n\binom{n-1}{d-1}\equiv d(s+1)\pmod{d^2}$. Let $M$ be a set of $s$ disjoint edges of $G$ that do not contain any of $1,\dots,d,n-d+1,\dots,n$. Delete $M$ from~$G$. Now, the sum of 1-degrees in $G$ is $\equiv d\pmod{d^2}$. The edges in $M$ will be the only $d$-sets not covered by our final simplicial $d$-complex. This together with \Cref{obs: size of F and its shadow} \ref{enum: path vs shadow} then the simplicial $d$-complex has the diameter $\floor{\frac{1}{d}\binom{n}{d}-\frac{d+1}{d}}$ as desired.
			
			Next, we will adjust the vertex degrees in $G$ using turns. By \Cref{obs: C vs C'}, each turn ``moves'' $d-1$ from one vertex degree to another vertex degree. Our goal now is to find a set of pairs of vertices that encode how one could fix the degree with these turns.
			
			\textbf{Claim.} There is a digraph $D$ on the vertex set $V(G)$ such that the following conditions hold:
			\begin{enumerate}
				\item\label{enum:S fixes degrees} for all $v\in V(G)$, if $\first_D(v)$ is the number of outgoing edges of $v$, and $\second_D(v)$ the number of incoming edges of $v$, then 
				\[\deg_G(v)+\big(\first_D(v)-\second_D(v)\big)(d-1)\equiv \begin{cases}
					i(d-1)+1&\exists i\in[d]:v\in\{i,{n+1-i}\}\\
					0&\text{else}
				\end{cases}\pmod{d^2};\]
				\item\label{enum:S no element too often} for all $v\in V(G)$, $\first_D(v)+\second_D(v)\le 5d^2$;
				\item\label{enum:S at most one at end} if $d=2$, then for all $(u,v)\in D$, $\{u,v\}\not\in \big(M\cup\big\{\{1,2\},\{n-1,n\}\big\}\big)$.
			\end{enumerate}
			\begin{claimproof}
				Start with $D=\emptyset$. Then \ref{enum:S no element too often} and \ref{enum:S at most one at end} are already satisfied. We will go from $v=1$ to $v=n$ and satisfy \ref{enum:S fixes degrees} for $v$ while maintaining the other conditions and \ref{enum:S fixes degrees} for all $u$ with $u<v$ adding at most $2d^2$ edges to $D$. Thus, throughout, $D$ will have at most $2d^2n$ edges.
				
				Suppose, we want to satisfy \ref{enum:S fixes degrees} for $v$ next where $v\in[n-1]$. Let $a\in [d^2-1]_0$ be the number such that 
				\[\deg_G(v)+(a+\first(v)-\second(v))(d-1)\equiv \begin{cases}
					i(d-1)+1&\exists i\in[d]:i\in\{i,n+1-i\}\\
					0&\text{else.}
				\end{cases}\pmod{d^2}\]
				Note that this number exists because $d-1$ and $d^2$ are coprime. We want to add $a$ more edges to $D$ where $v$ is the first element. For this, let $T$ be the set of vertices that already appear in $3d^2-1$ edges of $D$ or already appear in an edge with $v$ or ${v+1}$ in $D$. Since $|D|\le 2d^2n$, $T$ will contain at most $\frac{2}{2.9}n+2\cdot 5d^3$ vertices. Next, pick $a$ distinct vertices $u_1,\dots,u_a$ in $V(G)\backslash \big(T\cup \{v,{v+1},1,2,n-1,n\}\cup\bigcup_{e\in M}e\big)$ and add the $2a\le 2d^2$ edges $(v,u_j),(u_j,{v+1})$ with $j\in[a]$ to $D$. One can quickly see that \ref{enum:S at most one at end} and \ref{enum:S no element too often} are all still satisfied, and \ref{enum:S fixes degrees} is now satisfied for $1,\dots,v$.
				
				Continuing that way, we can satisfy \ref{enum:S fixes degrees} for all $1,\dots,{n-1}$. We claim that then \ref{enum:S fixes degrees} is also satisfied for~$n$. Indeed, using that $\sum_{v\in V(G)}\deg_G(v)\equiv d\pmod{d^2}$ and $\sum_{v\in V(G)}(\first(v)-\second(v))(d-1)=0$, we get
				\begin{align*}
					d&\equiv\sum_{v\in V(G)}\big(\deg_G(v)+(\first_D(v)-\second_D(v))(d-1)\big)\\
					&\overset{\ref{enum:S fixes degrees}}{=}\deg_G(n)+\big(\first_D(n)-\second_D(n)\big)(d-1)+2\sum_{i=1}^d\big(i(d-1)+1\big)-d\\
					&=\deg_G(n)+\big(\first_D(n)-\second_D(n)\big)(d-1)+d(d+1)(d-1)+2d-d\\
					&\equiv\deg_G(n)+\big(\first_D(n)-\second_D(n)\big)(d-1).
				\end{align*}
				This concludes the proof of the claim.
			\end{claimproof}
			
			Let $D$ be the digraph given by the claim and let $D=\{p_1,\dots,p_t\}$. We will now start to build the simplicial complex by defining a sequence 
			\[\Seq=\big(v^0_{d+5},v^0_{d+6}\dots,v^0_{2d+4},v^1_1,v^1_2,\dots,v^1_{2d+4},v^2_1,\dots,v^2_{2d+4},\dots,v^t_1,\dots,v^t_{2d+4}\big)\]
			such that the following conditions are satisfied:
			\begin{enumerate}[label=\arabic*)]
				\item\label{enum: Seq start} $v^0_{d+5}=1,v^0_{d+6}=2,\dots,v^0_{2d+4}=d$;
				\item\label{enum: Seq end}$\big\{v^t_{d+5},\dots,v^t_{2d+4}\big\}\cap\{n+1-d,\dots,n\}=\emptyset$;
				\item\label{enum: Seq turn vertices}  if $p_i=({j_1}, {j_2})$, then $v^i_{d+2}=j_1$ and $v^i_{d+3}=j_2$ for all $i\in[t]$;
				\item\label{enum: Seq edges distinct} all the $d$-sets of the straight simplicial $d$-complex defined via the sequence $\Seq$ and all the $d$-sets of the form $\{v^i_{d+2},v^i_{d+4},v^i_{d+5},\dots,v^i_{2d+3}\}\backslash \{v^i_{j}\}$ with $j\in \{d+4,\dots,2d+2\}$ are pairwise distinct and in $G\backslash\big\{\{n-d+1,\dots,n\}\big\}$;
				\item\label{enum: Seq no vertex too often} no element of $[n]$ appears in $\Seq$ more than $24d^3$ times.
			\end{enumerate}
			We start by defining $v^0_{d+5},v^0_{d+6} \dots,v^0_{2d+4},v^1_{d+2},v^1_{d+3},v^2_{d+2},v^2_{d+3}, \dots,v^t_{d+2},v^t_{d+3}$ such that \ref{enum: Seq start} and \ref{enum: Seq turn vertices} are satisfied. Note that this will not violate \ref{enum: Seq edges distinct} since we ensured that $\{1,2,\dots,d\}$ is still an edge in~$G$. For $d=2$, we have to be more careful here, but property \ref{enum:S at most one at end} of the claim ensures that every 2-set of \ref{enum: Seq edges distinct} that is already fully defined is unique and in $G\backslash\big\{\{n-1,n\}\big\}$. 
			
			The remaining elements of $\Seq$ are defined from left to right. Suppose we want to define the element $v^i_j$ next with $i\in[t]$ and $j\in[2d+4]$. Let $V'$ be the set of elements in $[n]$ that already appear $24d^3$ many times in $\Seq$. Since $\Seq$ consists of $d+t\cdot (2d+4)\le 12d^3n$ elements, $V'$ contains at most $\frac{12d^3n}{24d^3}=\frac{n}{2}$ elements. Let $E'$ be the set of $d$-sets mentioned in \ref{enum: Seq edges distinct}. Let $E''\subs E'$ be the subset of $d$-sets that are already completely defined. For each $e\in E'$ that contains $v^i_j$ where the other $d-1$ elements of $e$ are already defined, let $V_e$ be the set of elements $v\in [n]$ where if we defined $v^i_j$ to be $v$, $e$ would become an edge that is already in $E''$ or not in~$G$. There can be at most $(d^2+d)$ such $e$ and each $V_e$ has a size of at most $24d^3\cdot (d^2+d)\le 25d^5$. Define $v^i_j$ to be any element of $[n]\backslash \Big(V'\cup \bigcup_{v_j^i\in e\in E'}V_e\cup\{n+1-d,\dots,n\}\Big)$ that is not also already an element of $\Seq$ that appears among the $d$ elements before or after $v^i_j$. If $n$ is large enough with respect to $d$, we always have an option for $v^i_j$. Thus, we can define the sequence $\Seq$ such that it fulfills the required properties.
			
			Let $\C$ be the straight simplicial $d$-complex with vertex sequence $\Seq$. For each $i\in[t]$ we apply \Cref{obs: C vs C'} on $\C$ and the vertex sequence $v_1^i,\dots,v_{2d+4}^i$. By \ref{enum: Seq edges distinct}, the required $d$-sets are missing from~$\C$. We end up with a simplicial $d$-complex $\C_1$. By \ref{enum: Seq edges distinct} again, all $d$-sets of $\C_1$ are in~$G$. Let $G'$ be the $d$-graph that contains all edges of $G$ that are not $d$-sets in $\C_1$. Furthermore, we add the edge $\big\{v_{d+5}^t,\dots,v_{2d+4}^t\big\}$ to $G'$. By \Cref{obs: C vs C'}, \ref{enum: Seq end}, \ref{enum: Seq turn vertices}, and property \ref{enum:S fixes degrees} of the claim, we now have
			\[\deg_{G'}(v)\equiv\begin{cases}i(d-1)+1&\exists i\in[d]:v\in \big\{v^t_{d+4+i},n+1-d\big\}\\0&\text{else}
			\end{cases}\pmod{d^2}.\]
			Furthermore, $G''$ has minimum $(d-1)$-degree at least $n-24d^3\cdot (d^2+d-1)$ by \ref{enum: Seq no vertex too often}. Thus, we can apply \Cref{theo: large min degree implies extra-tight trail} and get an extra-tight trail with ends $\big(v^t_{d+5},v^t_{d+6},\dots,v^t_{2d+4}\big)$ and $(n-d+1,\dots,n)$. But this corresponds to a straight simplicial $d$-complex $\C_2$ which we can combine with $\C_1$ to get a simplicial $d$-complex whose dual graph is a path and which has every element of $\binom{[n]}{d}\backslash M$ as $d$-sets. Since these are all but at most $d-1$ many $d$-sets, \Cref{obs: size of F and its shadow} shows that the simplicial $d$-complex has diameter
			\[\floor{\frac{1}{d}\binom{n}{d}-\frac{d+1}{d}}.\qedhere\]
		\end{proof}
		
		\section{Proof of Theorem~\ref{theo: large min degree implies extra-tight trail} }\label{sec:proof extra-tight trail}
		
		In this section, we will first state all the key lemmas, i.e.\@ the Approximate Decomposition Lemma, the Cover Down Lemma, and the Absorber Lemma. Afterwards, we will prove \Cref{theo: large min degree implies extra-tight trail} assuming all the stated lemmas. 
		The proofs of the three main lemmas will follow in subsequent sections.
		
		\subsection{Key Lemmas}
	
		The result of Bohman and Newman~\cite{bohman2022complexes} already yields (in the complete $d$-graph) an extra-tight trail that covers all but a negligible proportion of the edges. Our first key lemma is a strengthening of this. Namely, we need to find such an extra-tight trail in $d$-graphs with minimum degree $\delta(G)\ge (1-\alpha)n$. While it may appear plausible that for very small $\alpha$, the analysis of Bohman and Newman still works, this would effect the size of the leftover. The crucial feature of our result is that we can ensure the leftover to be still arbitrarily small, that is, the quality of the approximate decomposition does not depend on~$\alpha$.
		
		\begin{lemma}[Approximate Decomposition Lemma]\label{lem: approximate decomposition}
			Suppose $1/n\ll \gamma\ll\alpha\ll1/d\le 1/2$. Let $G$ be an $n$-vertex $d$-graph with $\delta(G)\ge(1-\alpha)n$. Then there exists an extra-tight trail such that the leftover $L$ of all edges of $G$ not covered by the trail satisfies $\codegree(L)\le \gamma n$.
		\end{lemma}
		
		Here, the maximum degree condition on $L$ ensures that the number of uncovered edges is less than $\gamma n^d$, and that these edges are not concentrated at a few vertices.
		
		With this lemma, we end up with one extra-tight trail that covers most of the edges of $G$, but there we are stuck, unless we prepared well ahead some extra uncovered structure that would help us to continue. This is exactly the idea of the iterative absorption method. 
		For that, one builds randomly a ``vortex'', i.e. a sequence $V(G)=U_0\supseteq U_1\supseteq\dots\supseteq U_{\ell}$ of subsets with natural regularity properties, such that $\abs{U_i}/\abs{U_{i+1}}$ and $\abs{U_\ell}$ are constant.  
		
		In each iteration $i$ of the extra-tight trail-building process, we use the Approximate Decomposition Lemma and additional arguments to finds an extra-tight trail covering all uncovered edges in $G[U_i]\backslash G[U_{i+1}]$ while not using too many edges of $G[U_{i+1}]$.
		Here it is crucial that we can boost the parameters, meaning the maximum degree of the extra-tight trail in $G[U_{i+1}]$ is much smaller than the maximum degree of the complement of $G[U_{i+1}]$.
		This will ensure that at the end of this process, {\em all} edges in $G\backslash G[U_\ell]$ are covered, while almost all edges in $G[U_{\ell}]$ are uncovered.  
		
		The following Cover Down Lemma captures the iterative step. There we have to assume that certain degree conditions are fulfilled to be able to cover all desired edges by an extra-tight trail. Later, when we apply the Cover Down Lemma in the proof of \Cref{theo: large min degree implies extra-tight trail}, these degree conditions will be automatically satisfied by the fact that we started with a graph that satisfies \eqref{eq: d-trail divisible}. 
		
		\begin{lemma}[Cover Down Lemma]\label{lem: Cover Down}
			Suppose $1/n\ll\alpha,\mu\ll 1/d\le 1/2$ with $\alpha<\mu^{1.1}$. Let $G$ be a $d$-graph on $n$ vertices and $U\subs V(G)$ with $\abs U=\mu n$. Suppose that $\delta(G)\ge (1-\alpha)n$ and $\deg(v)\equiv0\pmod{d^2}$ holds for all $v\in V(G)\backslash U$. Then there exists an extra-tight trail $T$ whose ends are in $U$ and such that $G\backslash G[U] \subs T\subs G$ and $\Delta(T[U])\le \mu^2 n$.
		\end{lemma}
		
		To also incorporate the leftover in $U_\ell$ into our extra-tight trail, we want to use absorbers. In our case, this means we set aside an extra-tight trail before using the Cover Down Lemma that can absorb the leftover edges. However, we do not know beforehand how the leftover will look like. Hence, the absorber must work for all possible leftovers. There, $U_\ell$ being of constant size is vital because this allows us to build an absorber that works for all possibilities.
		
		\begin{lemma}[Absorber Lemma]\label{lem: Absorber Lemma}
			Let $1/n\ll \eta\ll 1/m\ll \alpha\ll 1/d\le 1/2$. Let $G$ be a $d$-graph on $n$ vertices with $\delta(G)\ge (1-\alpha)n$ and $U\subs V(G)$ a vertex set of size~$m$. Then $G$ contains an extra-tight trail $A$ 
			such that
			\begin{itemize}
				\item $V(A) \supseteq U$ and $U$ is an independent set in $A$;
				\item $\Delta(A)\le \eta n$;
				\item for each $d$-graph $L\subs G[U]$ with $\delta(L)\ge (1-\alpha)m$ where each 1-degree is divisible by $d^2$, there is an extra-tight trail with edge set $L\cup A$ having the same ends as $A$.
			\end{itemize}
		\end{lemma}
		With some extra work, one could omit the minimum degree requirement on $L$. However, since we get this for free in our main proof, we omit this here.
		
		\subsection{Proof of Theorem~\ref{theo: large min degree implies extra-tight trail}}
		In this subsection, we will prove our main \Cref{theo: large min degree implies extra-tight trail} assuming the Absorber Lemma and the Cover Down Lemma. While the Approximate Decomposition Lemma is not needed in this proof, we will need it when we prove the Cover Down Lemma.
		
		We start by formally defining the concept of a vortex and proving its existence in~$G$.
		
		\begin{definition}
			Let $G$ be a $d$-graph on $n$ vertices. An \emph{$(\alpha,\mu,m)$-vortex in $G$} is a sequence $U_0\supseteq U_1\supseteq\dots\supseteq U_\ell$ such that 
			\begin{enumerate}
				\item $U_0=V(G)$;
				\item $\abs{U_i}=\floor{\mu\abs{U_{i-1}}}$ for all $i\in[\ell]$;
				\item $\abs{U_{\ell}}=m$;
				\item\label{enum: vortex, minimum degree} $\deg_G(S,U_i)\ge(1-\alpha)\abs{U_i}$ for all $i\in[\ell]$ and $S\in \binom{U_{i-1}}{d-1}$.
			\end{enumerate}
		\end{definition}
		The following lemma and proof are similar to Lemma 3.7 in \cite{barber2020minimalist} where this is done for graphs instead of hypergraphs.
		\begin{lemma}\label{lem:There is a vortex}
			Let $1/m'\ll \alpha,\mu\ll 1/d\le 1/2$. Suppose $G$ is a $d$-graph on $n\ge m'$ vertices with $\delta(G)\ge (1-\alpha)n$. Then $G$ has an $(\alpha+\mu^{2},\mu,m)$-vortex for some $\floor{\mu m'}\le m\le m'$. 
		\end{lemma}
		\begin{proof}
			We define $n_0\coleq n$ and $n_i\coleq\floor{\mu n_{i-1}}$. Note $\mu^in\ge n_i\ge \mu^i n-\frac{1}{1-\mu.}$. Let $\ell\coleq 1+\max\lbrace i\ge 0:n_i\ge m'\rbrace$ and let $m\coleq n_\ell$. By definition, $\floor{\mu m'}\le m\le m'$. Finally, for $i\in[\ell]$, let 
			\begin{equation}\label{eq: alpha_i <= alpha/3}
				\mu_i\coleq n^{-1/3}\sum_{j=1}^i\mu^{-(j-1)/3}=n^{-1/3}\frac{\mu^{-i/3}-1}{\mu^{-1/3}-1}\le\frac{\big(\mu^{i-1}n\big)^{-1/3}}{1-\mu^{1/3}}\le\frac{m'^{-1/3}}{1-\mu^{1/3}}\le\frac{\mu^{2}}{3},
			\end{equation}
			and let $\mu_0\coleq 0$. 
			
			Suppose that for some $i\in[\ell]$, we have already found an $(\alpha +3\mu_{i-1},\mu,n_{i-1})$-vortex $U_0,\dots,U_{i-1}$ in~$G$. This is true for $i=1$. In particular, $\delta(U_{i-1})\ge (1-\alpha-3\mu_{i-1}) n_{i-1}$. Let $U_i$ be a subset chosen uniformly at random among all subsets of $U_{i-1}$ of size $n_{i}$. For each $S\in \binom{U_{i-1}}{d-1}$, we have $\E[\deg_{G}(S,U_i)]\in[(1-\alpha-3\mu_{i-1})n_i,n_i]$. Therefore, the hypergeometric variant \eqref{eq: absolute hypergeometric chernoff} of the Chernoff bound implies
			\begin{align*}\P\Big[&\deg_{G}(S,U_i)\le\big(1-\alpha-3\mu_{i-1}-2n_{i-1}^{-1/3}\big)n_i\Big]\le\P\Big[\deg_{G}(S,U_i)\le \E[\deg_{G}(S,U_i)]- 2n_{i-1}^{-1/3}n_i\Big]\\
				&\le \exp\!\left(-\frac{\left(2n_{i-1}^{-1/3}n_i\right)^2}{2\E[\deg_G(S,U_i)]}\right)\le \exp\!\left(-\frac{\left(2n_{i-1}^{-1/3}n_i\right)^2}{2n_i}\right)= \exp\!\left(-2n_{i-1}^{-2/3}n_i\right)\\
				&\le \exp\!\left(-2n_{i-1}^{-2/3}(\mu n_{i-1}-1)\right)\le \exp\!\left(-\mu n_{i-1}^{1/3}\right).\end{align*}
			Thus, if $m'\le n_{i-1}$ is large enough, we can use a union bound over all $S\in\binom{U_{i-1}}{d-1}$ and get that, with high probability, there is a choice for $U_i$ such that $\deg_{G}(S,U_i)\ge\big(1-\alpha-3\mu_{i-1}-2n_{i-1}^{-1/3}\big)n_i$ holds for all $S\in \binom{U_{i-1}}{d-1}$. Fix such a choice of $U_i$. Then \[3\mu_{i-1}+2n_{i-1}^{-1/3}\le3n^{-1/3}\sum_{j=1}^{i-1}\mu^{-(j-1)/3}+2\mu^{-(i-1)/3}n^{-1/3} \le 3n^{-1/3}\!\left(\sum_{j=1}^{i-1}\mu^{-(j-1)/3}+\mu^{-(i-1)/3}\right)= 3\mu_i\] implies that $U_0,\dots,U_i$ form a $(\alpha +3\mu_{i},\mu,n_{i})$-vortex in~$G$. Repeating this for all $i\in[\ell]$, we finally obtain an $(\alpha +3\mu_{\ell},\mu,m)$-vortex and because of \eqref{eq: alpha_i <= alpha/3}, $\mu_\ell\le\mu^{2}/3$, the lemma follows.
		\end{proof}
		
		We are now ready to prove the main theorem on extra-tight Euler trails. The proof consists of the following steps:
		\begin{proofoutline}
			\item Fix a vortex: Apply \Cref{lem:There is a vortex} to get a vortex $U_0\supseteq U_1\supseteq\dots\supseteq U_\ell$.
			\item Build an absorber for the final set: Apply \Cref{lem: Absorber Lemma} to obtain an absorber for $U_\ell$ and connect it to one of the given ends $(v_d',\dots,v_1')$. The other end of the absorber is extended into $U_\ell$ such that it can be connected at the very end when all but the edges in $U_\ell$ are used. We call the resulting extra-tight trail $\Tend$. In the following, we want to cover the edges of $G'\coleq G-\Tend$. The other end $(v_1,\dots,v_d)$ is extended by $d$ arbitrary vertices to form the extra-tight trail $T_0$.
			\item Iterative usage of Cover Down: Given an extra-tight trail $T_i$ covering $G'\backslash G'[U_i]$, we show how this can be extended to an extra-tight trail $T_{i+1}$ covering $G'\backslash G'[U_{i+1}]$ using the Cover Down \Cref{lem: Cover Down}. This is illustrated in \Cref{fig: Cover Down step}. 
			\item Final connecting and absorbing the leftover. In the end, we have an extra-tight trail $T_\ell$ which covers all edges outside of $G'[U_\ell]$. We connect it to the end of $\Tend$. By the property of the absorber, we can absorb the remaining uncovered edges into $\Tend$, yielding an extra-tight trail covering all edges of $G$.
		\end{proofoutline}
		
		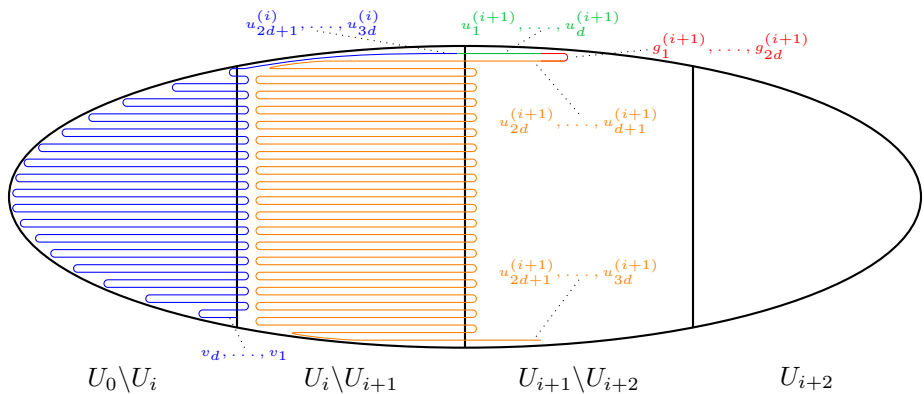
\begin{figure}[h]
			\centering
			\begin{tikzpicture}
				\draw[thick] (0,0) ellipse [x radius=6cm, y radius=2cm];
				\draw[thick] (-3,-1.732) -- (-3,1.732);
				\draw[thick] (0,-2)      -- (0,2);
				\draw[thick] (3,-1.732)  -- (3,1.732);
				\draw[blue] (-3.,-1.6) -- (-3.45,-1.6) arc (270:90:0.05) -- (-2.9,-1.5) arc (-90:90:0.05) -- (-4.15,-1.4) arc (270:90:0.05)-- (-2.9,-1.3) arc (-90:90:0.05) -- (-4.7,-1.2) arc (270:90:0.05) -- (-2.9,-1.1) arc (-90:90:0.05)
				-- (-5.1,-1.) arc (270:90:0.05) -- (-2.9,-.9) arc (-90:90:0.05)
				-- (-5.4,-.8) arc (270:90:0.05) -- (-2.9,-.7) arc (-90:90:0.05)
				-- (-5.6,-.6) arc (270:90:0.05) -- (-2.9,-.5) arc (-90:90:0.05)
				-- (-5.8,-.4) arc (270:90:0.05) -- (-2.9,-.3) arc (-90:90:0.05)
				-- (-5.9,-.2) arc (270:90:0.05) -- (-2.9,-.1) arc (-90:90:0.05)
				-- (-5.9,-.0) arc (270:90:0.05) -- (-2.9,.1) arc (-90:90:0.05)
				-- (-5.86,.2) arc (270:90:0.05) -- (-2.9,.3) arc (-90:90:0.05)
				-- (-5.75,.4) arc (270:90:0.05) -- (-2.9,.5) arc (-90:90:0.05)
				-- (-5.56,.6) arc (270:90:0.05) -- (-2.9,.7) arc (-90:90:0.05)
				-- (-5.25,.8) arc (270:90:0.05) -- (-2.9,.9) arc (-90:90:0.05)
				-- (-4.9,1) arc (270:90:0.05) -- (-2.9,1.1) arc (-90:90:0.05)
				-- (-4.45,1.2) arc (270:90:0.05) -- (-2.9,1.3) arc (-90:90:0.05)
				-- (-3.8,1.4) arc (270:90:0.05) -- (-2.9,1.5) arc (-90:90:0.05)
				-- (-3.05,1.6) arc (270:90:0.05) -- (-2.9,1.7) to[in=180,out=10] (-.1,1.9);
				\draw[mygreen](-.1,1.9) -- (1.,1.9);
				\draw[red](1.0,1.9) -- (1.3,1.9) arc (90:-90:0.05) -- (1.0,1.8);
				\draw[orange](1.0,-1.9) -- (-1.4,-1.9) to [in=180,out=180] (-2.2,-1.8)
				-- (.1,-1.8) arc (-90:90:0.05) -- (-2.7,-1.7) arc (270:90:0.05)
				-- (.1,-1.6) arc (-90:90:0.05) -- (-2.7,-1.5)arc (270:90:0.05)
				-- (.1,-1.4) arc (-90:90:0.05) -- (-2.7,-1.3)arc (270:90:0.05)
				-- (.1,-1.2) arc (-90:90:0.05) -- (-2.7,-1.1)arc (270:90:0.05)
				-- (.1,-1.) arc (-90:90:0.05) -- (-2.7,-.9)arc (270:90:0.05)
				-- (.1,-.8) arc (-90:90:0.05) -- (-2.7,-.7)arc (270:90:0.05)
				-- (.1,-.6) arc (-90:90:0.05) -- (-2.7,-.5)arc (270:90:0.05)
				-- (.1,-.4) arc (-90:90:0.05) -- (-2.7,-.3)arc (270:90:0.05)
				-- (.1,-.2) arc (-90:90:0.05) -- (-2.7,-.1)arc (270:90:0.05)
				-- (.1,-.) arc (-90:90:0.05) -- (-2.7,.1)arc (270:90:0.05)
				-- (.1,.2) arc (-90:90:0.05) -- (-2.7,.3)arc (270:90:0.05)
				-- (.1,.4) arc (-90:90:0.05) -- (-2.7,.5)arc (270:90:0.05)
				-- (.1,.6) arc (-90:90:0.05) -- (-2.7,.7)arc (270:90:0.05)
				-- (.1,.8) arc (-90:90:0.05) -- (-2.7,.9)arc (270:90:0.05)
				-- (.1,1.) arc (-90:90:0.05) -- (-2.7,1.1)arc (270:90:0.05)
				-- (.1,1.2) arc (-90:90:0.05) -- (-2.7,1.3)arc (270:90:0.05)
				-- (.1,1.4) arc (-90:90:0.05) -- (-2.7,1.5)arc (270:90:0.05)
				-- (.1,1.6) arc (-90:90:0.05) -- (-2.5,1.7) to[out=180,in=180] (-1.8,1.8)
				-- (1.3,1.8);
				\textbe{-4.5,-2}{$U_{0}\backslash U_i$}
				\textbe{-1.5,-2}{$U_i\backslash U_{i+1}$}
				\textbe{1.5,-2}{$U_{i+1}\backslash U_{i+2}$}
				\textbe{4.5,-2}{$U_{i+2}$}
				\def\textcol{orange}
				\textat{1.5,-1}{\tiny $u_{2d+1}^{(i+1)},\dots,u_{3d}^{(i+1)}$}
				\textat{1.5,1}{\tiny $u_{2d}^{(i+1)},\dots,u_{d+1}^{(i+1)}$}
				\draw[dotted] (1.5,-1.1) -- (0.9,-1.9) (1.5,1.1) -- (0.9,1.8) (-2.9,-2) -- (-3.1,-1.6) (-1.95,2.2) -- (-.2,1.9)
				(1,2.2) -- (.5,1.9) (2.4,1.95) -- (1.4,1.85);
				\def\textcol{mygreen}
				\textat{1,2.3}{\tiny $u_1^{(i+1)},\dots,u_d^{(i+1)}$}
				\def\textcol{red}
				\textat{3.5,2}{\tiny $g_1^{(i+1)},\dots,g_{2d}^{(i+1)}$}
				\def\textcol{blue}
				\textat{-2,2.3}{\tiny$u_{2d+1}^{(i)},\dots,u_{3d}^{(i)}$}
				\textat{-2.9,-2.1}{\tiny$v_d,\dots,v_1$}
			\end{tikzpicture}
			\caption{The inductive step using the Cover Down Lemma. By the inductive hypothesis, there is a (blue) extra-tight trail covering $G'\backslash G'[U_i]$. One end is extended by a (green) extra-tight trail into $U_{i+1}\backslash U_{i+2}$. Then, the Cover Down \Cref{lem: Cover Down} gives us an (orange) extra-tight trail covering all remaining edges in $G'[U_i]\backslash G'[U_{i+1}]$. Finally, these two trails are connected via a short (red) extra-tight trail.}
			\label{fig: Cover Down step}
		\end{figure}
		
		\begin{proof}[Proof of \Cref{theo: large min degree implies extra-tight trail}]
			Let $1/n\ll \eta \ll 1/m'\ll\alpha,\mu\ll 1/d$ with $\mu^{1.5}<\alpha<\frac{1}{2}\mu^{1.1}$.
			
			\textbf{Step 1.} Apply \Cref{lem:There is a vortex} on $G$ to get an $(\alpha+\mu^2,\mu,m)$-vortex $U_0\supseteq U_1\supseteq\dots\supseteq U_\ell$ for some $\floor{\mu m'}\le m\le m'$. In particular, we have $1/n\ll \eta \ll 1/m\ll\alpha,\mu\ll 1/d$.
			
			\textbf{Step 2.} We apply the Absorber \Cref{lem: Absorber Lemma} with the parameters $U_{\ref{lem: Absorber Lemma}}\coleq U_\ell$, $\alpha_{\ref{lem: Absorber Lemma}}\coleq 2(\mu+\alpha)$, and $G_{\ref{lem: Absorber Lemma}}\coleq G[(U_0\backslash U_1)\cup U_\ell]-\big\{\{v_1,\dots,v_d\},\{v_1',\dots,v_d'\}\big\}$ to get an extra-tight trail $A=(a_1,\dots,a_q)$ with ends $(a_1,\dots, a_d)$ and $(a_{q-(d-1)},\dots,a_q)$ such that $U_\ell$ is independent in $A$ and $\Delta(A)\le \eta n$. Furthermore, $A$ has the following absorbing property: For each $d$-graph $L$ on $U_\ell$ with $\delta(L)\ge (1-4\mu)m$ where each 1-degree is divisible by $d^2$, there is an extra-tight trail with edge set $L\cup A$ having the same ends $(a_1,\dots, a_d)$ and $(a_{q-(d-1)},\dots,a_q)$. Note that since $U_\ell$ is independent in $A$ and by the way we chose $G_{\ref{lem: Absorber Lemma}}$, we are guaranteed that neither $\{v_1,\dots,v_d\}$ nor $\{v_1',\dots,v_d'\}$ is in $A$.
			
			Let $\{g'_1,\dots,g_d'\}\in G[U_\ell]-\big\{\{v_1,\dots,v_d\},\{v_1',\dots,v_d'\}\big\}$. We start building our final extra-tight trail by gluing one end of $A$ to one of the given ends $(v_d',\dots,v_1')$ and the other end of $A$ to some vertices to $(g_1',\dots,g_d')$: by \Cref{lem: glue two extra-tight trails}, we can find vertices $g_1,\dots,g_{4d}\in U_0\backslash U_1$ such that \[\big(g_1',\dots,g_d',g_1,\dots,g_{2d},a_1,\dots,a_q,g_{2d+1},\dots,g_{4d},v_d',\dots,v_1'\big)\] is an extra-tight trail $\Tend$ in~$G$. Note that $\Delta(\Tend)\le 2\eta n$ and $\Delta(\Tend[U_i])\le 2$ for all $i>0$. Let $G'\coleq G\backslash \Tend$.
			
			\textbf{Step 3.}  Using \Cref{lem: glue two extra-tight trails}, we greedily pick some $u_{2d+1}^{(0)},\dots,u_{3d}^{(0)}\in U_0\backslash U_1$ such that
			\[\big(v_1,\dots,v_d,u_{2d+1}^{(0)},\dots,u_{3d}^{(0)}\big)\] forms an extra-tight trail $T_0$ in $G'$. Inductively, assume that for some $i\in[\ell-1]_0$, we have already found an extra-tight trail $T_i$ in $G'$ with ends $(v_1,\dots,v_d)$ and $\big(u_{2d+1}^{(i)},\dots,u_{3d}^{(i)}\big)$ such that
			\begin{enumerate}
				\item\label{enum: indu: u in U} $u_{2d+1}^{(i)},\dots,u_{3d}^{(i)}\in U_{i}$
				\item\label{enum: indu: E(T_i)} $G'\backslash G'[U_i]\subs T_i\subs \big(G'\backslash G'[U_{i+1}]\big)\cup\big\{\{v_1,\dots,v_d\}\big\}$
				\item\label{enum: indu: Delta(U_i)} $\Delta(T_i[U_i])\le 2\mu^2\abs{U_i}$.
			\end{enumerate}
			Note that these conditions are satisfied for $i=0$. We show now how to get the same result for $i+1$ instead of $i$ (where $U_{\ell+1}\coleq \emptyset$).
			
			Using \Cref{lem: glue two extra-tight trails}, we can find $u_1^{(i+1)},\dots,u_{d+1}^{(i+1)}\in U_{i+1}\backslash U_{i+2}$ such that 
			\[\big(\underbrace{v_1,\dots,v_d,\dots,u_{2d+1}^{(i)},\dots,v_{3d}^{(i)}}_{T_i},u_1^{(i+1)},\dots,u_d^{(i+1)}\big)\]
			forms an extra-tight trail $T_{i+1}'$ in $G'$. Note that by \ref{enum: indu: Delta(U_i)}, $\Delta(T_{i+1}'[U_{i}])\le 3 \mu^2 |U_i|$. Let $G^{(i+1)}\coleq G'\backslash T_{i+1}'$. Thus, $G^{(i+1)}$ is obtained from $G$ by deleting the edge-sets of two extra-tight trails $T_{i+1}'$ and $\Tend$. Since two of the ends of $T_{i+1}'$ and $\Tend$ are $(v_1,\dots,v_d)$ and $(v_d',\dots,v_1')$, and the other two ends only have vertices in $U_{i+1}$, the degree condition \eqref{eq: d-trail divisible} on $G$ implies that all vertices in $U_i\backslash U_{i+1}$ have degree divisible by $d^2$ in $G^{(i+1)}$. 
			
			Hence, we can apply the Cover Down \Cref{lem: Cover Down} with $G_{\ref{lem: Cover Down}}\coleq G^{(i+1)}[U_i]-G^{(i+1)}[U_{i+2}]$, $U_{\ref{lem: Cover Down}}\coleq U_{i+1}$, and $\alpha_{\ref{lem: Cover Down}}\coleq 2\alpha$. 
			To check the minimum degree condition of the Cover Down Lemma, let $S\in \binom{U_i}{d-1}$. By property \ref{enum: vortex, minimum degree} of an $\big(\alpha+\mu^2,\mu,m\big)$-vortex, we have $\deg_G(S,U_i)\ge \big(1-\alpha-\mu^2\big)\abs{U_i}$. 
			Therefore, \begin{align*}\deg_{G_{\ref{lem: Cover Down}}}(S)&\ge \big(1-\alpha-\mu^2\big)\abs{U_i}-\Delta(T_{i+1}'[U_i])-\Delta(\Tend[U_i])-\abs{U_{i+2}}\\&\ge \big(1-\alpha-\mu^2\big)\abs{U_i}-3\mu^2 |U_i|
				-\Delta(\Tend[U_i])-\mu^{2} |U_i|\\&\ge (1-2\alpha)\abs{U_i}.\end{align*}
			Hence, all conditions of the Cover Down \Cref{lem: Cover Down} are satisfied and we get $u_{d+1}^{(i+1)},\dots,u_{3d}^{(i+1)}\in U_{i+1}$ together with an extra-tight trail $T_{i+1}''$ with ends $\big(u_{d+1}^{(i+1)},\dots,u_{2d}^{(i+1)}\big)$ and $\big(u_{2d+1}^{(i+1)},\dots, u_{3d}^{(i+1)}\big)$ such that $G'\backslash G'[U_{i+1}]\subs T_{i+1}''\subs G'\backslash G'[U_{i+2}]$ and $\Delta(T_{i+1}''[U_{i+1}])\le \mu^2\abs{U_i}$. 
			By \Cref{lem: glue two extra-tight trails}, we can find vertices $g_1^{(i+1)},\dots,g_{2d}^{(i+1)}\in U_{i+1}\backslash U_{i+2}$ such that
			\[\big(\underbrace{v_1,\dots,v_d,\dots,u_1^{(i+1)},\dots,u_d^{(i+1)}}_{T_{i+1}'},g_1^{(i+1)},\dots,g_{2d}^{(i+1)},\underbrace{u_{d+1}^{(i+1)},\dots,u_{2d}^{(i+1)},\dots,u_{2d+1}^{(i+1)},\dots,u_{3d}^{(i+1)}}_{T_{i+1}''}\big)\] forms an extra-tight trail $T_{i+1}$ in $G'$ which concludes the inductive step.
			
			\textbf{Step 4.} In the end, we have an extra-tight trail $T_{\ell}$ with ends $(v_1,\dots,v_d)$ and $\big(u_{2d+1}^{(\ell)},\dots,u_{3d}^{(\ell)}\big)$ where $u_{2d+1}^{(\ell)},\dots,u_{3d}^{(\ell)}\in U_\ell$, $G'\backslash G'[U_\ell]\subs T_{\ell}\subs G'$, and $\Delta(T_{\ell}[U_\ell])\le 2\mu^2\abs{U_\ell}$. Let $G''\coleq G'[U_\ell]\backslash T_{\ell}$. It consists of $m$ vertices and by property \ref{enum: vortex, minimum degree} of an $\big(\alpha+\mu^2,\mu,m\big)$-vortex, we have $\delta(G'')\ge \big(1-\alpha-\mu^2\big)m-\Delta(T_{\ell}[U_\ell])\ge (1-1.5\alpha)m$. Thus, \Cref{lem: glue two extra-tight trails} implies there exist $2d$ vertices in $U_\ell$ that connect the end $\big(u_{2d+1}^{(\ell)},\dots,u_{3d}^{(\ell)}\big)$ of $T_{\ell}$ to the end $(g_1',\dots,g_d')$ of $\Tend$ to form one long extra-tight trail $T_{\ell+1}$ with the ends $(v_1,\dots,v_d)$ and $(v_d',\dots,v_1')$. But now, the remaining graph $L$ of uncovered edges of $G[U_\ell]$ fulfills the conditions of the absorbing property of~$A$. Hence, by changing the vertex sequence of $A$ within $\Tend$ within $T_{\ell+1}$, we get an extra-tight trail that covers all the edges of~$G$.
		\end{proof}

		\section{Absorber Lemma}\label{sec: Absorber}

The goal of this section is to prove the Absorber \Cref{lem: Absorber Lemma}, i.e.\@ constructing an extra-tight trail $A$  that can absorb the edges of any leftover $L \subseteq G[U]$ into an extra-tight trail with the same ends.
As indicated in Section~\ref{sec: Proof overview}, we draw inspiration from the approach of Glock, Joos, Kühn, and Osthus in \cite{glock2020euler}, where they construct a tight $d$-uniform Euler tour in dense enough hypergraphs by first constructing an appropriate  ``backbone'' tight trail and then merging into it one by one the tight cycles of an appropriate decomposition of the rest.
If we were to follow~\cite{glock2020euler} closely, we would just build the sequence of our extra-tight trail $A$ in such a way that it contains, disjointly, every $(d-1)$-tuple $(u_1,\dots,u_{d-1})$ of distinct vertices from $U$. Then, for each cycle $C$ of a decomposition of some $L$ into extra-tight cycles, we would divert the sequence of $A$ immediately after the appearance of some consecutive $(d-1)$-tuple of $C$, sending it first to go around in $C$ and only after that proceed further on $A$ (see Figure~\ref{fig:cycle absorption}). 
\begin{figure}[h] 
    \centering
    \begin{tikzpicture}[scale = 0.8]
    \draw[red] (0,1.1) ellipse [x radius = 4cm, y radius = 1cm];
    
\def\s{1}

\draw[
  draw=none,
  preaction={
    decorate,
    decoration={
      markings,
      mark=between positions 0 and 1 step 15pt
        with {\fill circle[radius=2.5pt];}
    }
  }
]
  (0,1.1) ++(25:{4cm*\s} and {1cm*\s})
  arc[
    start angle=25,
    end angle=155,
    x radius={4cm*\s},
    y radius={1cm*\s}
  ];
    
    \draw[blue] (-4,0) -- (4,0);
    \draw[fill=\mycolor, draw=\mycolor] (0,.1) ellipse [x radius=3pt, y radius=7pt];
    \draw[fill=\mycolor, draw=\mycolor] (1,0.1) ellipse [x radius=3pt, y radius=7pt];
    \draw[fill=\mycolor, draw=\mycolor] (2,0.1) ellipse [x radius=3pt, y radius=7pt];
    \draw[fill=\mycolor, draw=\mycolor] (-1,0.1) ellipse [x radius=3pt, y radius=7pt];
    \draw[fill=\mycolor, draw=\mycolor] (-2,0.1) ellipse [x radius=3pt, y radius=7pt];
    \newdot{w0}{-3,0}
    \newdot{wd}{3,0}
    \newdot{ck}{-3.7,.75}
    \newdot{c1}{3.7,.75}
    \newdot{ud}{{0 + 4*cos(-45.40933)}, {1.1 + 1*sin(-45.40933)}}
    \newdot{u0}{{0 + 4*cos(225.40933)}, {1.1 + 1*sin(225.40933)}}
    \textbe{w0}{$w_0$}
    \textbe{-2,0}{$u_1$}
    \textbe{-1,0}{$u_2$}
    \textbe{2,0}{$u_{d-1}$}
    \textbe{wd}{$w_d$}
    \textat{3.9,.45}{$c_1$}
    \textat{3.9,1.8}{$c_2$}
    \textat{-3.9,1.9}{$c_{k-1}$}
    \textat{-3.9,.45}{$c_k$}
    \textab{u0}{$u_0$}
    \textab{ud}{$u_d$}
    {\def\textcol{blue}\textat{3.8,-0.3}{\scalebox{1.3}{$A$}}}
    {\def\textcol{red}\textat{0,1}{\scalebox{1.3}{$C$}}}
    \textat{5.5,1.1}{\scalebox{2}{$\leadsto$}}
    \begin{scope}[shift = {(11,0)}]
    \draw[draw=purple]
    ({0 + 4*cos(-45.40933)}, {1.1 + 1*sin(-45.40933)}) arc[start angle=-45.40933, end angle=225.40933, x radius=4cm, y radius=1cm];

\draw[
  draw=none,
  preaction={
    decorate,
    decoration={
      markings,
      mark=between positions 0 and 1 step 15pt
        with {\fill circle[radius=2.5pt];}
    }
  }
]
  (0,1.1) ++(25:{4cm*\s} and {1cm*\s})
  arc[
    start angle=25,
    end angle=155,
    x radius={4cm*\s},
    y radius={1cm*\s}
  ];
    
    \newdot{ud}{{0 + 4*cos(-45.40933)}, {1.1 + 1*sin(-45.40933)}}
    \newdot{u0}{{0 + 4*cos(225.40933)}, {1.1 + 1*sin(225.40933)}}
    \node (a) at ({0 + 4*cos(-45.40933)}, {1.1 + 1*sin(-45.40933)}) {};
    \draw[purple] (-4,0) to[out = 0, in=193.84] (ud);
    \draw[purple] (4,0) to[out = 180, in=-13.84] (u0);
    \draw[fill=\mycolor, draw=\mycolor] (0,-.05) ellipse [x radius=3pt, y radius=7pt];
    \draw[fill=\mycolor, draw=\mycolor] (1,-0.) ellipse [x radius=3pt, y radius=7pt];
    \draw[fill=\mycolor, draw=\mycolor] (2,0.08) ellipse [x radius=3pt, y radius=7pt];
    \draw[fill=\mycolor, draw=\mycolor] (-1,-0.) ellipse [x radius=3pt, y radius=7pt];
    \draw[fill=\mycolor, draw=\mycolor] (-2,0.08) ellipse [x radius=3pt, y radius=7pt];
    \newdot{w0}{-3,0}
    \newdot{wd}{3,0}
    \newdot{ck}{-3.7,.75}
    \newdot{c1}{3.7,.75}
    \textbe{w0}{$w_0$}
    \textbe{-2,0}{$u_1$}
    \textbe{-1.,-.08}{$u_2$}
    \textbe{2,0}{$u_{d-1}$}
    \textbe{wd}{$w_d$}
    \textat{3.9,.45}{$c_1$}
    \textat{3.9,1.8}{$c_2$}
    \textat{-3.9,1.9}{$c_{k-1}$}
    \textat{-3.9,.45}{$c_k$}
    \textab{u0}{$u_0$}
    \textab{ud}{$u_d$}
    \end{scope}
    \end{tikzpicture}
    \caption{Absorbing a cycle $C$ into the trail $A$. Vertices that appear in both $A$ and $C$ are shown in thicker ellipses.
   The new sequence is $(\ldots , w_0, u_1, u_2, \ldots , u_{d-1}, u_d, c_1 , \ldots , c_k, u_0, u_1, u_2, \ldots , u_{d-1}, w_d, \ldots )$.
    \label{fig:cycle absorption}}
\end{figure}
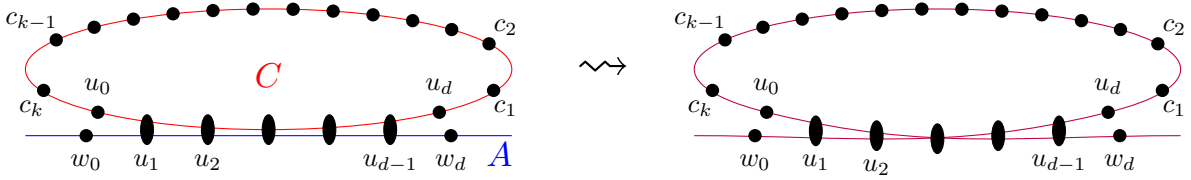

This absorption approach works like a charm in \cite{glock2020euler} for tight structures, but for extra-tight structures problems arise:  the insertion of $C$ into the sequence of $A$ produces a sequence which covers a set of edges a little bit different from $A\cup C$. 

Namely, the $2(d-1)$ edges in 
\begin{equation} \label{eq: E1} E_1 = \Big\{\{u_0,u_1,\dots,u_{i-1},u_{i+1},\dots, u_{d-1},w_d\}, \{w_0,u_1,\dots,u_{i-1},u_{i+1},\dots, u_{d-1},u_d\}: i\in [d-1]\Big\},\end{equation} which were not covered here by $A\cup C$, will be covered  here unnecessarily by the ``rewired'' sequence. On the other hand, those $2(d-1)$ edges of $A$ and $C$ which jump over some vertex in $(u_1, \ldots , u_{d-1})$, i.e.  those contained in
\begin{equation} \label{eq: E2} E_2 =\Big\{\{w_0,u_1,\dots,u_{i-1},u_{i+1},\dots, u_{d-1},w_d\}, \{u_0,u_1,\dots,u_{i-1},u_{i+1},\dots, u_{d-1},u_d\}: i\in [d-1]\Big\}, \end{equation}
are not covered anymore by the new, rewired sequence (cf.~\Cref{fig:cycle absorption}).

To overcome this, we introduce our notion of an $(E_1,E_2)$-switcher, which is a pair of extra-tight trails whose edge set differs exactly in $E_1$ and $E_2$. 
\begin{definition}
 Given $d$-graphs $E_1, E_2$, an \emph{$(E_1,E_2)$-switcher} 
 is a pair $(T_1, T_2)$ of extra-tight trails with the same ends such that $T_1\setminus T_2 =E_1$, $T_2\setminus T_1 =E_2$ and $V(E_1\cup E_2)$ is independent in $T_1\setminus E_1$.
\end{definition}

The last property, that $V(E_1\cup E_2)$ is independent in $T_1\setminus E_1$, will later be helpful to ensure that no edge is repeated in the final absorber $A$ and that $U$ is independent in $A$.

The plan is that our eventual absorption trail $A$ will also have a segment where the sequence $T_1$ will be contained for every pair $(E_1,E_2)$ sets which could arise at some cycle-insertion.
If our absorption procedure for the cycles of a decomposition of  a particular $L$ happened to use a cycle-insertion with corresponding pair of edge sets $(E_1, E_2)$, then we also switch the corresponding sequence  $T_1$ in $A$ to $T_2$, thus undoing the wrong the cycle-insertion did to the set of edges covered.   The main result of the next subsections is that the desired switchers do exist.

\begin{lemma}\label{lem: switcher}
Let $d\ge 2$, $u_0,\dots,u_d,w_0,w_d$ be $d+3$ distinct vertices. Then for the two families $E_1$ and $E_2$ of $d$-sets defined in (\ref{eq: E1}) and (\ref{eq: E2}), there exists an $(E_1,E_2)$-switcher $(T_1,T_2)$.
\end{lemma}

For the proof of the above Switcher Lemma note that one can get from the family $E_1$ to the family $E_2$ by switching the occurrences of the vertices $u_0$ and $w_0$, in other words $E_1(u_0)  = E_2 (w_0)$ and $E_1(w_0) = E_2 (u_0)$.
        This motivates us to build an $(E_1,E_2)$-switcher $(T_1,T_2)$ where one obtains the vertex sequence of the extra-tight trail $T_2$ from the vertex sequence of the extra-tight trail $T_1$ by just swapping all occurrences of the vertices $u_0$ and $w_0$. Doing such a swap in {\em any} extra-tight trail creates another extra-tight trail and the status of an edge, in terms of whether it is covered, can only change if it contains exactly one of the swapped vertices. We will in fact make sure that $T_1$ does not cover an edge containing both $u_0$ and $w_0$. The precise set of edges that changes its coverage status is 
\begin{align*}
T_1\setminus T_2 & = \big(\!\left(T_1 (u_0) \setminus T_2(u_0) \right) \vee \{ u_0\} \big) \cup \big(\!\left( T_1 (w_0) \setminus T_2(w_0)\right)  \vee \{ w_0\} \big)\\
T_2 \setminus T_1 & = \big(\!\left(T_2 (u_0) \setminus T_1(u_0) \right) \vee \{ u_0\} \big) \cup \big(\!\left( T_2 (w_0) \setminus T_1(w_0)\right)  \vee \{ w_0\} \big),
\end{align*} 
where $A \vee B:= \{ a \cup b \mid a \in A,b \in B  \}$.
The swap causes the link graphs of $u_0$ and $w_0$ to swap, i.e.~$T_2(u_0) = T_1(w_0)$ and $T_2(w_0) = T_1(u_0)$. 
Thus, in order for $T_1$ and $T_2$ to satisfy the switcher property, we will need to make sure that the links of $u_0$ and $w_0$ in $T_1$ differ in exactly the right way, namely that
\begin{align*}
    T_1(u_0)\backslash T_1(w_0) & =E_1(u_0)  = E_2 (w_0)
    \mbox{ and}  \\ T_1(w_0)\backslash T_1(u_0) & =E_1(w_0) = E_2 (u_0).
\end{align*}

Since the construction of an $(E_1,E_2)$-switcher is quite complicated, as a warm-up, in the next subsection we motivate our approach in the simplest case $d=2$. This is only meant for the benefit of the exposition, the reader is welcome to skip over to the actual general proof.

\subsection{Motivating sketch for \texorpdfstring{$d=2$}{d=2}}

The issue of inserting the sequence of an extra-tight cycle into the sequence of an extra-tight trail is illustrated in \Cref{fig:switch for d=2} for the case $d=2$. 
There, the black and red edges form an extra-tight trail $(\dots,w_{-1},w_0,u_1,w_2,w_3,\dots)$ and an extra-tight cycle $(\dots,c_k,u_0,u_1,u_2,c_1,\dots)$ intersecting in the vertex~$u_1$.
If one attempted to merge the edges of these two into one extra-tight trail by simply taking the vertex-sequence $(\dots, w_0,u_1,u_2,c_1,\dots,c_k,u_0,u_1,w_2,\dots)$, then the set $E_2=\{w_0w_2,u_0u_2\}$ of red edges would no longer be used, and the set $E_1=\{u_0w_2,w_0u_2\}$ of blue edges would be used extra.

\begin{figure}[h]
    \centering
    \scalebox{1}{
    \begin{tikzpicture}[scale=2]
    	\def\y{0.3}
    	\draw[thick,line width=1mm, red, opacity=0.4](2,\y) -- (3,\y);
        \draw (0,\y) -- (2,\y) -- (2.5,0.87) -- (3,\y) -- (5,\y) -- (4.5,.87) -- (.5,.87) -- (0,\y) (.5,.87) -- (1,\y) -- (1.5,.87) -- (2,\y) -- (2.5,.87) -- (3,\y) -- (3.5,.87) -- (4,\y) -- (4.5,.87) -- (5,\y) (4.5,0.87) -- (5,0.87) (0,0.87) -- (.5,0.87);
        \draw[thick](2,\y) -- (3,\y);
        \begin{scope}[shift = {(2.5,2.12)}]
        	\pgfmathsetmacro{\x}{0.45*1.25}
        	\draw[thick,line width=1mm, red, opacity=0.4] (295.714:\x) -- (244.285:\x);
        	\draw[thick,line width=1mm, blue, opacity=0.4] (295.714:\x) to[out=230,in=80] (-.5,{\y-2.12}) (.5,{\y-2.12}) to[in=310,out=100] (244.285:\x);
        	\draw[thick] (295.714:\x) -- (244.285:\x);
            \draw (270:1.25) -- (321.428:1.25) -- (12.857:1.25) -- (64.285:1.25) -- (115.714:1.25) -- (167.142:1.25) -- (218.571:1.25) -- (270:1.25);
            \draw (295.714:\x) -- (347.142:\x) -- (38.571:\x) -- (90:\x) -- (141.428:\x) -- (192.857:\x) -- (244.285:\x);
            \draw[dashed, thick] (295.714:\x) to[out=230,in=80] (-.5,{\y-2.12}) (.5,{\y-2.12}) to[in=310,out=100] (244.285:\x);
            \draw (270:1.25) -- (295.714:\x) -- (321.428:1.25) -- (347.142:\x) -- (12.857:1.25) -- (38.571:\x) --(64.285:1.25) -- (90:\x) --(115.714:1.25) -- (141.428:\x) -- (167.142:1.25) -- (192.857:\x) --(218.571:1.25) -- (244.285:\x) -- (270:1.25);
            \textat{-.5,{\y-2.2}}{$w_0$}
            \textat{-1,-1.15}{$w_{-1}$}
            \textat{0,-1.45}{$u_1$}
            \textat{1,-1.15}{$w_3$}
            \textat{.5,{\y-2.2}}{$w_2$}
            \textat{-1,-.9}{$c_k$}
            \textat{1,-.9}{$c_1$}
            \textat{-.25,-.37}{$u_0$}
            \textat{.23,-.37}{$u_2$}
        \end{scope}
    \end{tikzpicture}}
    \caption{Merging the sequence of an extra-tight cycle into an extra-tight trail in the two-dimensional case.}
    \label{fig:switch for d=2}
\end{figure}
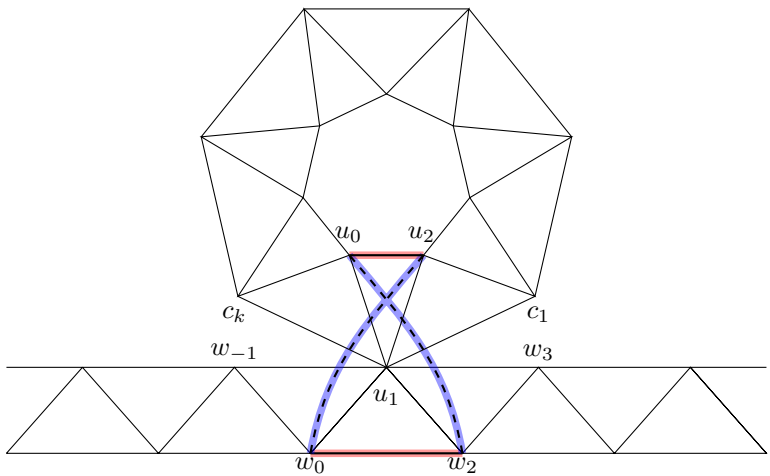

By our general contemplation, in order to construct the switcher, we need to make sure the ($1$-uniform) links of $u_0$ and $w_0$ in $T_1$ differ exactly the right way. 
\begin{align*}
    T_1(u_0)\backslash T_1(w_0) & =  \{w_2\} 
    \mbox{ and}  \\ T_1(w_0)\backslash T_1(u_0) & = \{u_2\}.
\end{align*}
Denoting these graphs by $\tilde G_1:=T_1(u_0)$ and $\tilde G_2:=T_1(w_0)$ this means $\tilde{G}_1 \setminus \tilde{G}_2 = \{ w_2\}$ and $\tilde{G}_2\setminus \tilde{G}_1 = \{ u_2 \}$.
In the case $d=2$, this turns out to be not that hard.
For example,
\begin{align*}
T_1=(x_1,w_2,u_0,x_2,x_3,w_0,u_2,x_1),
\end{align*}
works as $\tilde{G}_1=\{x_1,w_2,x_2,x_3\}$ and $\tilde{G}_2=\{x_2,x_3,u_2,x_1\}$.
With $T_2=(x_1,w_2,w_0,x_2,x_3,u_0,u_2,x_1)$ obtained by switching just $w_0$ and $u_0$, $(T_1,T_2)$ is indeed an $(E_1,E_2)$-switcher.

Unfortunately, we were not able to find a generalisation of this simple construction for larger $d$.
Towards the general case, as a simpler goal, we now indicate how to build an extra-tight trail $T$ where $u_0$ and $w_0$ have the same link graph $\tilde{G}$, and hence swapping all occurrences of $u_0$ and $w_0$ in the sequence produces an extra-tight trail covering the very same set of edges.
Note that each occurrence of a vertex $v$ within the sequence of an extra-tight trail, which is not at one of the ends, results in four neighbours of $v$, spanning a path of three edges in the extra-tight trail. Namely, in the extra-tight trail $( \ldots , v_1, v_2, v, v_3, v_4, \ldots )$, the neighbours $v_1, v_2, v_3, v_4$ of $v$ span the edges $v_1v_2, v_2v_3, v_3v_4$ in the extra-tight trail. Furthermore, different occurrences of $v$ in the sequence necessarily generate vertex-disjoint paths. 

On the one hand, for our task, the vertices occurring in the link of $u_0$ and $w_0$ have to be the same. On the other hand, as an extra-tight trail should cover every $2$-set at most once, we ask the vertex-disjoint paths of length three that appear in the extra-tight trail around an occurrence of $u_0$ in the sequence to be edge-disjoint from those that appear around an occurrence of $w_0$ in the sequence.
While this is not necessary as the example above shoes, it will be crucial for our construction of the switcher.

This motivates us choosing $16$ vertices for the ($1$-uniform) link $\tilde{G}$ of $u_0$ and $w_0$ and aiming to cover them by two sets of four paths of length three in an edge-disjoint fashion (see the blue and red paths on \Cref{fig: switcher motivation}).
We then create our sequence $T$ of length $68$ (cf.~\Cref{fig: switcher motivation}) by inserting $u_0$ in the middle of each blue path and $w_0$ in the middle of each red path and connecting up these eight sequences of five vertices with, say, four new padding vertices in between them (see \Cref{lem: glue two extra-tight trails}). 
The padding vertices are chosen to be all distinct and their role is just to make sure that no pair of vertices appears more than once within three consecutive entries of the sequence. 
The key property of the sequence on \Cref{fig: switcher motivation} is that swapping the occurrences of $u_0$ and $w_0$ does not change the underlying graph.

\begin{figure}[h]
    \centering
    \begin{tikzpicture}
        \foreach \x in {1,...,4} {
            \foreach \y in {1,...,4} {
        \newdot{a\x\y}{{2*\y},{-2*\x}}
        }
        \newdot{a\x5}{10,{-2*\x}}
    }
    \newdot{a01}{2,0}
    \newdot{a02}{4,0}
    \newdot{a03}{6,0}
    \newdot{a04}{8,0}
    \newdot{a20}{0,-4}
    \newdot{a30}{0,-6}
    \newdot{a40}{0,-8}
    \newdot{a51}{2,-10}
    \newdot{a52}{4,-10}
    \newdot{a53}{6,-10}
    \foreach \y in {1,...,4} {
    \newdot{w0}{5,{-2*\y}}
    \textab{w0}{$w_0$}
    \draw[mygreen](w0) to[out=20, in=160] (a\y4);
    \draw[mygreen](a\y1) to[out=20, in=160] (w0);
    }
    \foreach \y in {1,...,4} {
    \newdot{u0}{{2*\y},-5}
    \textle{u0}{$u_0$}
    \draw[mygreen](u0) to[out=-110, in=110] (a4\y);
    \draw[mygreen](a1\y) to[out=-110, in=110] (u0);
    }
    \draw[red] (a11) -- (a14) (a21) -- (a24) (a31) -- (a34) (a41) -- (a44);
    \draw[blue] (a11) -- (a41) (a12) -- (a42) (a13) -- (a43) (a14) -- (a44);
    \draw[mygreen] (a12) -- (a13) (a22) -- (a23) (a32) -- (a33) (a42) -- (a43);
    \draw[mygreen] (a21) -- (a31) (a22) -- (a32) (a23) -- (a33) (a24) -- (a34);
    \draw[red] (a12) to[out=-20, in=200] (a13);
    \draw[red] (a22) to[out=-20, in=200] (a23);
    \draw[red] (a32) to[out=-20, in=200] (a33);
    \draw[red] (a42) to[out=-20, in=200] (a43);
    \draw[blue] (a21) to[out=-70, in=70] (a31);
    \draw[blue] (a22) to[out=-70, in=70] (a32);
    \draw[blue] (a23) to[out=-70, in=70] (a33);
    \draw[blue] (a24) to[out=-70, in=70] (a34);
    \foreach \x in {2,...,4} {
        \draw[dotted] (a0\x) -- (a1\x);
        \draw[dotted] (a\x0) -- (a\x1) (a\x4) -- (a\x5);
    }
    \foreach \x in {1,...,4} {
        \foreach \y in {1,...,4} {
            \draw (a\x\y) circle (2mm);
    		\newdot{a\x\y}{{2*\y},{-2*\x}}
    	}
    }
    \foreach \y in {1,...,4} {
    \newdot{w0}{5,{-2*\y}}
    }
    \foreach \y in {1,...,4} {
    \newdot{u0}{{2*\y},-5}
    }
     \draw (a11) circle (2mm);
    \draw (a12) circle (2mm);
    \draw (a13) circle (2mm);
    \draw (a14) circle (2mm);
    \draw[dotted] (a42) -- (a52);
    \draw[dotted] (a43) -- (a53);
    \draw[dotted] (a45)--(a35) (a25)--(a15)--(a14) (a20) -- (a30) (a41)--(a51) -- (a40) (a02)--(a01) -- (a11) (a04) -- (a03) (a52)--(a53);
    \newdot{b}{0,-5.333}
    \newdot{b}{0,-4.666}
    \newdot{b}{10,-7.333}
    \newdot{b}{10,-6.666}
    \newdot{b}{10,-3.333}
    \newdot{b}{10,-2.666}
    \newdot{b}{4.666,-10}
    \newdot{b}{5.333,-10}
    \newdot{b}{2.666,0}
    \newdot{b}{3.333,0}
    \newdot{b}{6.666,0}
    \newdot{b}{7.333,0}
    \newdot{b}{0.666,-8.666}
    \newdot{b}{1.333,-9.333}
    \textle{1,-2}{start}
    \draw[->] (.9,-2) -- (1.7,-2);
    \textle{8.5,-9.3}{end}
    \draw[->] (8,-8.3) -- (8,-9.1);
    \end{tikzpicture}
    \caption{Extra-tight trail where $w_0$ and $u_0$ have the same link $\tilde{G}$. The 1-edges of $\tilde{G}$ are given by circles. Only those edges of the extra-tight trail are shown (in red, blue, or green) that do not contain a padding vertex. Swapping $u_0$ and $w_0$ in the indicated sequence of vertices results in an extra-tight trail covering the same edges. 
    }
    \label{fig: switcher motivation}
\end{figure}
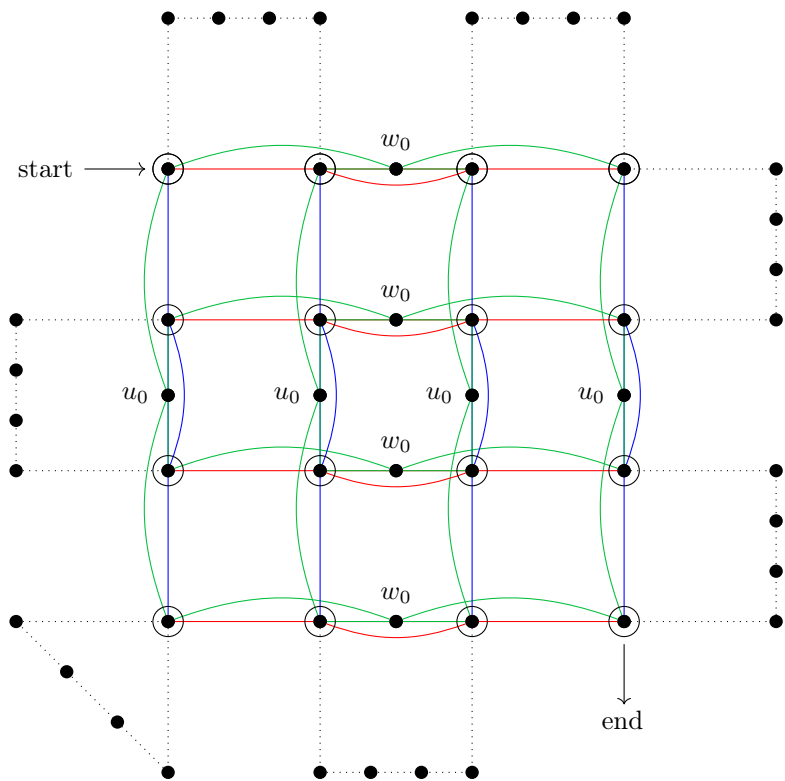

Recall, that in the actual proof, the link graphs of $u_0$ and $w_0$ will not be the same graph $\tilde{G}$, but rather well chosen $\tilde{G}_1$ and $\tilde{G}_2$ such that $\tilde{G}_1 \setminus \tilde{G}_2 = \{ w_2\}$ and $\tilde{G}_2\setminus \tilde{G}_1 = \{ u_2 \}$. 
Then, we still plan to cover both graphs with paths of order $4$ where the corresponding 2-edges are distinct.

\subsection{Constructing switchers for arbitrary \texorpdfstring{$d$}{d}}

For general $d$, we proceed analogously to the above, but the structures get more complex. Similarly to the above construction for $d=2$, we will need to get  hold of some $(d-1)$-graphs $\tilde{G}_1$ and $\tilde{G}_2$ that can play the role of the link of $u_0$ and $w_0$, and hence differ in a very specific way. 
In an extra-tight path $\extratightpath{2d+1}{d}$, the link of the middle vertex, which we think of being an occurrence of $u_0$ or $w_0$ within a long extra-tight trail, is an extra-tight path $\extratightpath{2d}{d-1}$ of length one less.
Hence, we will need to find $\extratightpath{2d}{d-1}$-decompositions of $\tilde{G}_1$ and~$\tilde{G}_2$.
Each copy of $\extratightpath{2d}{d-1}$ is the shadow of some tight path $P_{2d}^{(d)}$, which are the red and blue edges above, and we furthermore will have to make sure that they are all distinct in the decompositions.

To find the $\extratightpath{2d}{d-1}$-decompositions for the construction of switchers and the decomposition of the leftover $L$ in the Absorber Lemma into extra-tight cycles of constant length, we will use the general theorem of \cite{glock2023existence} about $F$-decompositions.
For this, we need to introduce some terminology.

\begin{definition}[\cite{glock2023existence}]\label{def: divisibility vector}
    For a $d$-graph $F$, the \emph{divisibility vector} $\Deg(F)\coleq (g_0,\dots,g_{d-1})\in\N^d$ is given by $g_i\coleq\gcd\big\{\abs{F(S)}:S\in\binom{V(F)}{i}\big\}$ for $i \in [d-1]_0$. In particular, $g_0=\abs{F}$ and we write $\Deg(F)_i\coleq g_i$. Given two $d$-graphs $F$ and $G$, $G$ is called \emph{$F$-divisible} if $\Deg(F)_i$ divides $\abs{G(S)}$ for all $i\in[d-1]_0$ and all $S\in\binom{V(G)}{i}$. An \emph{$F$-packing} in $G$ is a collection of edge-disjoint copies of $F$ in~$G$. An $F$-packing $\F$ is an \emph{$F$-decomposition} of $G$ if every edge of $G$ is contained in (exactly) one copy of $F$ in~$\F$.
\end{definition}

    It is easy to see that $G$ can only have an $F$-decomposition if $G$ is $F$-divisible. When $G$ is a dense quasirandom $d$-graph, then the simple divisibility condition is also sufficient for an $F$-decomposition. We use the following convenient notion of quasirandomness. 

\begin{definition}[\cite{Keevash}]
     A $d$-graph $G$ on $n$ vertices is called \emph{$(c,h,p)$-typical} if for any family $A$ of $(d-1)$-subsets of $V(G)$ with $\abs{A}\le h$, we have $\abs{\bigcap_{S\in A}G(S)}=(1\pm c)p^{\abs A}n$.
\end{definition}
    For us, finding an $F$-decomposition is not enough. When we build the $(E_1,E_2)$-switcher, we have to find two $\extratightpath{2d}{d-1}$-decompositions where certain $d$-sets are distinct (the red and blue edges in \Cref{fig: switcher motivation}). The following theorem enables us to do this.

\begin{theorem}[\cite{glock2023existence}, Theorem 9.6]\label{theo: design paper 2 uniformities}
Let $1/n\ll\gamma\ll c,1/h\ll p,1/f$ and $r\in[f-1]$.
Let $F$ be any $r$-graph on $f$ vertices. Suppose that $G$ is a $(c,h,p)$-typical $F$-divisible $r$-graph on $n$ vertices. Let $O$ be an $(r+1)$-graph on $V(G)$ with $\Delta(O)\le\gamma n$. Then $G$ has an $F$-decomposition $\F$ such that $\bigcup_{F\in \cF} {\binom{V(F)}{r+1}}$ and $O$ are disjoint.
\end{theorem}

We will apply the previous theorem for the $(d-1)$-uniform extra-tight path $\extratightpath{2d}{d-1}$ when building the $(E_1,E_2)$-switcher. In the proof of the Absorber Lemma, we will also apply the theorem for the $d$-uniform extra-tight cycle~$\extratightcycle{f}{d}$. To this end we compute the divisibility vector of these two graphs.

\begin{lemma}\label{lem: Deg-vector of Fcycle and Fpath}If $d\ge 2$ and $f\ge d+3$ are fixed integers, then \begin{equation*}
    \Deg\!\left(\extratightcycle{f}{d} \right)_i=\begin{cases} f\cdot d &i=0\\d^2&i=1\\1&\text{else,}\end{cases}
    \end{equation*}
    for $i\in[d-1]_0$ and 
    \begin{equation*}\Deg\!\left(\extratightpath{2d}{d-1} \right)_i=\begin{cases} d^2&i=0\\1&\text{else,}\end{cases}
    \end{equation*}
    for $i\in[d-2]_0$.
\end{lemma}
\begin{proof}
    By \Cref{def: extra-tight trail tour}, there is a 1-to-1 correspondence between the edges of $\extratightcycle{f}{d}$ and $[f]\times[d]$. Therefore, the number of edges in $\extratightcycle{f}{d}$ is $\Deg\!\big(\extratightcycle{f}{d}\big)_0=f\cdot d$. By \Cref{obs: degree in trails}, we have that every vertex in $\extratightcycle{f}{d}$ has degree $d^2$ using that $f\ge d+3$. Thus, we have $\Deg\!\big(\extratightcycle{f}{d}\big)_1=d^2$.
    Next, we will determine $\Deg\!\big(\extratightcycle{f}{d}\big)_i$ for some fixed $i\in[d-1]\backslash\{1\}$. For this, let $(v_1,\dots,v_f)$ be the vertex sequence of $\extratightcycle{f}{d}$ where we view the indices mod~$f$.
    We will compute $\big\lvert\extratightcycle{f}{d}(S)\big\rvert$ for $S \in \big\{ \{v_d,v_{d+1},\dots,v_{d+i-1}\},\{v_{d-1},v_{d+1},v_{d+2},\dots,v_{d+i-1}\} \big\}$ and show that these numbers are coprime.

    Starting with $S=\{v_d,v_{d+1},\dots,v_{d+i-1}\}$, we determine which pairs $(\iota, \sigma)\in [f]\times[d]$ correspond to an edge that contains~$S$. We need that $\iota\in\{i-1,i,\dots,d\}$ and then we have to choose $\sigma\in[d]$ such that $v_{\iota+\sigma}\not\in S$. There are $d-i$ choices for $\sigma$ unless $\iota=d$ in which case we have $d-i+1$ choices for~$\sigma$. We can conclude that $\big\lvert\extratightcycle{f}{d}(S)\big\rvert=(d-i+2)(d-i)+1=(d-i+1)^2$.

    Doing the same thing for $S=\{v_{d-1},v_{d+1},v_{d+2},\dots,v_{d+i-1}\}$, we see that we need $\iota\in\{i-1,\dots,d-1\}$. Afterwards, we again have $d-i$ choices for $\sigma$ unless $\iota=d-1$ in which case we have $d-i+1$ choices. Therefore, $\big\lvert\extratightcycle{f}{d}(S)\big\rvert=(d-i+1)(d-i)+1=(d-i)^2+(d-i)+1$.
    Because of 
    \[\gcd\!\big((d-i+1)^2,(d-i)^2+(d-i)+1\big)=\gcd\!\big((d-i+1)^2,d-i\big)=1,\]
    we can conclude $\Deg\!\big(\extratightcycle{f}{d} \big)_i=1$.

    For $\extratightpath{2d}{d-1}$, note that the edges of $\extratightpath{2d}{d-1}$ are exactly the edges of a $(d-1)$-uniform extra-tight path with vertex sequence $(v_1,\dots, v_{2d})$. By \Cref{def: extra-tight trail tour}, we have that $\extratightpath{2d}{d-1}$ has $\big(2d-(d-1)\big)\cdot (d-1)+1=d^2$ edges. Furthermore, \Cref{obs: degree in trails} implies that $v_1$ has degree $d-1$ whereas $v_2$ has degree $2d-3$. Thus, $\Deg\!\big(\extratightpath{2d}{d-1}\big)_0=d^2$ and $\Deg\!\big(\extratightpath{2d}{d-1}\big)_1=1$. For $i\in[d-2]\backslash\{1\}$, we can compute $\big\lvert\extratightpath{2d}{d-1}(S)\big\rvert$ for $S \in \big\{ \{v_d,v_{d+1},\dots,v_{d+i-1}\},\{v_{d-1},v_{d+1},v_{d+2},\dots,v_{d+i-1}\} \big\}$. Then we get the same numbers as for $\extratightcycle{f}{d}$ (but with $d$ replaced by $d-1$). Thus, $\Deg\!\big(\extratightpath{2d}{d-1}\big)_i=1$.
    \end{proof}

We are now ready to build the $(E_1,E_2)$-switcher with the help of~$\extratightpath{2d}{d-1}$.

\begin{proof}[Proof of \Cref{lem: switcher}]
We want to apply \Cref{theo: design paper 2 uniformities} twice on carefully chosen $(d-1)$-graphs $G_1$ and $G_2$ which will give us two $\extratightpath{2d}{d-1}$-decompositions $\F$ and~$\F'$.

We apply \Cref{theo: design paper 2 uniformities} with $r\coleq d-1$,  $f\coleq 2d$, $p \coleq 1$. Fix $\gamma,c,1/h,1/n_0 > 0$ small enough such that \Cref{theo: design paper 2 uniformities} holds for every $n\geq n_0$. Let $V$ be an $n$-element set such that $u_1,\dots,u_{d-1},u_d, w_d \in V$, $u_0,w_0 \notin V$ and define $\tilde G$ to be the $(d-1)$-graph on vertex set $V$ with edge set
\[\tilde G=\binom{V}{d-1}\backslash\binom{\{u_1,\dots,u_d,w_d\}}{d-1}.\] Let $\tilde G_1$ be the graph that is constructed from $\tilde G$ by adding the edges in the set $E_1(u_0)=E_2(w_0)=\big\{\{u_1,\dots,u_{i-1},u_{i+1},\dots,u_{d-1},w_d\}: i\in [d-1]\big\}$ and then deleting an edge set $E'\subs\binom{V\backslash\{u_1,\dots,u_d,w_d\}}{d-1}$ of size at most $d^2$ such that after the deletion of $E'$, the number of edges of $\tilde G_1$ is divisible by $d^2$. Similarly, $\tilde G_2$ is constructed from $\tilde G$ by adding the edges in $E_1(w_0)=E_2(u_0)=\big\{\{u_1,\dots,u_{i-1},u_{i+1},\dots,u_{d-1},u_d\}: i\in [d-1]\big\}$ and then deleting the same edge set~$E'$. Note that the number of edges of $\tilde G_2$ is also divisible by~$d^2$. \Cref{lem: Deg-vector of Fcycle and Fpath} then implies that both $\tilde G_1$ and $\tilde G_2$ are $\extratightpath{2d}{d-1}$-divisible.
    
    Clearly, the complete $(d-1)$-graph is $(c,h,p)$-typical if $n$ is large enough. Since $\tilde G_1$ and $\tilde G_2$ are both constructed from a complete $(d-1)$-graph by deleting at most $\big\lvert\binom{\{u_1,\dots,u_d,w_d,\}}{d-1}\big\lvert+\abs{E'}\le \frac{(d+1)\cdot d}{2}+d^2$ edges, they are also $(c,h,p)$-typical provided $n$ is large enough.
    
     Therefore, we can first apply \Cref{theo: design paper 2 uniformities} on $\tilde G_1$, with the $(d-1)$-uniform extra-tight path $\extratightpath{2d}{d-1}$, and the empty $d$-graph~$O$. We get an $\extratightpath{2d}{d-1}$-decomposition $\mathcal{F}=\!\big(F_i^{(d-1)}\big)_{i\in[k]}$. 
    Each $F_i^{(d-1)}$ is the shadow of some $d$-uniform tight path, that we denote by $F_i^{(d)}$. 

     Now let $O=\bigcup_{i\in[k]}F_i^{(d)}$. 
     To compute $\Delta(O)$, fix a set $S\subs V(O)$ of size $d-1$. Since $S$ is an edge in at most one copy of $\extratightpath{2d}{d-1}$ 
     in $\F$, it can be the subset of at most $\Delta\!\big(P_{2d}^{(d)}\big)=2$ edges of~$O$. Thus, $\Delta(O)\le 2\le \gamma n$ for large enough $n$. We can then apply \Cref{theo: design paper 2 uniformities}, but this time 
     on $\tilde G_2$ to get another $\extratightpath{2d}{d-1}$-decomposition $\F'=\!\big(F'^{(d-1)}_i\big)_{i\in[k]}$ such that no $d$-set of $O$ is contained in any $V\!\big(F_i^{(d-1)}\big)$. 
Each $F_i'^{(d-1)}$ is the shadow of some $d$-uniform tight path, that we denote by $F_i'^{(d)}$.
 The key property of the decomposition $\cF'$ implies that no $d$-edge of some $F_{j'}'^{(d)}$ appears as a $d$-edge in some $F_{j}^{(d)}$.
 
    We are now ready to build the switcher.
    We will first build the sequence $T_1$ and then obtain $T_2$ by switching all occurrences of $u_0$ and $w_0$.
    The sequence $T_1$ will be the concatenation of sequences of length $2d+1$ for each member of the decompositions $\cF$ and $\cF'$, together with $2d$ individual ``padding vertices'' inbetween them. Namely, if $(v_1,\dots, v_{2d})$ is the vertex sequence of a copy of $\extratightpath{2d}{d-1}$ in $\cF \cup \cF'$ then we add the vertex sequence
     \begin{equation}\label{eq: u_0 sequence}
         (v_1,\dots, v_{d}, u_0, v_{d+1}, \dots, v_{2d})
     \end{equation}
if the copy is in $\cF$ and add the vertex sequence 
\begin{equation}\label{eq: w_0 sequence}
    (v_1,\dots, v_{d}, w_0, v_{d+1}, \dots, v_{2d})
    \end{equation}
if the copy is in $\cF'$. Furthermore, between these vertex sequences, with \Cref{lem: glue two extra-tight trails} we always add $2d$ entirely new vertices that have not appeared before.
We call these vertices the \emph{unique} vertices of our sequence.
Let $T_1$ be the resulting sequence of length $2k(4d+1) -2d$. 

To see that $T_1$ indeed defines an extra-tight trail, we first note that every interval of length $2d+1$ in $T$ consists of distinct vertices as the sequences of both \eqref{eq: u_0 sequence} and \eqref{eq: w_0 sequence} come from paths and a new vertex and $2d$ unique vertices are inserted in between them.
To have an extra-tight trail we need that no $d$-set $e$ appears multiple times among $d+1$ consecutive elements of $T_1$.

If $e$ contains a unique vertex $v$, then all appearances of $e$ within $d+1$ consecutive elements of $T$ must occur within the interval $S$ of $2d+1$ consecutive elements of $T$ having $v$ at its midpoint. We have seen above that this interval consists of distinct vertices so the placement of $e$ is unique.

If $e$ does not contain any unique vertex, then each occurrence of it within $d+1$ consecutive elements of $T$ must occur within one of the subsequences of type \eqref{eq: u_0 sequence} or type \eqref{eq: w_0 sequence}. 
If furthermore $u_0 \in e$, then $e\setminus \{u_0\}$ is a $(d-1)$-edge in $\tilde G_1$ which appears in exactly one of the copies $F_j^{(d-1)}$ of $\extratightpath{2d}{d-1}$ in the decomposition $\F$, hence the placement of $e$ is also unique. 
If instead $w_0 \in e$ then $e\setminus \{w_0\}$ is a $(d-1)$-edge in $\tilde G_2$ which appears in exactly one of the copies $F_j'^{(d-1)}$ of $\extratightpath{2d}{d-1}$ in the decomposition $\F'$, hence the placement of $e$ is again unique.
Finally, if $u_0,w_0 \not\in e$, then $e$ is contained in either $F_j^{(d)}$ or $F'^{(d)}_j$ for some unique $j\in [k]$. Here the $F_i^{(d)}$ are all disjoint from each other since the $(d-1)$-subsets of their members are distinct due to $\cF$ being a decomposition. Analogously, the $F'^{(d)}_j$ are also all disjoint from each other since $\cF'$ is a decomposition. Finally, any $F_j^{(d)}$ and $F'^{(d)}_{j'}$ are also disjoint as in the creation of $\cF'$, we required that no $d$-subset of the vertex sets $V\!\big(F_j^{(d)}\big)$ is contained in any of the vertex sets $V\!\big(F_{j'}'^{(d)}\big)$. We can thus conclude that the sequence $T$ indeed induces an extra-tight trail. 

Next, we prove that our extra-tight trail $T_1$ covers all edges of $E_1$. For any $j\in[d-1]$, consider the $d$-set $e=\{u_0,u_1,\dots,u_{j-1},u_{j+1},\dots,u_{d-1},w_d\}\in E_1$. Note that $e\backslash\{u_0\}$ is in $\tilde G_1$. As the link of $u_0$ in the extra-tight path of the sequence \eqref{eq: u_0 sequence} is exactly $F_i^{(d-1)}$ and since the $F_i^{(d-1)}$ form a decomposition of $\tilde G_1$, we can conclude that the link of $u_0$ in our extra-tight trail is exactly $\tilde G_1$. Hence, $e$ is covered by the extra-tight trail. An analogous argument shows that the link of $w_0$ in our extra-tight trail is exactly $\tilde G_2$ and that the $d$-set $\{w_0,u_1,\dots,u_{j-1},u_{j+1},\dots,u_{d-1},u_d\}$ of $E_1$ is covered by the extra-tight path of a sequence of \eqref{eq: w_0 sequence}. Hence, the full $E_1$ is covered.

Next, we will show that $V(E_1\cup E_2)$ is an independent set in $T_1 \setminus E_1$.
Let $e\in\binom{\{u_0,\dots,u_d,w_0,w_d\}}{d}$.
If $u_0\in e$, we can use that the link of $u_0$ in $T_1$ is exactly $\tilde G_1$.
By definition of $\tilde G_1$, we have $\tilde G \cap \binom{\{u_1,\dots,u_d,w_0,w_d\}}{d}=E_1$. Hence, $e\not\in T_1 \setminus E_1$.
An analogous argument works if $w_0\in e$.
Thus, we we can assume that $e\in\binom{\{u_1,\dots,u_d,w_d\}}d$.
We will show that $e$ cannot be covered by the extra-tight trail. Suppose it is. Then it must already be covered by one of the sequences of \eqref{eq: u_0 sequence} or \eqref{eq: w_0 sequence}. Without loss of generality, assume it is a sequence of \eqref{eq: u_0 sequence}. But then every $(d-1)$-subset of $e$ is in the link of $u_0$. However, the link of $u_0$ in $T_1$ is $\tilde G_1$ which only contains $d-1$ edges of $\binom{\{u_0,\dots,u_d,w_0,w_d\}}{d-1}$. As the number of $(d-1)$-subsets of $e$ is $d$, this is a contradiction.

Finally, to obtain the extra-tight trail $T_2$ with the same ends as $T_1$ we use the exact same vertex sequence but swap all~$u_0$ and~$w_0$.
Then, as we only swapped the role of two vertices, the vertex-sequence indeed still defines an extra-tight trail.
To see that $T_2$ now covers exactly $(T_1 \setminus E_1)\cup E_2$, note that only the edges containing $u_0$ or $w_0$ are affected when $u_0$ and $w_0$ are swapped. Before the swap, the link of $u_0$ was $\tilde G_1$, whereas it is $\tilde G_2$ after the swap. For $w_0$, it is the other way around. Since, by construction, $\tilde G_1\backslash \tilde G_2=E_1(u_0)=E_2(w_0)$ and $\tilde G_2\backslash \tilde G_1=E_1(w_0)=E_2(u_0)$, the swap only causes that $E_1$ is no longer covered, whereas $E_2$ is. As neither $u_0$ nor $w_0$ appears at an end of the extra-tight trail, the ends of the extra-tight trail are the same before and after the swap.
\end{proof}

\subsection{Proof of the Absorber Lemma}
Recall that the plan is to construct the extra-tight trail $A$ such that for each extra-tight cycle $C$ in the left-over $L \subseteq G[U]$, $A$ can be rewired to path through $C$, while simultaneously containing an appropriate switcher that, when activated, reverts all the effects of the rewiring.
Hence, for {\em any} sequence $S=(u_0,u_1,\dots,u_d)$ potentially contained in $C$, we have to include the sequence $D_S=(w_0,u_1,\dots,u_{d-1},w_d)$ for some $w_0,w_d \in V(G) \setminus U$ as well as the sequence $T_1$ from an $(E_1,E_2)$-switcher 
$(T_1,T_2)$ in our absorption trail $A$.
That way, if some $S$ is used to absorb a cycle $C$, then at $D_S$ we divert the sequence of $A$ through $C$ and swap the corresponding subsequence $T_1$ of $A$ to $T_2$, with only the edges of $C$ covered additionally.
To make it possible to {\em find} such an extra-tight trail $A$ in $G$, it is crucial that the number of potential attachment tuples and the length of the trails $T_1$ are a function of $m$ and $d$, while $1/n\ll 1/m, 1/d$.

	\begin{proof}[Proof of \Cref{lem: Absorber Lemma}]
			Let $1/n\ll\eta\ll 1/m\ll \alpha\ll 1/d\le 1/2$.
			
            Let $\mathcal S$ be the set of all ordered tuples $S=(u_0,u_1,\dots,u_{d-1},u_d)$ of $d+1$ distinct vertices of $U$, such that all edges of the corresponding extra-tight path are in~$G$.
        For each $S=(u_0,u_1,\dots,u_d)\in\mathcal S$, we want to find two vertices $w_0^S$ and $w_d^S$ in $V(G)\backslash U$ and an embedding of $T^S_1\cup T^S_2$ in $G$ where $\big(T^S_1,T_2^S\big)$ is the $\big(E^S_1,E^S_2\big)$-switcher given by \Cref{lem: switcher} with
			\begin{align*}
				E^S_1&=\Big\{\{u_0,u_1,\dots,u_{i-1},u_{i+1},\dots, u_{d-1},w^S_d\}, \{w^S_0,u_1,\dots,u_{i-1},u_{i+1},\dots, u_{d-1},u_d\}: i\in [d-1]\Big\}\\
				E^S_2&=\Big\{\{w^S_0,u_1,\dots,u_{i-1},u_{i+1},\dots, u_{d-1},w^S_d\}, \{u_0,u_1,\dots,u_{i-1},u_{i+1},\dots, u_{d-1},u_d\}: i\in [d-1]\Big\}.
			\end{align*}
			Furthermore, we want to make sure that $V\big(T_1^S\cup T_2^S\big)\cap U=S$ and $V\big(T_1^S\cup T_2^S\big)\cap V\big(T_1^{S'}\cup T_2^{S'}\big)=S\cap S'$ for any other tuple $S'\in\mathcal S$ where $S\cap S'$ is the set of all vertices that appear in both tuples $S$ and $S'$.
			
			Embedding all $T_1^S\cup T_2^S$ in $G$ can be done greedily, by going through all $S\in\mathcal S$ one by one and embedding $T_1^S\cup T_2^S$ one vertex at a time, starting with the vertices of~$S$ which are already in $G$. Every time we embed the next vertex of $T_1^S\cup T_2^S$ in $V\backslash U$, we just have to make sure that all the necessary edges are present in $G$ and that the vertex chosen does not appear in the embedding of any of the other~$T_1^{S'}\cup T_2^{S'}$. As $\abs{V\big(T_1^S\cup T_2^S\big)}$ only depends on $d$ and $\abs{\mathcal S}$ only depends on $m$ and $d$, we can always find a suitable next vertex by the minimum degree condition $\delta (G) \geq (1-\alpha ) n$, provided $\alpha<\binom{\big\lvert V\big(T_1^S\cup T_2^S\big)\big\rvert}{d-1}^{-1}$ and $n$ is chosen large enough.
			
			Let $D_S$ be the extra-tight trail $\big(w_0^S,u_1,\dots,u_{d-1},w_d^S\big)$. 
			We connect all the $D_S$ and $T_1^S$ to one long vertex sequence by iteratively inserting $2d$ entirely new gluing vertices of $V\backslash U$ between them using \Cref{lem: glue two extra-tight trails}. This, and that all new vertices can be chosen distinct, is again possible by the minimum degree condition $\delta (G) \geq (1-\alpha)n$. Finally, we similarly add $d$ entirely new vertices $w_1,\dots, w_d$ of $V\backslash U$ at the beginning and $d$ entirely new vertices $w_d',\dots, w_1'$ of $V\backslash U$ at the end. Let $A$ be the resulting sequence and note that we have made sure that every edge of $A$ does appear in $G$. 
			
			We still have to see however that $A$ induces an extra-tight trail, that is no $d$-subset $e$ appears multiple times among $d+1$ consecutive elements of $A$. If $e$ contains one of the gluing vertices, then $e$ does not appear multiple times by the same reason as in the previous lemma. If $e$ does not contain a gluing vertex, then it is an edge in the extra-tight trail induced by $D_S$ or $T_1^S$ for some $S\in \cS$. First, we argue why $e\not\in T_1^{S'}\cup D_{S'}$ for any $S'\in \cS\setminus\{S\}$. This is because by construction $V(D_S)\cap V(D_{S'})$, $V\big(T_1^S\big)\cap V\big(D_{S'}\big)$, $V(D_S)\cap V\big(T_1^{S'}\big)$, $V\big(T_1^S\big)\cap V\big(T_1^{S'}\big)$ are all contained in $S\cap S'\subs U$ and no edge of $D_S$ and $T_1^{S}$ is entirely in $U$. Thus, we only have to check that $e$ does not appear in both $D_S$ and $T_1^S$. Here we use that, by the definition of a switcher, $V(D_S)\subs V\big(E_1^S\cup E_2^S\big)$ is an independent set in $T_1^S\backslash E_1^S$. Furthermore, $D_S\cap E_1^S=\emptyset$. Thus, we can conclude that $A$ induces an extra-tight trail.
			
			Next, we check that $A$ has all the desired properties stated in the Absorber Lemma. Note that by the same argument as above, $A$ does not contain any edge that lies entirely in~$U$. Furthermore, we can make sure that $\Delta(A)\le\eta n$, as even the size of $A$ depends only on $d$ and $m$. 
			
			To check the final property, let $L$ by any $d$-graph on $U$ with $\delta(L)\ge (1-\alpha)n$ where each 1-degree is divisible by~$d^2$. We have to show that there is an extra-tight trail with edge set $L\cup A$ whose ends are the same as that of $A$. 
			
			We first show this in the case when $L$ can be partitioned into extra-tight cycles of length at least~$d+3$ by absorbing the edge set of each of them into $A$. Let $C$ be one of these extra-tight cycles and let $u_0,\dots,u_d$ be $d+1$ consecutive vertices of $C$ in that order. In other words, $C$ is defined by some cyclic sequence $u_0,\dots,u_d,c_1,\dots,c_k$. Let $S=(u_0,\dots,u_d)$. By construction, $A$ contains the vertex sequence $D_S$, i.e.\@ $w_0^S,u_1,\dots,u_{d-1},w_d^S$. We replace this in $A$ by $w_0^S,u_1,\dots,u_{d-1},u_d,c_1,\dots,c_k,u_0,u_1,\dots,u_{d-1},w_d^{S}$
			and remove $C$ from $L$. 

			By this change, the edges of $E^S_2$ are no longer covered in $A\cup L$, whereas the edges of $E^S_1$ are now covered twice since they were already covered by $T_1^S$. But this is exactly what the $(E^S_1,E^S_2)$-switcher $(T_1^S,T_2^S)$ was created for: by replacing the vertex sequence $T_1^S$ in $A$ by $T_2^S$, we can ensure that the same edges are covered as before. Therefore, this procedure can iteratively absorb each extra-tight cycle of the decomposition of $L$, until we end up with an extra-tight trail covering exactly the edges $A\cup L$ while still having the same ends $(w_1,\dots,w_d)$ and $(w_d',\dots,w_1')$ as $A$.
			
			It remains to show that $L$ can indeed always be decomposed into extra-tight cycles. To use \Cref{theo: design paper 2 uniformities}, we have to make sure that $L$ is $\extratightcycle{f}{d}$-divisible for some fixed $f$. For convenience, here we will aim for $f=3d$. By \Cref{lem: Deg-vector of Fcycle and Fpath} and by the fact that all degrees of $L$ are divisible by $d^2$, we only have to ensure that the number of edges is divisible by~$3d^2$. This is not necessarily the case to begin with, but
			since all degrees of $L$ are divisible by $d^2$, the $d$-uniform Handshake Lemma at least implies that $|L|$ is at divisible by~$d$. Hence, there is an integer $k\in\{0,\dots,3d-1\}$ such that $|L| \equiv kd\pmod{3d^2}$.
			
			Our strategy of partitioning $L$ into extra-tight cycles then starts by finding a copy of $\extratightcycle{3d+k}{d}$ in~$L$. Removing it leaves us with a graph $L'$ satisfying $|L'| \equiv 0\pmod{3d^2}$ as 
			$|EC_{3d+k}^{(d)}| = (3d+k)d$ by \Cref{obs: degree in trails}.
			Hence $L'$ is $\extratightcycle{3d}{d}$-divisible, because deleting a tight cycle does not change the divisibility of the vertex-degrees by $d^2$.
			
			The extra-tight cycle of length $3d+k$ can be found in $L$ vertex by vertex: when choosing the next vertex, one only has to ensure that all the at most $d^2$ edges that are formed with the already chosen vertices in the extra-tight cycle are present in $L$. Each edge can exclude at most $\alpha m$ vertices. Thus, if $\alpha <\frac{1}{d^2}$ and $m$ is large enough, then there is always at least one vertex in $U$ that can be chosen next.
			
			After we removed the extra-tight cycle of length $3d+k$, the remaining graph $L'$ still satisfies $\delta(L')\ge (1-\alpha)m - 4 \geq (1-2\alpha)m$. As $L'$ is $\extratightcycle{3d}{d}$-divisible we apply \Cref{theo: design paper 2 uniformities} with 
			$O\coleq\emptyset$ to find the desired 
			$\extratightcycle{3d}{d}$-decomposition of the graph $L'$.
			
			To conclude, we recapitulate the order in which the parameters are chosen for the entire proof to work. Once $d$ is fixed, the size of each switcher $\big(T^S_1,T^S_2\big)$ is determined. 
			\Cref{theo: design paper 2 uniformities} with parameters $f\coleq 3d$, $r\coleq d$, $p\coleq 1$ defines parameters $c$, $h$, $\gamma$ and $n_0$ such that the statement of \Cref{theo: design paper 2 uniformities} holds for all
			$n_{\ref{theo: design paper 2 uniformities}}\ge n_0$. 
			Afterwards, we fix $\alpha$ small enough such that we can greedily build all switchers in $G$ ($\alpha<\binom{\abs{V(T^S_1\cup T_2^S)}}{d-1}^{-1}$), find all gluing vertices ($\alpha<d^{-2}$) and make sure that each graph $L'$ on $m$ vertices with $\delta(L')\ge (1-2\alpha)m$ is $(c,h,1)$-typical ($\alpha\le \frac{c}{2h}$) so that we can actually apply \Cref{theo: design paper 2 uniformities} for $L'$ with $\extratightcycle{3d}{d}$. Next, we choose $m$ such that it is at least $n_0$ and that we can build the extra-tight cycle of length $k\le 3d$ greedily in any graph $L$ on $m$ vertices with minimum degree $\delta(L)\ge (1-\alpha)m$. Finally, we pick any positive number for $\eta$ and choose $n$ large enough such that all switchers and gluing vertices can be found greedily in any graph $G$ on $n$ vertices with $\delta(G)\ge (1-\alpha)n$ and such that the maximum degree of the final absorber, which only depends on $d$ and $m$, is at most~$\eta n$.
		\end{proof}
		
		\section{Rooted Embeddings}\label{sec: Rooted Embeddings}
		At several places, we will have to find embeddings of a fixed graph $T$ into $G$ where the images of several vertices of $T$ are already determined.
		For example, we already used that in the proof of \Cref{theo: large min degree implies extra-tight trail} when we greedily found $d$ vertices such that they together with some other already determined vertices form an extra-tight trail.
		These embeddings are called ``rooted embeddings'' and their existence has already been studied in \cite{glock2023existence}. For completeness, we repeat their definitions here.
		
		Let $T$ be a $d$-graph and $X\subs V(T)$. A \emph{root} of $(T,X)$ is a set $S\subs X$ such that $\abs{S}\in[d-1]$ and $\abs{T(S)}>0$. A \emph{$G$-labelling} of $(T,X)$ is an injective map $\Lambda\colon X\to V(G)$ and, given a $G$-labelling $\Lambda$, we say that an injective homomorphism $\phi\colon T\to G$ is a \emph{$\Lambda$-faithful embedding} of $(T,X)$ into $G$ if $\phi\vert_X=\Lambda$. Furthermore, a $G$-labelling $\Lambda$ of $(T,X)$ \emph{roots} a vertex set $S\subs V(G)$ if $S\subs \Ima(\Lambda)$ and $\abs{T\big(\Lambda^{-1}(S)\big)}>0$. Finally, the \emph{degeneracy of $T$ rooted at $X$} is the smallest $D$ such that there is an ordering $v_1,\dots,v_k$ of the elements of $V(T)\backslash X$ such that for every $\ell\in [k]$, we have $\abs{T[X\cup\{v_1,\dots,v_\ell\}](v_\ell)}\le D.$
		
		We will need to find several rooted embeddings simultaneously in the same graph $G$ such that the images are edge-disjoint. The following does that and is a slight variation of~\cite[Lemma 5.20]{glock2023existence}.
		On one hand, it is less general since it removes restrictions encoded by a $d+1$-graph $O$ and the concept of a \emph{hull} of~$(T,X)$, which strengthens our condition~\ref{enum: rooted embedding with U disjoint} below. On the other hand, it has the additional requirement that all vertices $v\in V(T)\backslash X$ are embedded in a special vertex set $U\subs V(G)$.
		
		\begin{lemma}\label{lem: rooted embedding with U}
			Let $1/n\ll \gamma \ll \gamma'\ll \xi,1/t,1/D$ and $d\in [t]$\towriteornottowrite{. Suppose that $\alpha\in(0,1]$ is an arbitrary scalar (which might depend on $n$) and let $m\le \alpha\gamma n^d$ be an integer.}{\ and $m\le\gamma n^d$ be an integer.} For every $j\in[m]$, let $T_j$ be a $d$-graph on at most $t$ vertices and $X_j\subs V(T_j)$ such that $T_j[X_j]$ is empty and $T_j$ has degeneracy at most $D$ rooted at~$X_j$. Let $G$ be a $d$-graph on $n$ vertices and $U\subs V(G)$ such that for all $A\subs\binom{V(G)}{d-1}$ with $\abs A\le D$, we have $\abs{\bigcap_{S\in A}G(S)\cap U}\ge \xi n$.
			For every $j\in[m]$, let $\Lambda_j$ be a $G$-labelling of $(T_j,X_j)$. Suppose that for all $S\subs V(G)$ with $\abs S\in[d-1]$, we have that
			\begin{align}
				\abs{\{j\in[m]:\Lambda_j \mbox{ roots }S\}}\le\towriteornottowrite{\alpha}{}\gamma n^{d-\abs S}-1.\label{eq: few roots with U}
			\end{align}
			Then for every $j\in [m]$, there exists a $\Lambda_j$-faithful embedding $\phi_j$ of $(T_j,X_j)$ into $G$ such that the following hold:
			\begin{enumerate}
				\item for all distinct $j,j'\in [m]$, the $d$-graphs $\phi_j(T_j)$ and $\phi_{j'}(T_{j'})$ are edge-disjoint;\label{enum: rooted embedding with U disjoint}
				\item $\Delta(\bigcup_{j\in[m]}\phi_j(T_j))\le\towriteornottowrite{\alpha}{}\gamma' n$;\label{enum: rooted embedding with U maxdeg}
				\item for all $j\in[m]$ and all  $v\in V(T_j)\backslash X_j$ we have $\phi_j(v)\in U$.\label{enum: rooted embedding only uses U}
			\end{enumerate}
		\end{lemma}
		
		The proof of this lemma is almost identical to the proof of \cite[Lemma 5.20]{glock2023existence}%
		\towriteornottowrite{, hence, we postpone the writeup to the appendix.}{, hence, we omit it here.}
		We will mainly use this lemma to connect a given set of extra-tight trails into one large extra-tight trail. This is given by the following lemma:
		\begin{lemma}\label{lem: path connecting}
			Let $1/n\ll \gamma\ll\gamma'\ll  \xi,1/D\ll 1/d$. Suppose that $G$ is a $d$-graph on $n$ vertices and let $U\subs V(G)$ such that for all $A\subs\binom{V(G)}{d-1}$ with $\abs{A}\le D$, we have $\abs{\bigcap_{S\in A} G(S)\cap U}\ge \xi n$. Let $\cP$ be a collection of pairwise edge-disjoint extra-tight paths such that each $S\subs V(G)$ of size $d-1$ appears at an end of at most $\gamma n$ of the paths in $\cP$ and $\bigcup_{P\in \cP} P\cap G=\emptyset$.
			
			Then there is a subset $H\subs G$ such that $H\cup\bigcup_{P\in\cP}P$ is the edge set of an extra-tight trail whose first and last $d$ vertices all lie in $U$ and $\Delta(H)\le \gamma'n$.
		\end{lemma}
		\begin{proof}
			Let $\cP=\{P_1,\dots, P_{m}\}$ and $\big(u^{(j)}_1,\dots,u^{(j)}_d\big)$ resp.\@ $\big(w^{(j)}_d,\dots,w^{(j)}_1\big)$ be the first resp.\@ last $d$ vertices of $P_j$ where $j\in[m]$. Since each $S\subs V(G)$ of size $d-1$ appears at an end of at most $\gamma n$ of the paths in $\cP$, we have $m=\abs{\cP}\le \gamma n^d$.
			Further, let $P_0=\big(u^{(0)}_1,\dots,u^{(0)}_d\big)$ and $P_{m+1}=\big(w^{(m+1)}_d,\dots,w^{(m+1)}_1\big)$ be two extra-tight paths given by two arbitrary edges from $G[U]$.
			In the end, these vertices will form the ends of our extra-tight trail.
			
			Let $T$ be the extra-tight path with vertex sequence $(w_d,w_{d-1},\dots,w_1,v_1,\dots,v_d,u_1,\dots,u_d)$ but which only contains those $d$-edges that have non-empty intersection with $\{v_1,\dots,v_d\}$. The idea is to embed a copy of $T$ into $G$ for each $j\in[m]_0$ such that $w_i$ is mapped to $w_i^{(j)}$ and $u_i$ is mapped to $u_i^{(j+1)}$ for each $i\in[d]$.
			Then $P_j$, the embedded copy of $T$, and $P_{j+1}$ form one long extra-tight trail.
			
			To make this more precise, we apply \Cref{lem: rooted embedding with U}\towriteornottowrite{\ with $\alpha_{\ref{lem: rooted embedding with U}}\coleq 1$}{}. For $j\in [m]_0$, define $T_j\coleq T$ and $X_j\coleq\{w_d,\dots,w_1,u_1,\dots,u_d\}$. We choose $D$ depending on $d$ large enough such that $T_j$ has degeneracy at most $D$ rooted at~$X_j$. By assumption, $G_{\ref{lem: rooted embedding with U}}\coleq G-\big\{u^{(0)}_1,\dots,u^{(0)}_d\big\}-\big\{w^{(m+1)}_1,\dots,w^{(m+1)}_d\big\}$ satisfies the requirements for \Cref{lem: rooted embedding with U} with $\xi_{\ref{lem: rooted embedding with U}}\coleq \xi/2$. For $j\in [m]_0$, let $\Lambda_j$ be the $G_{\ref{lem: rooted embedding with U}}$-labelling of $(T_j,X_j)$ that maps $w_i$ to $w_i^{(j)}$ and $u_i$ to $u_i^{(j+1)}$ for each $i\in[d]$.
			
			To check the last condition of \Cref{lem: rooted embedding with U}, let $S\subs V(G)$ with $\abs S\in[d-1]$. Since the extra-tight trail $T$ has $d$ vertices between $w_1$ and $u_1$, the only $j\in[m]$ such that $\Lambda_j$ roots $S$ are those where $S\subs \big\{w^{(j)}_d,\dots,w^{(j)}_1\big\}$ or $S\subs \big\{u^{(j+1)}_1,\dots,u^{(j+1)}_d\big\}$. This happens for at most $2\gamma n^{d-\abs S}$ of the $j\in[m]$. Hence, we can apply \Cref{lem: rooted embedding with U} with $\gamma_{\ref{lem: rooted embedding with U}}\coleq 3\gamma$.
			We get edge-disjoint, $\Lambda_j$-faithful embeddings $\phi_j$ of $(T_j,X_j)$ such that $\Delta(\bigcup_{j\in[m]_0}\phi_j(T_j))\le \gamma'n$. By construction, $H\coleq\bigcup_{j\in[m]_0}\phi_j(T_j)\cup \big\{\{u^{(0)}_1,\dots,u^{(0)}_d\},\{w^{(m+1)}_1,\dots,w^{(m+1)}_d\}\big\}$ has the desired properties.
		\end{proof}
		
		\section{Approximate Decomposition Lemma}\label{sec: Approximate}
		The goal of this section is to prove the Approximate Decomposition \Cref{lem: approximate decomposition} which we will need in the proof of the Cover Down Lemma.
		We employ a technique that was recently used by Gishboliner, Glock, and Sgueglia~\cite{gishboliner2023tight} in a different setting, namely in a work on tight Hamilton cycles. 
		Roughly speaking, the idea is that if one wants to find a long ``path--like'' structure, a natural approach is to use a random walk argument which extends the structure via randomized steps. However, if the structure one needs to find is very long, the random procedure needs guidance to avoid forbidden self-intersections. This can make the analysis quite intricate. Instead, the approach used in~\cite{gishboliner2023tight} is to sample many constant-length paths using the same simple random walk distribution, where self-intersection is not an issue. The obtained paths then form an auxiliary hypergraph in which one can find an almost-perfect matching using standard nibble results. Depending on the specific setup, the matching condition corresponds then to avoiding forbidden intersections between the different paths. Of course, the result of this is not one long structure, but rather many disjoint pieces. So in the final step, one needs to stitch these individual pieces together. 
		
		Since we look at extra-tight trails instead of tight Hamilton cycles, our proof of the Approximate Decomposition Lemma has some differences: we want to cover almost all edges instead of almost all vertices, which is why our auxiliary hypergraph will be defined differently, and we have to analyze the probability that an edge is covered by a randomly sampled constant-length path. Furthermore, the final step of stitching the individual pieces together is a bit different, because we also need to make sure that no edge appears too often at an end of these paths. Also, since we consider extra-tight paths, we cannot use perfect fractional matchings to define the suitable probability for the random walk, but we have to use fractional $K_{d+1}$-decompositions. This is what we define first.  
		
		\newcommand{\myx}{\textbf{x}}
		\begin{definition}
			For a $d$-graph $G$, let $K_{d+1}(G)$ be the set of all $(d+1)$-subsets $S$ of $V(G)$ such that $G[S]$ is a complete $d$-graph.
			A \emph{fractional $K_{d+1}$-decomposition} of $G$ is a function $\myx\colon K_{d+1}(G)\to [0,1]$ such that for every $e\in G$, we have \[\sum_{\substack{S\in K_{d+1}(G)\\e\subs S}}\myx(S)=1.\] For $\mu\in(0,1]$, a fractional $K_{d+1}$-decomposition $\myx$ is \emph{$\mu$-normal} if $\mu n^{-1}\le \myx(S)\le\mu^{-1}n^{-1}$ for all $S\in K_{d+1}(G)$.
		\end{definition}
		
		We need make sure that we have a sufficiently normal fractional $K_{d+1}$-decomposition in our graph~$G$: 
		
		\begin{lemma}\label{lem: there is a mu-normal fractional decomposition}
			Let $1/n\ll \alpha \ll 1/d\le 1/2$. Let $G$ be an $n$-vertex $d$-graph with $\delta(G)\ge (1-\alpha)n$. Then $G$ has a $0.9$-normal fractional $K_{d+1}$-decomposition.
		\end{lemma}
		In \cite[Theorem 1.5]{BARBER2017148}, it is shown that $G$ indeed has a fractional $K_{d+1}$-decomposition because $\alpha\ll 1/d$. Unfortunately, the theorem does not make any comments on the $\mu$-normality of the decomposition. However, a closer examination of the proof of \cite[Theorem 1.5]{BARBER2017148} reveals that their decomposition is in fact $\mu$-normal if $\alpha$ is small enough. Since the proof is verbatim the same, we only sketch it.
		
		\proof
		The proof of \cite[Theorem 1.5]{BARBER2017148} works as follows: First, each $K_{d+1}$ is assigned the weight $\omega\coleq 1/\kappa$ where $\kappa\coleq\sum_{e\in G}\kappa_e^{(d+1)}/|G|$ and $\kappa_e^{(d+1)}\in [(1-d\alpha)n,n]$ is the number of copies of $K^{(d)}_{d+1}$ that contain~$e$. Hence, at the beginning each $K_{d+1}$ is assigned the weight $1/\kappa\in [n^{-1}, \frac{1}{1-d\alpha}n^{-1}]$. 
		
		However, this is not necessarily a fractional $K_{d+1}$-decomposition.
		Therefore, the authors change the weights slightly to obtain a fractional $K_{d+1}$-decomposition.
		They show that the weight of each copy of $K^{(d)}_{d+1}$ is changed by at most $\frac{2^{d+2}d^2(d+1)^{2d-1}\alpha\omega}{d!}\sum_{j=0}^d\frac{1}{2^jj!}$.
		Hence, for fixed $d$ and small enough $\alpha$, we have a $0.9$-normal fractional $K_{d+1}$-decomposition.
		\endproof

		Now that we have established that a fractional $K_{d+1}$-decomposition $\myx\colon K_{d+1}(G)\to [0,1]$ exists, we will describe how this can be used to define a probability distribution.
		First, we extend $\myx$ to a function $\binom{V(G)}{d+1}\to [0,1]$ by setting $\myx(S)\coleq0$ for all $S\in\binom{V(G)}{d+1}\backslash K_{d+1}(G)$. We define the following random walk $\Y=(Y_1,Y_2,\dots)$ in~$V(G)$. Let $(V)_d$ be the set of all ordered $d$-tuples of $V(G)$ whose unordered set forms an edge in~$G$. The first $d$ vertices $(Y_1,\dots,Y_d)$ are chosen uniformly at random among~$(V)_d$.
		
		\newcommand{\arr}[1]{\overrightarrow{#1}}
		For $i\ge d$, let $\arr{Z_i}\coleq(Y_{i-(d-1)},\dots, Y_i)$ be the ordered set of the last $d$ vertices of the walk and let $Z_i\coleq\{Y_{i-(d-1)},\dots,Y_i\}$ be the corresponding set. For any $v\in V(G)\backslash Z_i$, we define the transition probability as follows:
		\begin{equation}
			\Pr[Y_{i+1}=v|Y_{i-(d-1)},\dots, Y_i]=\myx(Z_i\cup\{v\}).
		\end{equation}
		
		For $v\in Z_i$, the probability $\Pr[Y_{i+1}=v|Y_{i-(d-1)},\dots, Y_i]$ is set to 0. Note that this indeed defines a probability distribution because
		\[\sum_{v\in V(G)\backslash Z_i}\myx(Z_i\cup\{v\})=\sum_{\substack{S\in K_{d+1}(G)\\Z_i\subs S}}\myx(S)=1.\]
		Here, we use the fact that $Z_i$ is always an edge of~$G$. Note that $\Y$ is equivalent to the random walk $\calZ= \big(\arr{ Z_{d}}, \arr{Z_{d+1}},\dots\big)$. Here, $\calZ$ is a Markov chain with state space~$(V)_d$. Furthermore, if we define $\pi\big(\arr S\big)\coleq \frac{1}{\abs{(V)_d}}$ for all $\arr S\in (V)_d$, then $\pi$ defines a stationary distribution of the Markov chain. Indeed, let $P\big(\arr S, \arr T\big)$ be the transition probability of $\arr S$ to $\arr T$ for any $\arr S, \arr T\in (V)_d$. For any fixed $\arr T=(t_1,\dots, t_d)\in (V)_d$, we then get
		\begin{align*}\sum_{\arr S\in (V)_d}\pi\Big(\arr S\Big)P\Big(\arr S,\arr T\Big)&=\frac{1}{\abs{(V)_d}}\sum_{\arr S\in(V)_d}P\Big(\arr S,\arr T\Big)\\&
			=\frac{1}{\abs{(V)_d}}\sum_{v\in V(G)\backslash\{t_1,\dots, t_{d-1}\}} P\big((v,t_1,\dots,t_{d-1}),(t_1,\dots,t_d)\big)\\&=\frac{1}{\abs{(V)_d}}\sum_{v\in V(G)\backslash\{t_1,\dots, t_{d-1}\}} \myx(\{v,t_1,\dots,t_d\}) =\frac{1}{\abs{(V)_d}}\sum_{\substack{S\in K_{d+1}(G)\\ \{t_1,\dots,t_d\}\subs S}}\myx(S)=\frac{1}{\abs{(V)_d}}\\&=\pi(\arr{T}).
		\end{align*}
		
		The following lemma shows that if $j_1,\dots,j_d$ are $d$ out of $d+1$ consecutive numbers, then the probability that $(Y_{j_1},\dots,Y_{j_d})$ forms a fixed element of $(V)_d$ is $\frac{1}{\abs{(V)_d}}$.
		
		\begin{lemma}\label{lem: Edge appearance uniform}
			Let $i\ge 1$ be fixed and let $i= j_1<j_2<\dots<j_d\le i+d$. Fix an $\arr S=(s_1,\dots,s_d)\in (V)_d$. Then \[\Pr[Y_{j_1}=s_1,\dots, Y_{j_d}=s_d] =\frac{1}{\abs{(V)_d}}.\]
		\end{lemma}
		\begin{proof}
			If $\{j_1,\dots,j_d\}=\{i,i+1,\dots,i+d-1\}=Z_{i+d-1}$, then $\Pr[Y_{j_1}=s_1,\dots, Y_{j_d}=s_d]=\Pr\!\big[\arr{Z_{i+d-1}}=\arr S\big]=\frac{1}{\abs{(V)_d}}$ where the last equality comes from the fact that the uniform distribution is stationary for the Markov chain $\big(\arr{Z_d},\arr{Z_{d+1}},\dots\big)$.
			
			Hence, we can assume that $j_1=i,\dots, j_{\ell}=i+\ell-1, j_{\ell+1}=i+\ell+1, \dots, j_{d}=i+d$ for some $1\le \ell\le d-1$. Let $E$ be the event that $Y_{j_1}=s_1,\dots, Y_{j_d}=s_d$. Then
			\begin{equation*}
				\begin{adjustbox}{max width=\linewidth}
					$\begin{aligned}
						\Pr[E]&=\sum_{v\in V(G)}\Pr\!\left[E|\arr{Z_{i+d-1}}=(s_1,\dots,s_\ell,v,s_{\ell+1},\dots,s_{d-1})\right]\cdot 
						\Pr\!\left[\arr{Z_{i+d-1}}=(s_1,\dots,s_\ell,v,s_{\ell+1},\dots,s_{d-1})\right]\\
						&=\sum_{\substack{v\in V(G)\\\{s_1,\dots,s_d,v\}\in G}} 
						\Pr\!\left[E|\arr{Z_{i+d-1}}=(s_1,\dots,s_\ell,v,s_{\ell+1},\dots,s_{d-1})\right]\cdot 
						\Pr\!\left[\arr{Z_{i+d-1}}=(s_1,\dots,s_\ell,v,s_{\ell+1},\dots,s_{d-1})\right]\\
						&=\frac{1}{\abs{(V)_d}}\sum_{\substack{v\in V(G)\\\{s_1,\dots,s_d,v\}\in G}} 
						\Pr\!\left[E|\arr{Z_{i+d-1}}=(s_1,\dots,s_\ell,v,s_{\ell+1},\dots,s_{d-1})\right]\\
						&=\frac{1}{\abs{(V)_d}}\sum_{\substack{v\in V(G)\\\{s_1,\dots,s_d,v\}\in G}}
						\myx(\{s_1,\dots,s_d,v\})=\frac{1}{\abs{(V)_d}}.
					\end{aligned}$
				\end{adjustbox}
			\end{equation*}
			
			\vspace{-1cm}
		\end{proof}
		\vspace{.5cm}
		
		With the random walk, we are now able to find the desired probability distribution, which we will use to sample the extra-tight paths later. It is similar to Lemma 5.3 in \cite{gishboliner2023tight}, but here we analyze the probability of an edge appearing in a path instead of a vertex appearing. Furthermore, we bound the probability of an edge appearing at an end of the path, which we will need when we glue the extra-tight paths together into one extra-tight trail.
		
		\newcommand{\eP}{e^\ast}
		\begin{lemma}\label{lem: there is nice distribution}
			Let $1/n\ll1/t\ll\alpha\ll1/d\le 1/2$. Let $G$ be an $n$-vertex $d$-graph with $\delta(G)\ge (1-\alpha)n$. Let $\Omega$ be the set of all extra-tight paths of order $t$ in~$G$. Then there exists a probability distribution on $\Omega$ such that a randomly chosen element $P\in \Omega$ has the following properties:
			
			\begin{enumerate}
				\item\label{enum: Pr(P=Q)} for any given extra-tight path $Q$, we have $\P[P=Q] = \O_t\big(n^{-t}\big)$;
				\item\label{enum: Pr(e in E(P))} for every $e\in G$, we have $\abs{\P[e\in P]-\frac{\eP}{|G|}} = \O_t\big(n^{-d-1}\big)$ where $\eP$ is the number of edges in an extra-tight path of order $t$;
				\item\label{enum: Pr(e at end of P)} for every $e\in G$, we have $\abs{\P[\calE_e]-\frac{2}{|G|}} = \O_t\big(n^{-d-1}\big)$ where $\calE_e$ is the event that the $d$ vertices of $e$ are forming one of the two ends of~$P$.
			\end{enumerate}
		\end{lemma}
		\begin{proof}
			Let $1/n\ll1/t\ll\alpha\ll1/d \le 1/2$ and let $\myx$ be a $0.9$-normal fractional $K_{d+1}$-decomposition whose existence is guaranteed by \Cref{lem: there is a mu-normal fractional decomposition}.
			
			Let $\Y=(Y_1,Y_2,\dots)$ be the random walk defined via $\myx$ as above. Consistent with \cite{gishboliner2023tight}, we will use $\Pr$ as the probability measure corresponding to the random walk whereas $\P$ will denote the desired probability measure on~$\Omega$. Let $\B$ be the event that $Y_1,\dots, Y_t$ are pairwise distinct.
			Let $Q\in \Omega$ be any path with vertices $q_1,\dots, q_t$ in that order. Then we define the distribution on $\Omega$ via
			\[\P[P=Q]\coleq\Pr\!\big[\{Y_1=q_1,\dots,Y_t=q_t\}\cup\{Y_1=q_t,\dots, Y_t=q_1\}|\B\big].\]
			We already know that 
			\[\Pr[Y_1=q_1,\dots, Y_d=q_d]=\frac{1}{\abs{(V)_d}}=\frac{1}{d!\, |G|}=\O_{\alpha}\!\left(n^{-d}\right).\]
			Furthermore, for every $\arr{Z_i}=(Y_{i-d+1},\dots, Y_i)$ and every $v\in V(G)$, we have the following bound on the transition using that $\myx$ is $0.9$-normal:
			\[\Pr[Y_{i+1}=v|Y_{i-d+1},\dots,Y_i]=\myx(Z_i\cup\{v\})\le0.9^{-1}\cdot\frac{1}{n}=\O\!\left(n^{-1}\right).\]
			
			Thus, the chain rule yields $\Pr[Y_1=q_1,\dots,Y_t=q_t]=\O_{t}\!\left(n^{-t}\right)$. The number of walks of order $t$ which are not self-avoiding is $\O_t\!\left(n^{t-1}\right)$. Hence, $\Pr[\B^c]=\O_t\!\left(n^{t-1}\right)\cdot\O_{t}\!\left(n^{-t}\right)=\O_{t}\!\left(n^{-1}\right)$. Finally, we get
			\[\P[P=Q]=\frac{\Pr[Y_1=q_1,\dots,Y_t=q_t]+\Pr[Y_1=q_t,\dots,Y_t=q_1]}{\Pr[\B]}=\O_{t}\!\left(n^{-t}\right).\]
			This proves \ref{enum: Pr(P=Q)}.
			
			For \ref{enum: Pr(e in E(P))} and \ref{enum: Pr(e at end of P)}, we define the following set:
			\[E_W\coleq\!\left\{(j_1,\dots,j_d)\in [t]^d|\exists i\in [t-d+1]:i=j_1<j_2<\dots<j_d\le i+d\right\}.\]
			The set $E_W$ contains all ordered tuples of indices whose vertices in the walk form an edge in the extra-tight path.
			
			Fix an $e\in G$. For each $(j_1,\dots,j_d)\in E_W$, the number of walks of order $t$ which are not self-avoiding and where $\{Y_{j_1},\dots,Y_{j_d}\}=e$ is $\O_{t}\!\left(n^{t-d-1}\right)$ and thus $\Pr[\{\{Y_{j_1},\dots,Y_{j_d}\}=e\}\cap\B^c]=\O_{t}\!\left(n^{t-d-1}\right)\cdot\O_{t}\!\left(n^{-t}\right)=\O_{t}\!\left(n^{-d-1}\right)$.
			Furthermore, we get
			\begin{align*}\Pr[\{\{Y_{j_1},\dots,Y_{j_d}\}=e\}\cap\B | \B]&= \frac{1}{\Pr[\B]}\left(\Pr[\{Y_{j_1},\dots,Y_{j_d}\}=e]-\Pr[\{\{Y_{j_1},\dots,Y_{j_d}\}=e\}\cap\B^c]\right)\\
				&= \frac{1}{1-\O_t(n^{-1}) }\left(\frac{d!}{\abs{(V)_d}}-\O_{t}\!\left(n^{-d-1}\right) \right)=\frac{1}{|G|}-\O_{t}\!\left(n^{-d-1}\right),\end{align*}
			where we used \Cref{lem: Edge appearance uniform} and $\Pr[\B]=1-\O_{t}\!\left(n^{-1}\right)$ for the second equality. 
			
			Note that $\mathcal E_e=(\{\{Y_1,\dots,Y_d\}=e\}\cap\B)\cup (\{\{Y_{t-d+1},\dots,Y_t\}=e\}\cap\B)$ is a disjoint union.
			Hence, we can immediately conclude
			\[\P[\mathcal E_e] = \Pr[\{\{Y_{1},\dots,Y_{d}\}=e\}\cap\B | \B]+\Pr[\{\{Y_{t-d+1},\dots,Y_{j_t}\}=e\}\cap\B | \B]=\frac{2}{|G|}-\O_{t}\!\left(n^{-d-1}\right),\]
			which shows~\ref{enum: Pr(e at end of P)}. For \ref{enum: Pr(e in E(P))}, we can conclude
			\begin{align*}
				\P[e\in P]&=\Pr\!\left[\bigcup_{\{j_1,\dots,j_d\}\in E_W}\big\{\{Y_{j_1},\dots,Y_{j_d}\}=e\big\}\cap\B \Big| \B \right]\\
				&=\sum_{\{j_1,\dots,j_d\}\in E_W}\frac{\Pr\big[\big\{\{Y_{j_1},\dots,Y_{j_d}\}=e\big\}\cap\B\big]}{\Pr[\B]}\ge \frac{\eP}{|G|}-\O_{t}\!\left(n^{-d-1}\right),
			\end{align*}
			where we used $\Pr[\B]\le 1$ for the inequality. For the other direction in \ref{enum: Pr(e in E(P))}, we use $\Pr[\B]=1-\O_{t}\!\left(n^{-1}\right)$ and $\Pr\!\big[\big\{\{Y_{j_1},\dots,Y_{j_d}\}=e\big\}\cap\B\big]\le \Pr\big[\{Y_{j_1},\dots,Y_{j_d}\}=e\big]=\frac{1}{|G|}$ (by \Cref{lem: Edge appearance uniform}) to conclude
			\begin{align*}
				\P[e\in P]&=\sum_{\{j_1,\dots,j_d\}\in E_W}\frac{\Pr\big[\big\{\{Y_{j_1},\dots,Y_{j_d}\}=e\big\}\cap\B\big]}{\Pr[\B]}\le \eP\cdot\frac{\frac{1}{|G|}}{1-\O_{t}\!\left(n^{-1}\right)}\\&=\frac{\eP}{|G|}\!\left(1+\O_{t}\!\left(n^{-1}\right)\right)=\frac{\eP}{|G|}+\O_{t}\!\left(n^{-d-1}\right).\qedhere
			\end{align*}
		\end{proof}
		
		Now, we are ready to prove the existence of a packing of the constant-length extra-tight paths covering almost all edges. There, we will use the following nibble theorem.
		\newcommand{\myeps}{\eps}
		\begin{theorem}[Theorem 1.2 in \cite{ehard2020pseudorandom}]\label{theo: nibble}Let $1/\Delta\ll\delta,1/d \le 1/2$ and $\eps\coleq\delta/(50d^2)$. Let $H$ be a $d$-graph with $\onedegree(H)\le\Delta$ and with $\twodegree(H)\le\Delta^{1-\delta}$ as well as $|H|\le\exp\big(\Delta^{\myeps^2}\big)$. Suppose that $\mathcal W$ is a set of at most $\exp\big(\Delta^{\myeps^2}\big)$ weight functions $\omega\colon H\to\mathbb R_{\ge 0}$ with $\omega(H)\ge\max_{e\in H}\omega(e)\Delta^{1+\delta}$. Then there exists a matching $\mathcal M$ in $H$ such that $\omega(\mathcal M)= (1\pm \Delta^{-\myeps})\omega(H)/\Delta$ for all $ \omega\in\mathcal W$. 
		\end{theorem}
		
		\begin{lemma}\label{lem: approximate decomposition with multiple paths}
			Suppose $1/n\ll \gamma\ll\alpha\ll1/d\le 1/2$. Let $G$ be an $n$-vertex $d$-graph with $\delta(G)\ge(1-\alpha)n$. Then there exists a packing $\mathfrak{P}$ of extra-tight paths such that 
			\begin{enumerate}
				\item\label{enum: small leftover} the leftover $L$, of all edges of $G$ not covered by any of the paths of the packing, satisfies $\codegree(L)\le \gamma n$;
				\item\label{enum: not at too many ends} each $S\in \binom{V(G)}{d-1}$ appears at an end of at most $\gamma n$ of these paths. 
			\end{enumerate} 
		\end{lemma}
		\begin{proof}
			Let $t\coleq 1/\gamma$ be an integer, and $\delta$, $\beta$ be such that $1/n\ll\delta\ll\beta\ll 1/t\ll\alpha\ll 1/d$.
			
			Set $N\coleq n^{t-1/2}$ and let $\cP\coleq\{P_i:i\in [N]\}$ be a set of $N$ extra-tight paths of order $t$ sampled independently according to the distribution of \Cref{lem: there is nice distribution}. For an edge $e\in G$, let $X_e\coleq\{ i\in [N]:e\in P_i\}$. Similarly, let $W_e$ be the set of indices $i\in [N]$ where $e$ appears at an end of~$P_i$. By \Cref{enum: Pr(e in E(P))} and \Cref{enum: Pr(e at end of P)} of \Cref{lem: there is nice distribution}, we have $\E[|X_e|]=N\cdot \frac{\eP}{|G|}\big(1\pm\O_{t}\big(n^{-1}\big)\big)=\eP \frac{n^{t-1/2}}{|G|}\big(1\pm\O_{t}\big(n^{-1}\big)\big)$ and $\E[\abs{W_e}]=N\cdot\frac{2}{|G|}\big(1\pm\O_{t}\big(n^{-1}\big)\big)=2\frac{n^{t-1/2}}{|G|}\big(1\pm\O_{t}\big(n^{-1}\big)\big)$ where, as before, $\eP$ is the number of edges in an extra-tight path of order~$t$. Therefore, the Chernoff bound (\Cref{theo: Chernoff}) and a union bound yield that, with high probability,\begin{equation}\label{eq: X_e}\abs{X_e}=(1\pm\beta/3)\eP\frac{n^{t-1/2}}{|G|}=\Theta_t\!\left(n^{t-d-1/2}\right)\hspace{.5cm}\text{and}\hspace{.5cm}\abs{W_e}=(1\pm\beta/3)2\frac{n^{t-1/2}}{|G|}=\Theta_t\!\left(n^{t-d-1/2}\right)\end{equation}
			hold for all $e\in G$.
			
			Right now, the paths in $\cP$ are not necessarily unique. In the next step, we want to delete paths with the same edge set. However, we also have to make sure that there is no edge $e$ where too many paths are deleted that contain~$e$. Therefore, we have to bound how many redundant paths contain a fixed edge~$e$.
			
			Fix an edge $e\in G$. For $i<j\in [N]$, let $Y_{i,j}$ be the indicator variable of the event $\lbrace e\in P_i=P_j\rbrace$ and set $Z_e\coleq\sum_{i<j\in[N]}Y_{i,j}$. First, we compute the expectation of~$Z_e$. For this, it is enough to note that
			\[\P[Y_{i,j}=1]=\P[e\in P_i]\cdot\P[P_j=P_i]=\O_t\!\left(n^{-d-t}\right)\]
			using \ref{enum: Pr(P=Q)} and \ref{enum: Pr(e in E(P))} of \Cref{lem: there is nice distribution}. Thus, linearity of expectation yields $\E[Z_e]=\binom{N}{2}\P[Y_{i,j}=1]=\O_t\!\left(n^{t-d-1}\right)$. 
			
			We want to use the second moment method. For this, we need to bound the variance of~$Z_e$. Thus, we will determine the covariance of the $Y_{i,j}$ next. Let $i,i',j,j'\in[N]$ be pairwise distinct, then $\Cov[Y_{i,j},Y_{i',j'}]=0$ and $\Var[Y_{i,j}]=\E[Y_{i,j}^2] - \E[Y_{i,j}]^2=\O_t\!\left(n^{-d-t}\right)$. Furthermore, 
			\begin{align*}
				\Cov[Y_{i,j},Y_{i',j}]&=\E[Y_{i,j}\cdot Y_{i',j}]-\E[Y_{i,j}]\cdot\E[Y_{i',j}]=\P[e\in P_i=P_{i'}=P_j]-\O_t\!\left(n^{-2t-2d}\right)\\
				&=\O_t\!\left(n^{-2t-d}\right)-\O_t\!\left(n^{-2t-2d}\right)=\O_t\!\left(n^{-2t-d}\right)
			\end{align*}
			again using \ref{enum: Pr(P=Q)} and \ref{enum: Pr(e in E(P))} of \Cref{lem: there is nice distribution}. Similarly, we get $\Cov[Y_{i,j},Y_{i,j'}]=\O_t\!\left(n^{-2t-d}\right)$ and $\Cov[Y_{i,j},Y_{j,j'}]=\O_t\!\left(n^{-2t-d}\right)$. Hence, Bienaymé's identity yields
			\begin{align*}
				\Var[Z_e]&=\sum_{i<j\in[N]}\Var[Y_{i,j}]+\sum_{i<j<k\in[N]}2\Cov[Y_{i,k},Y_{j,k}]+2\Cov[Y_{i,j},Y_{i,k}]+2\Cov[Y_{i,j},Y_{j,k}]\\
				&= \binom{N}{2}\O_t\!\left(n^{-d-t}\right)+6\binom{N}{3}\O_t\!\left(n^{-2t-d}\right)=\O_t\!\left(n^{t-d-1}\right)+\O_t\!\left(n^{t-d-3/2}\right)=\O_t\!\left(n^{t-d-1}\right).
			\end{align*}
			Finally Chebyshev's inequality yields
			\begin{align*}
				\P\big[\abs{Z_e-\E[Z_e]}\ge n^{t-d-1}\big]&\le\frac{\Var[Z_e]}{n^{2t-2d-2}}=\O_t\!\left(n^{-t+d+1}\right).
			\end{align*}
			By a union bound over all $\binom{n}{d}$ ways to choose the edge $e\in G$, we can conclude that there is a large constant $C_t$ depending on $t$ such that 
			\begin{equation}\label{eq: Z_e}
				Z_e\le C_t n^{t-d-1} \end{equation} holds for all $e\in G$ with high probability where we use the fact that $1/t\ll 1/d$. 
			
			We fix a collection $\cP$ of paths that satisfy \eqref{eq: X_e} and \eqref{eq: Z_e} for all~$e$. Let $\cP'\subs\cP$ be the collection obtained from $\cP$ by deleting $P_i, P_j$ for every pair $i<j\in[N]$ with $P_i=P_j$. Let $X_e'\coleq\lbrace i\in X_e:P_i\in \cP'\rbrace$ and $W_e'\coleq\lbrace i\in W_e:P_i\in \cP'\rbrace$. By \eqref{eq: X_e} and \eqref{eq: Z_e}, we get that $\abs{X_e'}\ge (1-\beta/2)\eP\frac{n^{t-1/2}}{|G|}$ and $\abs{W_e'}\ge (1-\beta/2)2\frac{n^{t-1/2}}{|G|}$.
			
			We define an auxiliary $\eP$-graph $H$ with the edges of $G$ as vertices and the paths from $\cP'$ as edges, i.e.~$H=\{P :P\in\cP'\}$. Since we deleted paths with the same edge set, $H$ does not have multiple edges and $|H|=\abs{\cP'}$. To apply \Cref{theo: nibble}, we have to establish bounds on $\onedegree(H)$ and $\twodegree(H)$. By \eqref{eq: X_e}, we have $\deg_H(\{e\})\le\abs{X_e}\le (1+\beta/3)\eP\frac{n^{t-1/2}}{|G|}$. Hence, we can set $\Delta\coleq\ceil{(1+\beta/3)\eP\frac{n^{t-1/2}}{|G|}}=\Theta_t\big(n^{t-d-1/2}\big)$ and get $\onedegree(H)\le\Delta$. Next, we check $\twodegree(H)$. For $e,f\in G$, the number of extra-tight paths of order $t$ that contain $e$ and $f$ is at most $t!\cdot n^{t-\abs{e\cup f}}$. Thus, $\twodegree(H)\le t!n^{t-d-1}\le\Delta^{1-\delta}$ as $\delta\ll 1/t$. Finally, we also have $|H|\le N=n^{t-1/2}\le\exp\big(\Delta^{\myeps^2}\big)$ where $\myeps\coleq\delta/50\eP{}^2$. Thus, $H$ is suitable for \Cref{theo: nibble}.
			
			This theorem will give us a matching that corresponds to the paths we are looking for. To ensure that the paths fulfill the requirements of the lemma, we now define suitable weight functions. 
			
			For $S\in\binom{V(G)}{d-1}$, let $\omega_S\colon H\to[4]_0$ be defined such that $\omega_S(P)$ is the number of edges of $P$ which contain~$S$. Note that, since $S$ has size $d-1$, at most $4$ edges in the extra-tight path $P$ can contain~$S$. Fix an $S\in\binom{V(G)}{d-1}$.
			Then
			\[\omega_S(H)=\sum_{e\supseteq S}\abs{X'_e}\ge (1-\beta/2)\eP\frac{n^{t-1/2}}{|G|}\deg_G(S)=\Omega_t\!\left(n^{t-d+1/2}\right),\]
			where we use $\deg_G(S)\ge \delta(G)\ge (1-\alpha)n$ for the last equality.
			Since $\max_{P\in H}\omega_S(P)\Delta^{1+\delta} \le 4 \Delta^{1+\delta}=\O_t\!\left(n^{(t-d-1/2)(1+\delta)}\right)$ and $\delta\ll 1/t$, we can conclude $\omega_S(H)\ge \max_{P\in H}\omega_S(P)\Delta^{1+\delta}$.
			Hence, $\omega_S$ fulfills the requirement of \Cref{theo: nibble}.
			These weight functions will ensure that our matching obeys~\ref{enum: small leftover}. For \ref{enum: not at too many ends}, we do the following: For every $S\in\binom{V(G)}{d-1}$, let $\omega'_S\colon H\to \{0,1\}$ be defined such that $\omega_S(P)$ is 1 if and only if $S$ appears at an end of~$P$. We have
			\[\omega_S'(H)=\sum_{e\supseteq S}\abs{W'_e}\ge(1-\beta/2)2\frac{n^{t-1/2}}{|G|}\deg_G(S)\ge \max_{P\in H}\omega'_S(P)\Delta^{1+\delta}.\]
			Hence, $\omega'_S$ fulfills the requirement of \Cref{theo: nibble}.
			
			By \Cref{theo: nibble}, there is a matching $\mathcal{M}\subs H$ such that for all $S\in \binom{V(G)}{d-1}$, we have $\omega_S(\mathcal M)=\left(1\pm\Delta^{-\myeps}\right)\omega_S(H)/\Delta$ and $\omega'_S(\mathcal M)=\left(1\pm\Delta^{-\myeps}\right)\omega'_S(H)/\Delta$. We claim that the family $\mathfrak P$ of extra-tight paths corresponding to the edges of $\mathcal{M}$ are extra-tight paths fulfilling \ref{enum: small leftover} and \ref{enum: not at too many ends} (recall $\gamma\coleq 1/t$).
			
			To check \ref{enum: small leftover}, let $S\in \binom{V(G)}{d-1}$. We have to lower bound the number of edges containing $S$ covered by the extra-tight paths. Since $\mathcal{M}$ is a matching, the corresponding extra-tight paths are edge-disjoint. Therefore, the edges containing $S$ covered by the paths are exactly 
			\begin{align*}\omega_S(\mathcal M)&=\!\left(1\pm\Delta^{-\myeps}\right)\omega_S(H)/\Delta\\&\ge\!\left(1-\Delta^{-\myeps}\right)(1-\beta/2)\eP\frac{n^{t-1/2}}{|G|}\deg_G(S)/\ceil{(1+\beta/3)\eP\frac{n^{t-1/2}}{|G|}}\\ 
				&\ge\frac{\!\left(1-\Delta^{-\myeps}\right)(1-\beta/2)}{(1+\beta/2)}\deg_G(S)\ge\!\left(1-\frac{1}{t}\right)\deg_G(S).
			\end{align*}
			This shows that the leftover $L$ satisfies $\Delta(L)\le\frac{n}{t}$ which proves~\ref{enum: small leftover}. Furthermore, we have
			\begin{align*}
				\omega_S'(\mathcal M)&\le\left(1+\Delta^{-\myeps}\right)\omega'_S(H)/\Delta\le \left(1+\Delta^{-\myeps}\right)\sum_{e\supseteq S}\abs{W_e}/\Delta\\&\le\left(1+\Delta^{-\myeps}\right)n(1+\beta/3)2\frac{n^{t-1/2}}{|G|}/\!\ceil{(1+\beta/3)\eP\frac{n^{t-1/2}}{|G|}}\le\frac{\left(1+\Delta^{-\myeps}\right)2}{\eP}n\\
				&\le\frac{1}{t}n.
			\end{align*}
			This shows that no set $S\in\binom{V(G)}{d-1}$ appears at an end of more than $\frac{n}{t}$ paths which proves~\ref{enum: not at too many ends}.
			
			Hence, both properties are satisfied.
		\end{proof}
		
		We are now ready to prove the Approximate Decomposition Lemma. There, we just have to take the packing of the extra-tight paths found in the previous lemma and glue them together. For this, it is necessary to set aside a small random edge set before.
		
		\begin{proof}[Proof of the Approximate Decomposition \Cref{lem: approximate decomposition}]
			Let $1/n\ll \gamma_0 \ll \beta\ll \gamma\ll \alpha,1/D\ll 1/d\le 1/2$.
			
			Let $R$ be a random graph obtained by taking each edge of $G$ independently with probability~$\beta$. For each $S\in\binom{V(G)}{d-1}$, we have that $\beta (1-\alpha)n\le \E[R(S)]\le \beta n$ where we used that $\delta(G)\ge (1-\alpha)n$. Similarly, for each $A\subs \binom{V(G)}{d-1}$ with $\abs A\le D$, we have that $\E\big[\abs{\bigcap_{S\in A}R(S)}\big]\ge \beta^D(1-D\alpha)n$. By the Chernoff bound (\Cref{theo: Chernoff}) and a union bound, we get that there is a set $R$ such that $\Delta(R)\le 2\beta n$ and for all $A\subs \binom{V(G)}{d-1}$ with $\abs A\le D$, we have $\abs{\bigcap_{S\in A}R(S)}\ge \frac{1}{2}\beta^D(1-D\alpha) n$.
			
			Let $G'\coleq G-R$ and note $\delta(G')\ge \delta(G)-\Delta(R)\ge (1-2\alpha)n$. Thus, we can apply \Cref{lem: approximate decomposition with multiple paths} with $G_{\ref{lem: approximate decomposition with multiple paths}}\coleq G'$, $\alpha_{\ref{lem: approximate decomposition with multiple paths}}\coleq 2\alpha$, and $\gamma_{\ref{lem: approximate decomposition with multiple paths}}\coleq\gamma_0$. We get a packing $\mathfrak P$ of extra-tight paths such that the leftover $L'$ satisfies $\Delta(L')\le \gamma_0 n$ and each $S\in \binom{V(G)}{d-1}$ appears at an end of at most $\gamma_0 n$ of these paths.
			
			Finally, we connect these paths by applying \Cref{lem: path connecting} with $G_{\ref{lem: path connecting}}\coleq R$, $U_{\ref{lem: path connecting}}\coleq V(R)$, $\xi_{\ref{lem: path connecting}}\coleq \frac{1}{2}\beta^D(1-D\alpha)$, $\mathcal P_{\ref{lem: path connecting}}\coleq\mathfrak P$, $\gamma_{\ref{lem: path connecting}}\coleq\gamma_0$, and $\gamma'_{\ref{lem: path connecting}}$ any suitable number.
			Here we use that each $S\in \binom{V(G)}{d-1}$ appears at an end of at most $\gamma_0 n$ extra-tight paths in~$\mathfrak P$ and that $\gamma_0 \ll \xi_{\ref{lem: path connecting}}$.
			We get a subset $H\subs R$ such that $H\cup\bigcup_{P\in\mathfrak P}P$ is the edge set of an extra-tight trail whose leftover $L\subs L' \cup R$ in $G$ satisfies $\Delta(L)\le\Delta(L' \cup R)\le \gamma_0 n + 2 \beta n \le \gamma n$. 
		\end{proof}
		
		\section{Cover Down Lemma }\label{sec: Cover Down}
		In this section, we will prove the Cover Down \Cref{lem: Cover Down}. The general idea of the proof is as follows. By the Approximate Decomposition \Cref{lem: approximate decomposition}, we can decompose $G-G[U]$ into an extra-tight trail and a small leftover~$L$. The question is how we can incorporate $L$ into one big extra-tight trail. We do this by setting aside a small graph $G'$ in the beginning, which has the following absorbing property: No matter what (even smaller) edge set $L$ is left over by the Approximate Decomposition, $G'\cup L$ can be partitioned into extra-tight cycles of constant length. This alone is not enough since we also have to connect these cycles to our extra-tight trail. Therefore, we also ensure that each of these constant length cycles has $2d$ consecutive vertices in~$U$. Then, we can ``cut'' each cycle at these $2d$ vertices to get a set of extra-tight paths that have all ends in $U$ and still cover all edges outside of~$G[U]$. Then, we can glue all these paths together into one extra-tight trail using edges in~$G[U]$.
		
		The following lemma precisely states the necessary absorbing property:
		
		\begin{lemma}[Supercomplex Lemma]\label{lem: Supercomplex Lemma}
			Let $1/n\ll \gamma \ll \eta\ll\alpha,\mu\ll 1/f\ll 1/d\le 1/2$ with $\alpha<\mu^{1.1}$. Let $G$ be a $d$-graph on $n$ vertices and $U\subs V(G)$ with $\abs{U}=\mu n$ and $\delta(G)\ge (1-\alpha)n$. Then there is a subgraph $G'\subs G$ with $\Delta(G')\le \eta n$ with the following property: for every $L\subs G$ with $\Delta(L)\le \gamma n$ such that $d^2\mid \deg_{G'\cup L}(v)$ for every $v\in V(G)\backslash U$, it is possible to find a packing of copies of $\extratightcycle{f}{d}$ in $G'\cup L$, covering all edges in $(G'\cup L)\backslash(G'\cup L)[U]$ and each of these extra-tight cycles has at most $d$ vertices outside of~$U$.
		\end{lemma}
		
		The proof of this lemma is quite intricate and uses the heavy machinery of supercomplexes that were introduced in~\cite{glock2023existence}. Therefore, we first see how it can be used to prove the Cover Down \Cref{lem: Cover Down}.
		
		\begin{proof}[Proof of the Cover Down \Cref{lem: Cover Down}]
			Let $1/n\ll \beta \ll \gamma \ll\eta\ll \gamma'\ll \xi \ll \alpha,\mu\ll 1/D\ll 1/f\ll 1/d\le 1/2$ with $\alpha<\mu^{1.1}$. Apply \Cref{lem: Supercomplex Lemma} with $\gamma_{\ref{lem: Supercomplex Lemma}}\coleq 2\gamma$ to get a subgraph $G'\subs G$ with $\Delta(G')\le \eta n$ and the absorbing property given in \Cref{lem: Supercomplex Lemma}. Let $G''\coleq G-G'$. Note that $\delta(G'')\ge \delta(G)-\Delta(G')\ge (1-2\alpha)n$.
			
			Next, we set aside a small edge set $R$, which we use later to extend an extra-tight trail into~$U$. Let $R$ be a random graph obtained by taking each edge of $G''$ independently with probability~$\beta$. For each $A\subs \binom{V(G'')}{d-1}$ of size $d$, we have $\abs{\bigcap_{S\in A}G''(S)\cap U}\ge(\mu-d\cdot 2\alpha)n\ge \frac{1}{2}\mu n$. Thus, $\E\big[\bigcap_{S\in A}R(S)\cap U\big]\ge \frac{1}{2}\beta^d\mu n$. By the Chernoff bound (\Cref{theo: Chernoff}) and a union bound, we get that there is a set $R$ such that $\Delta(R)\le 2\beta n$ and for all $A\subs \binom{V(G'')}{d-1}$ of size $d$, we have $\abs{\bigcap_{S\in A}R(S)\cap U}\ge\frac13\beta^d \mu n$. Let $G'''\coleq G''-R$ and note that $\delta(G''')\ge (1-3\alpha)n$. 
			
			Thus, we can apply the Approximate Decomposition \Cref{lem: approximate decomposition} with $G_{\ref{lem: approximate decomposition}}\coleq G'''-G'''[U]$ and $\alpha_{\ref{lem: approximate decomposition}}\coleq 3\alpha+\mu$ to get an extra-tight trail $T'$ with vertex sequence $(v_1,\dots,v_d,\dots,w_1,\dots,w_d)$ such that the leftover $L'$ of all edges of $G'''-G'''[U]$ not covered by the trail satisfies $\Delta(L')\le \gamma n$. Thanks to the property of $R$, we can now greedily find vertices $u_1,\dots,u_{2d}\in U$ such that \[\big(u_1,\dots,u_d,\underbrace{v_1,\dots,v_d,\dots,w_1,\dots,w_d}_{T'},u_{d+1},\dots,u_{2d}\big)\]
			is an extra-tight trail $T''$ where $T''\backslash T'\subs R$. 
			
			This is the only time where we need $R$, so now we are in the situation that each edge in $G-G[U]$ is either in $G'$, the extra-tight trail $T''$ (whose ends are in $U$), or the leftover $L$ with $\Delta(L)\le \Delta(L')+\Delta(R)\le 2\gamma n$. Since both ends of $T''$ are in $U$, we get $d^2\mid \deg_{T''}(v)$ for all $v\in V(G)\backslash U$. This, together with the assumption that $d^2\mid \deg_G(v)$ holds for all $v\in V(G)\backslash U$ implies that $d^2\mid \deg_{G'\cup L}(v)$ for all $v\in V(G)\backslash U$. Thus, by the property of $G'$ given by \Cref{lem: Supercomplex Lemma}, we get a packing $\mathcal C$ of extra-tight cycles of order at least $f$ covering all edges in $(G'\cup L)\backslash(G'\cup L)[U]$ and each of these extra-tight cycles has at most $d$ vertices outside of~$U$. 
			
			Since $f$ is large enough depending on $d$, this means that each $C\in \mathcal \C$ has $2d$ consecutive vertices $(u_1',\dots,u_{2d}')$ which are all in~$U$. By ``cutting'' $C$ between $u_d'$ and $u_{d+1}'$, we get an extra-tight path $P$ with vertex sequence $(u_d',\dots,u_1',\dots,u_{2d}',\dots,u_{d+1}')$. Note that $C\backslash P\subs G[U]$. Let $\mathcal P$ be the collection of all extra-tight paths, which we get by cutting each cycle in $\mathcal C$ in that way, together with the extra-tight trail~$T''$. Note that by $\Delta(G'\cup L)\le \Delta(G')+\Delta(L)\le 2\eta n$, each $S\in\binom{U}{d-1}$ is at the end of at most $2\eta n+1\le 3\eta n$ paths in~$\mathcal P$.
			
			We are now in the situation that $G\backslash G[U]\subs \bigcup_{P\in\mathcal P}P\subs G$, both ends of every $P\in\mathcal P$ lie entirely in $U$, and $\Delta\big(\bigcup_{P\in\mathcal P}P[U]\big)\le 2+\Delta(G')+\Delta(L)\le 2\eta n$. In a last step, we have to connect all the extra-tight trails in $\mathcal P$ into one extra-tight trail~$T$. 
			
			For this, we use \Cref{lem: path connecting} with $G_{\ref{lem: path connecting}}\coleq G[U]-\bigcup_{P\in\mathcal P}P$, $U_{\ref{lem: path connecting}}\coleq V(G_{\ref{lem: path connecting}})$, and $\gamma_{\ref{lem: path connecting}}\coleq 2\eta$. Since $\delta(G_{\ref{lem: path connecting}})\ge \delta(G[U])-2\eta n\ge (\mu-\alpha)n-2\eta n\ge (\mu-2\alpha)n$, we get that for all $A\subs\binom{V(G_{\ref{lem: path connecting}})}{d-1}$ with $\abs{A}\le D$, we have $\abs{\bigcap_{S\in A}G(S)}\ge (\mu-D\cdot 2\alpha)n\ge \xi n$. Thus, the requirements of \Cref{lem: path connecting} are fulfilled and we get an $H\subs G[U]$ such that $H\cup\bigcup_{P\in\mathcal P}P$ is the edge set of an extra-tight trail $T$ with $\Delta(H)\le \gamma'n$. Thus $\Delta(T[U])\le \gamma'n+2\eta n\le \mu^2n$ as desired.
		\end{proof}
		
		\subsection{Proof of the Supercomplex Lemma~\ref{lem: Supercomplex Lemma}}
		The last thing to do is to prove the Supercomplex \Cref{lem: Supercomplex Lemma}. As the name suggests, we use supercomplexes to prove this lemma, which were introduced in~\cite{glock2023existence}. To align with their nomenclature, we will omit the word \emph{simplicial} when referring to a simplicial complex from now on. In this section, we state the necessary lemmas and then prove the Supercomplex Lemma assuming these lemmas.
		
		First we need some notation.
		Suppose $G$ is a complex and $e \subseteq V(G)$.
		Define $G(e)$ as the complex on the vertex set $V(G) \backslash e$ containing all sets $e' \subseteq V(G) \backslash e$ such that $e \cup e' \in G$. For $i\in\mathbb N_0$, we write $G^{(i)}$ for the $i$-graph on the same vertex set as $G$ consisting of all $i$-sets of~$G$. For a set $S$, we let $G[S]$ be the complex on $S\cap V(G)$ containing all $e\in G$ with $e\subs S$. For a $d$-graph $H$ and a complex $G$, let $G[H]$ be the complex on $V(G)$ with edge set $G[H]\coleq \!\left\{e\in G:\binom{e}{d}\subs H\right\}$ and $G-H\coleq G[G^{(d)}-H]$. Furthermore, we extend our notion of $F$-decompositions to complexes in the following way:
		\begin{definition}
			If $F$ is a $d$-graph with $\abs{V(F)}=f$ and $G$ a complex, then $G$ is called \emph{$F$-divisible} if $G^{(d)}$ is $F$-divisible. An \emph{$F$-packing} in $G$ is an $F$-packing $\F$ in $G^{(d)}$ where $V(F')\in G^{(f)}$ for all $F'\in\F$. If, additionally $\bigcup_{F'\in\F}F'=G^{(d)}$, then $\F$ is an \emph{$F$-decomposition}.
			
			A $d$-graph $F$ is \emph{weakly regular} if there are positive integers $s_0,\dots,s_{d-1}$ such that for all $i\in[d-1]_0$ and $S\in\binom{V(F)}{i}$, we have $\abs{F(S)}\in\{0,s_i\}$.
		\end{definition}
		
		The definition of an $(\eps,\xi,f,d)$-supercomplex is quite technical. Since we do not need it for now, we postpone its definition (\Cref{def: complex,def: supercomplex}) to Section~\ref{sec: Existence of Supercomplex}, where we prove the existence of our desired supercomplex. All we need for now is that a supercomplex is a complex with the following property, which was the main motivation for \cite{glock2023existence} to consider supercomplexes in the first place.
		
		\begin{theorem}[\cite{glock2023existence}, Theorem 4.7]\label{theo: weakly regular + super implies decomposition}
			For all $d \in \mathbb N$, the following is true.
			Let $1/n \ll\eps \ll \xi,1/f$ and $f > d$.
			Let $F$ be a weakly regular $d$-graph on $f$ vertices
			and let $G$ be an $F$-divisible $(\eps,\xi,f,d)$-supercomplex on $n$ vertices.
			Then $G$ has an $F$-decomposition.
		\end{theorem}
		
		We will use this theorem in the proof of the Supercomplex Lemma to partition $G'\cup L$ into extra-tight cycles. Note that extra-tight cycles are not weakly regular, but this will not be a problem because of a result from \cite{glock2023existence} (cf.~\Cref{lem: F decomposes weakly regular graph}) which allows us to pass from a graph which is not weakly regular to a graph which is. 
		However, as mentioned in the statement of the Supercomplex Lemma, we also need that each of the extra-tight cycles has at most $d$ vertices outside of~$U$. To guarantee this, we will find a ``skew'' supercomplex, where each $f$-set in the supercomplex has at most $d$ vertices outside of~$U$.
		
		\begin{definition}
			For $d<f$, a $d$-graph $G$ and $U\subs V(G)$, we define the \emph{skew $(G,f,d)$-complex} $\Gskew$ to be the complex generated by all $f$-sets $S$ where $G[S]$ is a clique and $\abs{U\cap S}\le d$. 
		\end{definition}
		
		If we can find a supercomplex $\Gsuper\subs\Gskew$ and partition $\Gsuper$ into extra-tight cycles of order $f$, then each extra-tight cycle will have at most $d$ vertices outside of $U$ since its vertices form an $f$-set in~$\Gskew$.
		
		In the following, we give a proof outline of the Supercomplex \Cref{lem: Supercomplex Lemma}, stating all the relevant lemmas. Afterwards, we will prove the Supercomplex Lemma assuming all the relevant lemmas. These will then be proven in the succeeding subsections. To shorten the statements, we introduce the following terminology:
		
		\begin{definition}
			A $d$-graph is called \emph{$(D,\rho)$-rich} in a set $U\subs V(G)$ if for all $A\subs\binom{V(G)}{d-1}$ with $\abs A \le D$, we have $\abs{\bigcap_{S\in A}G^{(d)}(S)\cap U}\ge\rho n$.
			A complex $G$ is called \emph{$(D,\rho)$-rich} in a set $U\subs V(G)$ if $G^{(d)}$ is $(D,\rho)$-rich in~$U$.
			
			An extra-tight cycle in a $d$-graph $G$ with $U\subs V(G)$ is \emph{valid} if it has at most $d$ vertices outside of~$U$.
		\end{definition}
		
		\paragraph{Proof outline for the Supercomplex \Cref{lem: Supercomplex Lemma}.}
		
		In a first step, we find a skew supercomplex $\Gsuper$ in the skew $(G,f,d)$-complex~$\Gskew$. This is given by the following lemma where the 1.1 can be replaced by any number bigger than 1:
		\begin{lemma}[Existence of skew supercomplex]
			\label{lem:skew_super_complex}
			Let $1/n\ll \eps \ll \xi, \rho\ll\eta\ll \alpha , \mu \ll 1/D,1/f \ll 1/d\le 1/2$ with $\alpha<\mu^{1.1}$. Let $G$ be a $d$-graph on $n$ vertices and $U\subseteq V(G)$ with $\abs U=\mu n$ and $\delta(G)\ge (1-\alpha)n$ and $\Gskew$ the skew $(G,f,d)$-complex. Then there is a subcomplex $\Gsuper\subs\Gskew$ that is an $(\eps,\xi,f,d)$-supercomplex, which is $(D,\rho)$-rich in~$U$, and where $\Delta\big(\Gsuper^{(d)}\big)\le \eta n$.
		\end{lemma}
		The $d$-layer $\Gsuper^{(d)}$ is exactly the $d$-graph $G'$ whose existence the Supercomplex Lemma guarantees. To prove its absorbing property, we then assume that we have a small edge set~$L$. We cover this edge-set with valid extra-tight cycles using the following lemma, where $G'$ takes the role of~$G$:
		
		\begin{lemma}\label{lem: Cover L}
			Let $1/n\ll\gamma\ll\gamma'\ll\eps,\xi,\rho\ll1/D\ll 1/f\ll  1/d\le 1/2$. Suppose $G$ is a $d$-graph on $n$ vertices that is $(D,\rho)$-rich on a set $U\subs V(G)$. Furthermore, let $L\subs \binom{V(G)}{d}$ be disjoint to $G$ with $\Delta(L)\le \gamma n$. Then there is a collection $\mathcal F$ of edge-disjoint copies of valid $\extratightcycle{f}{d}$ such that
			\begin{enumerate}
				\item\label{enum: Cover L, cover in L cup G} $L\subs\bigcup_{F\in\F} F\subs L\cup G$;
				\item\label{enum: Cover L, max deg} $\Delta(\bigcup_{F\in\F}F)\le \gamma'n$.
			\end{enumerate}
		\end{lemma}
		
		The previous lemma used a few edges of $G'$, but the majority of edges of $G'$ is still not covered by valid extra-tight cycles. The goal is to cover them now. For that, we have to make sure that it is still a reasonable supercomplex and that it is reasonably rich. The following lemma shows that this is the case:
		
		\begin{lemma}[(i) is shown in \cite{glock2023existence}, Lemma 5.9 (v)]\label{lem: super minus a little is super}
			Let $1/n\ll \gamma\ll \eps,\xi\ll 1/D,1/d,1/f$ with $f>d$. Let $G$ be a complex on $n$ vertices and let $S$ be an $d$-graph on $V(G)$ with $\Delta(S)\le \gamma n$.
			
			\begin{enumerate}
				\item\label{enum: super minus a little is super} If $G$ is an $(\eps,\xi,f,d)$-supercomplex, then $G-S$ is an $\big(2\eps, \xi/2, f,d\big)$-supercomplex.
				\item\label{enum: rich minus a little is rich} If $G$ is $(D,\rho)$-rich in $U\subs V(G)$, then $G-S$ is $(D,\rho/2)$-rich.
			\end{enumerate}
		\end{lemma}
		
		Thus, we can still use \Cref{theo: weakly regular + super implies decomposition} to partition the remaining edges of $G'$ into extra-tight cycles. However, \Cref{theo: weakly regular + super implies decomposition} still imposes two obstacles: It can only decompose a supercomplex into a graph $F$ which is weakly regular (and extra-tight cycles are not weakly regular), and the supercomplex must be $F$-divisible. The first obstacle is easily dealt with by the following already known result:
		
		\begin{lemma}[\cite{glock2023existence}, Lemma 9.2]\label{lem: F decomposes weakly regular graph}
			Let $2\le d< f$. Let $F$ be any $d$-graph on $f$ vertices. There exists a weakly regular $d$-graph $F^\ast$ on at most $2f\cdot f!$ vertices which has an $F$-decomposition.
		\end{lemma}
		
		\begin{definition}
			Let $\Fast_f$ be the weakly regular $d$-graph which has an $\extratightcycle{f}{d}$-decomposition whose existence is given by the previous lemma.
		\end{definition}
		
		Thus, we can partition the supercomplex into copies of $\Fast_f$ and thus into copies of $\extratightcycle{f}{d}$ using \Cref{theo: weakly regular + super implies decomposition} as soon as it is $\Fast_f$-divisble. This is done with the following lemma:
		
		\begin{lemma}[Degree Fixing Lemma]\label{lem: fixing degree}
			Let $1/n\ll\gamma \ll\rho\ll1/D\ll 1/f\ll 1/d\le 1/2$. Let $G$ be a $d$-graph on $n$ vertices that is $(D,\rho$)-rich on a set $U\subs V(G)$. Furthermore, assume that $d^2\mid\deg_G(v)$ for all $v\in V(G)\backslash U$. Then there exist disjoint subsets $S$ and $H$ of $G$ such that
			\begin{enumerate}
				\item\label{enum: fixing degree, S} $S\subs\binom{U}{d}$;
				\item\label{enum: fixing degree, H} $H$ has an $\extratightcycle{f}{d}$-decomposition where each copy of $\extratightcycle{f}{d}$ is valid;
				\item\label{enum: fixing degree, G-S-H} $G-S-H$ is $\Fast_f$-divisible;
				\item\label{enum: fixing degree, max deg} $\Delta(S\cup H)\le \gamma n$.
			\end{enumerate}
		\end{lemma}
		
		With all the relevant lemmas stated, we will now prove the Supercomplex Lemma, following the proof outline above.
		
		\begin{proof}[Proof of the Supercomplex \Cref{lem: Supercomplex Lemma}]
			Let $1/n\ll \gamma\ll\gamma'\ll  \eps\ll\xi,\rho\ll\eta\ll\alpha,\mu\ll 1/D,1/f\ll 1/d\le 1/2$ with $\alpha<\mu^{1.1}$. Let $\Gskew$ be the skew $(G,f,d)$-complex and $\Gsuper\subs\Gskew$ the $(\eps,\xi,f,d)$-supercomplex that is $(D,\rho)$-rich in $U$ and where $\Delta\big(\Gsuper^{(d)}\big)\le \mu n$. We claim that $G'\coleq \Gsuper^{(d)}$ is the $d$-graph whose existence is claimed by the Supercomplex \Cref{lem: Supercomplex Lemma}. Note that $G'$ is $(D,\rho)$-rich and that $\Delta(G')\le\mu n$ is already given. Thus, we just have to check the absorbing property of~$G'$.
			
			For this, let $L\subs G$ be any edge set with $\Delta(L)\le\gamma n$ such that $d^2\mid\deg_{G'\cup L}(v)$ for every $v\in V(G)\backslash U$. We have to find a packing of valid extra-tight cycles in $G'\cup L$, each of order at least $f$, covering all edges in $(G'\cup L)\backslash(G'\cup L)[U]$.
			
			Applying \Cref{lem: Cover L} with $G_{\ref{lem: Cover L}}\coleq G'$, we get a collection $\mathcal F$ of edge-disjoint copies of valid $\extratightcycle{f}{d}$ in $L\cup G$ with $L\subs\bigcup_{F\in\F}F$ and $\Delta\big(\bigcup_{F\in\F}F\big)\le \gamma' n$. Let $G''\coleq G'-\bigcup_{F\in\F}F=(G'\cup L)-\bigcup_{F\in \F}F$. By \Cref{lem: super minus a little is super}, $G''$ is still $(D,\rho/2)$-rich. Furthermore, we have $d^2\mid\deg_{G''}(v)$ for all $v\in V(G)\backslash U$ since $d^2\mid\deg_{G'\cup L}(v)$ and $d^2\mid\deg_{\bigcup_{F\in\F}F}(v)$. 
			
			Thus, we can apply \Cref{lem: fixing degree} with $G_{\ref{lem: fixing degree}}\coleq G''$, $\rho_{\ref{lem: fixing degree}}\coleq\rho/2$. We get disjoint subsets $S\subs G''[U]$ and $H\subs G''$ where $H$ has a decomposition into valid $\extratightcycle{f}{d}$, $G'''\coleq G''-S-H$ is $\Fast_f$-divisible and $\Delta(S\cup H)\le \gamma n$.
			
			Applying \Cref{lem: super minus a little is super} with $\gamma_{\ref{lem: super minus a little is super}}\coleq 2\gamma$, we get that $\Gsuper[G''']$ is an $(2\eps,\xi/2,f,d)$-supercomplex. Since it is also $\Fast_f$ divisible and $\Fast_f$ is weakly regular, we can apply \Cref{theo: weakly regular + super implies decomposition} with $G_{\ref{theo: weakly regular + super implies decomposition}}\coleq G[G''']$ $\eps_{\ref{theo: weakly regular + super implies decomposition}}\coleq 2\eps$, $\xi_{\ref{theo: weakly regular + super implies decomposition}}\coleq \xi/2$ and get an $\Fast_f$-decomposition. Since $\Fast_f$ is $\extratightcycle{f}{d}$-divisible, this gives us an $\extratightcycle{f}{d}$-decomposition of~$G''$. By definition of $\Gsuper[G''']\subs \Gskew$, every $\extratightcycle{f}{d}$ in that decomposition is valid. Hence, we are done.
		\end{proof}
		
		\subsection{Proof of Lemma~\ref{lem: Cover L} and Lemma~\ref{lem: super minus a little is super}}
		The only things left to prove are the Lemmas \ref{lem:skew_super_complex}, \ref{lem: Cover L}, \ref{lem: super minus a little is super} \ref{enum: rich minus a little is rich}, and~\ref{lem: fixing degree}. We start with the two shortest proofs which do not need extra preparations. 
		\begin{proof}[Proof of \Cref{lem: Cover L}]
			We want to apply \Cref{lem: rooted embedding with U}\towriteornottowrite{\ with $\alpha \coleq 1$}{}. Let $L=\{e_1,\dots, e_m\}$. Since $\Delta(L)\le\gamma n$, we have $m\le \gamma n^d$. Suppose $\{v_1,\dots,v_d\}\in\extratightcycle{f}{d}$. For each $j\in[m]$, let $T_j$ be $EC'\coleq \extratightcycle{f}{d}-\{v_1,\dots,v_d\}$ and $X_j=\{v_1,\dots,v_d\}$. Pick $D$ large enough such that $EC'$ has degeneracy at most $D$ rooted at~$X_j$. By assumption, $G$ fulfills the minimum degree in $U$ condition of \Cref{lem: rooted embedding with U}. Finally, let $\Lambda_j$ be the $G$-labelling of $(T_j,X_j)$ that sends the vertices of $\{v_1,\dots,v_d\}$ to the vertices of $e_j\in L$. By $\Delta(L)\le\gamma n$, each $S\subs V(G)$ with $\abs S\in[d-1]$ is contained in at most $\gamma n^{d-\abs S}$ edges of $L$ and, therefore, at most $\gamma n^{d-\abs S}$ many $\Lambda_j$ root~$S$.
			
			Thus, we can apply \Cref{lem: rooted embedding with U} with $\gamma_{\ref{lem: rooted embedding with U}}\coleq 2\gamma$ and get, for every $j\in[m]$ a $\Lambda_j$-faithful embedding of $(T_j,X_j)$ that are edge-disjoint, where $\phi_j(v)\in U$ for all $v\in V(T_j)\backslash X_j$ and where $\Delta(\bigcup_{j\in[m]}\phi_j(T_j))\le\gamma'n$. Let $\F\coleq\{\phi_j(T_j)\cup\{e_j\}:j\in[m]\}$. Then $\F$ is the collection of edge-disjoint copies of valid $\extratightcycle{f}{d}$ that fulfill \ref{enum: Cover L, cover in L cup G} and~\ref{enum: Cover L, max deg}. 
		\end{proof}
		\begin{proof}[Proof of \Cref{lem: super minus a little is super} \ref{enum: rich minus a little is rich}.] 
			Let $A\subs\binom{V(G)}{d-1}$ be an arbitrary set with $\abs A\le D$. Since $G$ is $(D,\rho)$-rich, we know that $\abs{\bigcap_{e\in A}G^{(d)}(e)\cap U}\ge \rho n$. Each $e\in A$ can be contained in at most $\gamma n$ sets of~$S$. Thus, $\abs{\bigcap_{e\in A}(G-S)^{(d)}(e)\cap U}\ge(\rho-D\gamma)n$.
		\end{proof}
		
		\subsection{Proof of the Degree Fixing Lemma~\ref{lem: fixing degree}}
		The goal of this section is to prove the Degree Fixing \Cref{lem: fixing degree}.
		
		To prove this, we will mainly use the following version of~\cite[Lemma 9.4]{glock2023existence}. On one hand, our version is simpler than \cite[Lemma 9.4]{glock2023existence} since it only considers the case $O=\emptyset$ and $G$ is assumed to be $F$-divisible. On the other hand, we have an additional condition that the $F$-copies that partition $H$ have at most $d$ vertices outside of a predefined vertex set~$U$. Note that the $D-D^\ast$ of \cite[Lemma 9.4]{glock2023existence} is an $H$ here and the $H\cup D^\ast$ of \cite[Lemma 9.4]{glock2023existence} is $G-H$ here. 
		
		\begin{lemma}\label{lem: make divisible}
			Let $1/n\ll \gamma \ll \xi,1/f^\ast$ and $d\in[f^\ast-1]$. Let $F$ be a $d$-graph. Let $F^\ast$ be a $d$-graph on $f^\ast$ vertices which has an $F$-decomposition.
			Let $G$ be an $F$-divisble $d$-graph on $n$ vertices and $U\subs V(G)$ such that for all $A\subs\binom{V(G)}{d-1}$ with $\abs A\le \binom{f^\ast-1}{d-1}$, we have $\abs{\bigcap_{S\in A}G(S)\cap U}\ge \xi n$. Then there exists a subgraph $H\subs G$ such that 
			\begin{enumerate}
				\item\label{enum: make divisible, Delta(H) small} $\Delta(H)\le \gamma^{-1}$;
				\item\label{enum: make divisible, G-H divisible} $G-H$ is $F^\ast$-divisible;
				\item\label{enum: make divisible, H has special F-decomposition} $H$ has an $F$-decomposition where each copy of $F$ has at most $d$ vertices outside of~$U$.
			\end{enumerate}
		\end{lemma}
		Since the proof of this lemma is very similar to the proof of \cite[Lemma 9.4]{glock2023existence}, \towriteornottowrite{we postpone the writeup of the proof to the appendix}{we omit it here}.
		
		Broadly speaking, the lemma enables us to turn an $\extratightcycle{f}{d}$-divisible supercomplex into an $\Fast_f$-divisible supercomplex. Thus, we only have to find a way to turn $G$ into an $\extratightcycle{f}{d}$-divisible supercomplex. 
		By \Cref{lem: Deg-vector of Fcycle and Fpath}, we know that we only need to adjust the number of edges and the 1-degrees of $G$ in order to make it $\extratightcycle{f}{d}$-divisible. The following lemma shows how this can be done by only deleting edges inside~$G[U]$.
		
		\begin{lemma}\label{lem: Fixing 1-degree}
			Let $1/n \ll \gamma\ll\rho\ll  1/D\ll 1/f, 1/d\le 1/2$. Let $G$ be a $d$-graph on $n$ vertices that is $(D,\rho)$-rich on a set $U\subs V(G)$. Furthermore, assume that $d^2\mid\deg_G(v)$ for all $v\in V(G)\backslash U$.
			
			Then there is a subset $S\subs G[U]$ with $\Delta(S)\le\gamma n$ such that $d^2\mid\deg_{G-S}(v)$ holds for all $v\in V(G)$ and where $(fd)\mid |G-S|$.
		\end{lemma}
		\begin{proof}
			
			Take any $d^2$-regular $d$-graph whose number of edges is divisible by $fd$ and let $\{v_1,\dots,v_d\}$ be one of its edges. Add a new vertex $v_0$ and replace $\{v_1,\dots,v_d\}$ by $\{v_0,v_2,v_3,\dots,v_d\}$. Let $H$ be the resulting graph. In this graph, all vertices have a degree that is divisible by $d^2$ except for $v_0$ whose degree is $1$ and $v_1$ whose degree is congruent to $-1$ mod~$d^2$. By embedding a copy of $H$ into $G$ and putting its edges into $S$, we can decrease the degree of the image of $v_0$ in $G-S$ by 1 mod $d^2$ and increase the degree of the image of $v_1$ in $G-S$ by $1$ mod~$d^2$. In this way, we can adjust the degree of every vertex in~$U$ like we did it in the proof of \Cref{theo: precise value of H(n d)}.
			
			First, we take up to $fd$ arbitrary edges of $G[U]$ into $S$ such that $(fd)\mid |G-S|$. In particular, the $d$-uniform Handshake Lemma implies now that the sum of vertex degrees in $G-S$ is divisible by~$d^2$. After this step, we want to adjust the degree by only moving copies of $H$ into~$S$. In the following set $T$, we collect all pairs $(u,w)$ where we want to embed a copy of $H$ in such a way that $v_0$ is mapped to $u$ and $v_1$ is mapped to~$w$. If we move this copy of $H$ into $S$, then the degree of $u$ in $G-S$ will decrease by $1\pmod{d^2}$ whereas the degree of $w$ will increase by $1\pmod{d^2}$.
			
			\textbf{Claim.} There is a digraph $T$ on the vertex set $V(G)$ such that the following conditions hold:
			\begin{enumerate}
				\item for all $v\in V(G)$, $\deg_{G-S}(v)-\first_T(v)+\second_T(v) \equiv 0\pmod{d^2}$;
				\item for all $v\in V(G)$, $\first_T(v)+\second_T(v)\le 5d^2$.
			\end{enumerate}
			\begin{claimproof}
				The claim is proved in the same way as the claim in the proof of \Cref{theo: precise value of H(n d)}.
			\end{claimproof}
			
			Let $T$ be the digraph given by the claim with an arbitrary edge ordering. We will apply \Cref{lem: rooted embedding with U} with \towriteornottowrite{$\alpha \coleq 1$, }{}$m\coleq \abs T\le 5d^2 n$, $\gamma'_{\ref{lem: rooted embedding with U}}\coleq \gamma/2$, $\gamma_{\ref{lem: rooted embedding with U}}$ suitable small, and $G_{\ref{lem: rooted embedding with U}}\coleq G-S$. If $(u,w)$ is the $j$-th edge of $T$, we have $T_j\coleq H$, $X_j=\{v_0,v_1\}$, $\Lambda_j(v_0)=u$, and $\Lambda(v_1)=w$. One can easily see that all conditions of \Cref{lem: rooted embedding with U} are satisfied.
			
			Thus, we get $\Lambda_j$-faithful embeddings $\phi_j$ of $(T_j,X_j)$ into $U$ such that the images are pairwise edge-disjoint and also edge-disjoint to the at most $3d^2$ edges in~$S$. Furthermore, we have $\Delta(\bigcup_{j\in[m]}\phi_j(T_j))\le\gamma n/2$. By adding the edges of all the images to $S$, we get that all 1-degrees in $G-S$ are divisible by $d^2$, $(fd)\mid |G-S|$, and $\Delta(S)\le\gamma n$.
		\end{proof}
		
		Now, we are ready to prove the Degree Fixing Lemma. Here, we only need to apply the previous lemma to make $G$ $\extratightcycle{f}{d}$-divisible and then \Cref{lem: make divisible} to make that graph $\Fast_f$-divisible.
		
		\begin{proof}[Proof of the Degree Fixing \Cref{lem: fixing degree}]
			By \Cref{lem: Fixing 1-degree}, there is a subset $S\subs G[U]$ with $\Delta(S)\le \gamma n$ such that for $G'\coleq G-S$, we have $d^2\mid\deg_{G'}(v)$ holds for all $v\in V(G)$, and where $(fd)\mid |G'|$. This, together with \Cref{lem: Deg-vector of Fcycle and Fpath}, shows that $G'$ is $\extratightcycle{f}{d}$-divisible. By \Cref{lem: super minus a little is super}, we get that $G'$ is $(D,\rho/2)$-rich.
			
			Thus, we can apply \Cref{lem: make divisible} with $F_{\ref{lem: make divisible}}\coleq \extratightcycle{f}{d}$, $F^\ast_{\ref{lem: make divisible}}\coleq \Fast_f$, $\xi_{\ref{lem: make divisible}}\coleq\rho/2$, and $G_{\ref{lem: make divisible}}\coleq G'$ using that $G'$ is $(D,\rho/2)$-rich. We get a subgraph $H\subs G'^{(d)}$ such that $\Delta(H)\le \gamma^{-1}$, $G'-H$ is $\Fast_f$-divisible, and $H$ has an $\extratightcycle{f}{d}$-decomposition where each copy of $F$ has at most $d$ vertices outside of~$U$.
		\end{proof}
		
		\subsection{Proof of the Existence of a Skew Supercomplex (Lemma~\ref{lem:skew_super_complex})}\label{sec: Existence of Supercomplex}
		The final step is to prove \Cref{lem:skew_super_complex}, i.e.\@ the existence of a $(D,\rho)$-rich, $(\eps,\xi,f,d)$-supercomplex $\Gsuper\subs \Gskew$ with $\Delta\big(\Gsuper^{(d)}\big)\le\eta n$. First, we show that we can ignore the maximum degree condition because we get it for free in the end. This is done by the following previously known result: 
		
		\begin{lemma}[\cite{glock2023existence}, Corollary 5.19]\label{lem: supercomplex sparsification}
			Let $1/n\ll \eps,\xi\ll \eta,1/f$ and $d\in[f-1]$. Suppose that $G$ is an $(\eps,\xi,f,d)$-supercomplex on $n$ vertices and that $H\subs G^{(d)}$ is a random subgraph obtained by including every edge of $G^{(d)}$ independently with probability~$\eta$. Then with high probability, $G[H]$ is a $(4\eps,\xi^2,f,d)$-supercomplex.
		\end{lemma}
		
		\begin{corollary}\label{cor: supercomplex sparsification}
			Let $1/n\ll \eps,\xi,\rho\ll \eta,1/D,1/f\le 1/d$. Suppose that $G$ is an $(\eps,\xi,f,d)$-supercomplex on $n$ vertices which is $(D,\rho)$-rich in $U\subs V(G)$. Then there is a subcomplex $G'\subs G$ that is a $(4\eps,\xi^2,f,d)$-supercomplex, $(D,\rho^2)$-rich in $U$ and where $\Delta(G'^{(d)})\le \eta n$.
		\end{corollary}
		\begin{proof}
			Let $H\subs G^{(d)}$ be a random subgraph obtained by including every edge of $G^{(d)}$ independently with probability~$\eta/2$. 
			By \Cref{lem: supercomplex sparsification}, $G[H]$ is a $(4\eps,\xi^2,f,d)$-supercomplex with high probability. Furthermore, each $(d-1)$-set is expected to have a degree of at most $(\eta/2)n$. Chernoff and union bound imply that $\Delta(G[H]^{(d)})\le \eta n$ holds with high probability as well. Finally, for each $A\subs \binom{V(G)}{d-1}$ with $\abs{A}\le D$, we have $\abs{\bigcap_{S\in A}G^{(d)}(S)\cap U}\ge\rho n$. Thus, we expect that $\abs{\bigcap_{S\in A}G[H]^{(d)}(S)\cap U}\ge \eta^D\rho n$. Using Chernoff and union bound again, we get that $G[H]$ is $(D,\rho^2)$-rich with high probability.
		\end{proof}
		
		Thus, we just have to find a $(D,\rho)$-rich, $(\eps, \xi,f,d)$-supercomplex in~$\Gskew$. To do this, we have to consider the precise definition of a supercomplex. This can also be found in~\cite{glock2023existence}. We repeat it here for completeness. 
		
		The crucial property appearing in the next definition is that of \textit{regularity}, which means that every $d$-set of a given complex $G$ is contained in roughly the same number of $f$-sets. If we view $G$ as a complex which is induced by some $d$-graph, this means that every edge lies in roughly the same number of cliques of size~$f$.
		
		\begin{definition}[\cite{glock2023existence}, Definition 4.1]\label{def: complex}
			Let $G$ be a complex on $n$ vertices, $f \in \mathbb{N}$ and $d \in [f - 1]_0$, with $0 \leq \eps, g, \xi \leq 1$. We say that $G$ is
			\begin{itemize}
				\item[(i)] $(\eps, g, f, d)$-\textit{regular}, if for all $e \in G^{(d)}$, we have 
				\[
				\abs{G^{(f)}(e)} = (g \pm \eps)n^{f - d};
				\]
				\item[(ii)] $(\xi, f, d)$-\textit{dense}, if for all $e \in G^{(d)}$, we have
				\[
				\abs{G^{(f)}(e)} \geq \xi n^{f - d};
				\]
				\item[(iii)] $(\xi, f, d)$-\textit{extendable}, if $G^{(d)}$ is empty or there exists a subset $X \subseteq V(G)$ with $\abs X \geq \xi n$ such that for all $e \in \binom{X}{d}$, there are at least $\xi n^{f - d}$ many $(f - d)$-sets $Q \subseteq V(G) \backslash e$ such that $\binom{Q \cup e}{d} \backslash \{e\} \subseteq G^{(d)}$.
			\end{itemize}
			We say that $G$ is a \textit{full} $(\eps, \xi, f, d)$-\textit{complex} if $G$ is
			\begin{itemize}
				\item $(\eps, g, f, d)$-regular for some $g \geq \xi$,
				\item $(\xi, f + d, d)$-dense,
				\item $(\xi, f, d)$-extendable.
			\end{itemize}
			
			We say that $G$ is an $(\eps, \xi, f, d)$-\textit{complex} if there exists an $f$-graph $Y$ on $V(G)$ such that $G[Y]$ is a full $(\eps, \xi, f, d)$-complex. Note that $G[Y]^{(d)} = G^{(d)}$ (recall that $d < f$). 
		\end{definition}
		
		\begin{definition}[\cite{glock2023existence}, Definition 4.3]\label{def: supercomplex}
			Let $G$ be a complex. We say that $G$ is an \emph{$(\eps,\xi,f,d)$-supercomplex} if for every $h \in [d]_0$ and every set $B \subseteq G^{(h)}$ with $1 \le \abs B \le 2^h$, we have that $\bigcap_{b \in B} G(b)$ is an $(\eps,\xi,f-h,d-h)$-complex. 
		\end{definition}
		
		\begin{definition}
			Let $G$ be a complex and $U\subs V(G)$. We say that a set $e\in G$ has \emph{type $i$} with respect to $U$ if $\abs{e\cap (V(G)\backslash U)}=i$. By $G_i$, we denote the set of sets of $G$ with type~$i$.
		\end{definition}
		
		We start by proving a lemma about~$\Gskew$. Morally, it says that for each vertex set $A$, number $s\in[\abs A,d+f]$, and type $j\in[d]_0$, there are roughly as many ways to extend $A$ to an $s$-set of type $j$
		in $\Gskew$ as one would expect. 
		Since the definition of supercomplexes requires that every $\bigcap_{b\in B}G(b)$ has certain properties, we not only show it for $\Gskew$, but also for every such intersection.
		\begin{lemma}\label{lem: Gskew has two properties}
			Let $1/n\ll\alpha , \mu \ll 1/f \ll 1/d\le 1/2$. Let $G$ be a $d$-graph on $n$ vertices and $U\subseteq V(G)$ with $\abs U=\mu n$ and $\delta(G)\ge (1-\alpha)n$ and $\Gskew$ the skew $(G,f,d)$-complex. Then the following properties are fulfilled for all $h\in[d]_0$, $B\subs \Gskew^{(h)}$ with $1\le\abs B\le 2^h$ and $G''\coleq\bigcap_{b\in B}\Gskew(b)$, $d'\coleq d-h$, $f'\coleq f-h$, $\alpha'\coleq f\mu^{-1}(2^{d+1}\cdot d)^d\alpha$, $n'\coleq n-\abs{\bigcup_{b\in B}b}$:
			\begin{enumerate}
				\item\label{enum: Every possible set can be extended} for every vertex set $A\subs V(G'')$ of size $a\in[d',d'+f']$ and type $k\in[d']_0$, every $s\in[a,f'+d']$ and every $j\in[k,d']$, there are
				\[\big(1\pm\alpha'\big)\frac{(1-\mu)^{j-k}}{(j-k)!}\cdot\frac{\mu^{s-j-a+k}}{(s-j-a+k)!}n'^{s-a}\]
				vertex sets $S\subs V(G'')$ containing $A$ of size $s$ and type $j$ such that $\binom{S}{d'}\backslash\binom{A}{d'}\subs G''^{(d')}$.
				\item\label{enum: Every possible edge can be extended} for every edge $A\in G''$ of size $a\in[d',d'+f']$ and type $k\in[d']_0$, every $s\in[a,f'+d']$ and every $j\in[k,d']$, there are
				\[\big(1\pm\alpha'\big)\frac{(1-\mu)^{j-k}}{(j-k)!}\cdot\frac{\mu^{s-j-a+k}}{(s-j-a+k)!}n'^{s-a}\]
				vertex sets $S\subs V(G'')$ containing $A$ of size $s$ and type $j$ such that $S\in G''^{(s)}$.
			\end{enumerate}
		\end{lemma}
		\begin{proof}
			Fix an $h\in[d]_0$, $B\subs\Gskew^{(h)}$ with $1\le\abs B\le 2^h$ and define $G'', d', f', \alpha', n'$ as in the statement of the lemma.
			
			To check \ref{enum: Every possible set can be extended}, take a set $A\subs V(G'')$ of size $a\in[d',d'+f']$ and type $k\in[d']_0$. Furthermore, fix an $s\in[a,f'+d']$ and a $j\in[k,d']$. Clearly, there are at most 
			\[\frac{\big((1-\mu)n'\big)^{j-k}}{(j-k)!}\cdot\frac{(\mu n')^{s-j-a+k}}{(s-j-a+k)!}\] vertex sets $S\subs V(G'')$ that contain $A$ and are of size $s$ and type~$j$. Hence, we only have to lower bound the number of vertex sets $S\subs V(G'')$ containing $A$ of size $s$ and type $j$ such that $\binom{S}{d'}\backslash\binom{A}{d'}\subs G''^{(d')}$. 
			
			We choose the vertices of $S$ one at a time. Suppose we have already picked $m$ vertices (including those of $A$). We have to pick the next vertex $v$ in the following way: For every $b\in B$ and every $(d-h-1)$-subset $T$ of the already chosen vertices, the set $b\cup T\cup\lbrace v\rbrace$ must be in $\Gskew^{(d)}$. By the minimum degree condition of $G$, we get that each choice of $b$ and $T$ can exclude at most $\alpha n$ vertices. Hence and because of $\abs{\bigcup_{b\in B}b}\le 2^h\cdot h$, there are definitely at least $\Big(1-\mu-\binom{m+2^h\cdot h}{d-1}\alpha\Big)n\ge \big(1-\mu-(2^{d+1}\cdot d)^{d}\alpha\big)n'$ vertices outside of $U$ which we can pick next. Similarly, there are at least $\big(\mu-(2^{d+1}\cdot d)^{d}\alpha\big)n'$ vertices inside of~$U$. Therefore, the number of vertex sets $S$ with the desired properties is at least
			\begin{align*}
				&\frac{\Big(\big(1-\mu-(2^{d+1}\cdot d)^{d}\alpha\big)n'\Big)^{j-k}}{(j-k)!}\cdot \frac{\Big(\big(\mu-(2^{d+1}\cdot d)^{d}\alpha\big)n'\Big)^{s-j-a+k}}{(s-j-a+k)!}\\
				&\ge \frac{\big(\big(1-\mu^{-1}(2^{d+1}\cdot d)^{d}\alpha\big)(1-\mu)\big)^{j-k}}{(j-k)!}\cdot \frac{\big(\big(1-\mu^{-1}(2^{d+1}\cdot d)^{d}\alpha\big)\mu\big)^{s-j-a+k}}{(s-j-a+k)!}n'^{s-a}\\
				&\ge \big(1-(s-a)\mu^{-1}(2^{d+1}\cdot d)^{d}\alpha\big) \frac{(1-\mu)^{j-k}}{(j-k)!}\cdot \frac{\mu^{s-j-a+k}}{(s-j-a+k)!}n'^{s-a}
			\end{align*}
			In the last step, we used Bernoulli's inequality.
			Therefore, the number is 
			\[(1\pm \alpha') \frac{(1-\mu)^{j-k}}{(j-k)!}\cdot \frac{\mu^{s-j-a+k}}{(s-j-a+k)!}n'^{s-a}.\]
			Condition \ref{enum: Every possible edge can be extended} is checked similarly: Let $A\in G''$ be an edge of size $a\in[d',d'+f']$ and type $k\in[d']_0$ and fix an $s\in[a,f'+d']$ and a $j\in[k,d']$. There are at most 
			\[\frac{\big((1-\mu)n'\big)^{j-k}}{(j-k)!}\cdot\frac{(\mu n')^{s-j-a+k}}{(s-j-a+k)!}\] edges $S\in G''$ that contain $A$ and are of size $s$ and type~$j$. Hence, we only have to lower bound the number of edges $S\in G''$ containing $A$ of size $s$ and type~$j$.
			
			We choose the vertices of $S$ one at a time. Suppose we have already picked $m$ vertices (including those of $A$) such that these $m$ vertices form an edge $S'$ in~$G''$. Then we want to pick the next vertex $v$ in such a way that $S''\coleq S'\cup \lbrace v\rbrace$ is also an edge in~$G''$. This is the case if for every $b\in B$, $S''\cup b$ is in~$\Gskew$. By the definition of $\Gskew$, this is the case if $\binom{S''\cup b}{d}\subs\Gskew^{(d)}$. Since $S'$ is in $G$, every $d$-subset of $S''\cup b$ that does not contain $v$ is already known to be in $\Gskew^{(d)}$. We have to choose $v$ such that for every $T\in \binom{S'\cup b}{d-1}$, $T\cup\lbrace v\rbrace\in\Gskew^{(d)}$. By the minimum degree condition of $\Gskew$, each choice of $T$ can exclude at most $\alpha n$ vertices. Hence and because of $\abs{\bigcup_{b\in B}b}\le 2^h\cdot h$, there are definitely at least $\Big(1-\mu-\binom{m+2^h\cdot h}{d-1}\alpha\Big)n\ge \Big(1-\mu-\big(2^{d+1}\cdot d\big)^{d}\alpha\Big)n'$ vertices outside of $U$ which we can pick next. Similarly, there are at least $\Big(\mu-\big(2^{d+1}\cdot d\big)^{d}\alpha\Big)n'$ vertices inside of~$U$. The remaining calculations are exactly as in the proof of~\ref{enum: Every possible set can be extended}.
		\end{proof}
		
		In the following definition, we will list all the properties we need to show that a complex is an $(\eps,\xi,f,d)$-complex for certain $\eps$ and~$\xi$.
		
		\begin{definition}
			We say that a complex $G$ on $n$ vertices is \emph{$(\alpha,\mu,f,d,c,C)$-advantageous} if the following properties are fulfilled with $w_j\coleq\big(\frac{\mu}{3fd}\big)^j$ for all $j\in[d]_0$:
			
			\begin{itemize}
				\item for all $i\in[d]_0$, $e\in G^{(d)}_i$ and $\ell\in[i,d]$
				\begin{equation}
					\abs{G^{(f+d)}_\ell(e)}\ge c\cdot \mu^{f-\ell+\ell\binom{f+d-1}{d-1}}\cdot n^{f}\label{eq: size of G^(f+d)(e)}
				\end{equation}
				\item for all $i\in [d]_0$, $Q\in G^{(f)}_i$ and $\ell\in[i,d]$
				\begin{equation}\label{eq: size of G^(f+d)(Q)}
					\abs{G^{(f+d)}_{\ell}(Q)}\le C\cdot \mu^{f-\ell+i-i\cdot\binom{f-1}{d-1}+\ell\cdot\binom{f+d-1}{d-1}}\cdot n^{d}
				\end{equation}
				\item 
				for all $i\in[d]_0$, $e\in G^{(d)}_i$ and $\ell\in[i,d]$
				\begin{equation}
					\abs{G^{(f)}_\ell(e)}=\big(1\pm\alpha\big)\frac{(1-\mu)^{\ell-i}}{(\ell-i)!}\cdot\frac{\mu ^{f-\ell-d+i}}{(f-\ell-d+i)!}\cdot w_i^{-1}\prod_{t=0}^dw_t^{\binom{\ell}{t}\binom{f-\ell}{d-t}}\cdot n^{f-d}\label{eq: size of G^(f)(e)}
				\end{equation}
				
				\item for all $i\in[d]_0$, $e\in\binom{V(G)}{d}$ of type $i$, and $\ell\in[i,d]$, the number of vertex sets $S$ of type $\ell$ and size $f$ containing $e$ such that $\binom{S}{d}\backslash\{e\}\subs G^{(d)}$ is at least
				\begin{equation}
					c\cdot \mu^{f-\ell-d+\ell\cdot\binom{f-1}{d-1}}\cdot n^{f-d}.\label{eq: size of E_f,ell(e) imprecise}
				\end{equation}
			\end{itemize}
		\end{definition}
		
		The $G'$ in the following lemma will be our final $(\eps,\xi, f,d)$-supercomplex.
		
		\begin{lemma}\label{lem: Gskew has a subcomplex full of advantages}
			Let $1/n\ll\rho\ll \alpha , \mu \ll c,1/C\ll 1/D,1/f \ll 1/d\le 1/2$ with $\alpha<\mu^{1.1}$. Let $G$ be a $d$-graph on $n$ vertices and $U\subseteq V(G)$ with $\abs U=\mu n$ and $\delta(G)\ge (1-\alpha)n$. Then there is a subcomplex $G'\subs\Gskew$ that is $(D,\rho)$-rich in $U$ and such that for all $h\in[d]_0$, $B\subs G'^{(h)}$ with $1\le\abs B\le 2^h$, $G''\coleq\bigcap_{b\in B}G'(b)$ is $(2\alpha',\mu,f',d',c,C)$-advantageous on $n'$ vertices where $d'\coleq d-h$, $f'\coleq f-h$, $\alpha'\coleq f\mu^{-1}(2^{d+1}\cdot d)^d\alpha$, $n'\coleq n-\abs{\bigcup_{b\in B}b}$.
		\end{lemma}
		\begin{proof}
			We create a random set $X$ of $d$-edges in the following way.
			For all $i\in[d]_0$, any edge in $\Gto d_i$ is taken into $X$ with probability $w_i\coleq\big(\frac{\mu}{3fd}\big)^i$ independently of each other.
			Let $G'\coleq \Gskew[X]$. We will show now that with high probability, $G'$ will fulfill all desired properties. 
			
			First, we show $(D,\rho)$-richness in~$U$.
			For this, let $A\subs\binom{V(G')}{d-1}$ with $\abs{A}\le D$. We need to show $\abs{\bigcap_{S\in A}G'^{(d)}(S)\cap U}\ge\rho n$. 
			Indeed, for each $e\in A$, there are at most $\alpha n$ elements in $U$ such that $e\cup \{u\}\not\in G$ by the minimum degree condition of~$G$. Therefore, $\abs{\bigcap_{S\in A}G^{(d)}(S)\cap U}\ge(\mu-\abs A\alpha)n$. 
			Thus, the expected size of $\bigcap_{S\in A}G'^{(d)}(S)\cap U$ is at least 
			\[w_d^{\abs A}(\mu-\abs A\alpha)n\ge \!\left(\frac\mu{3fd}\right)^{dD}(\mu-D\alpha)n.\]
			By Chernoff's bound (\Cref{theo: Chernoff}) and a union bound over the polynomially many choices of $A$, we get that with high probability, $G'$ is $(D,\rho)$-rich in~$U$.
			
			Next, let $h\in[d]_0$, $B\subs G'^{(h)}$ with $1\le\abs B\le 2^h$ and $G''\coleq\bigcap_{b\in B}G'(b)$, $d'\coleq d-h$, $f'\coleq f-h$, $\alpha'\coleq f\mu^{-1}(2^{d+1}\cdot d)^d\alpha$, $n'\coleq n-\abs{\bigcup_{b\in B}b}$. 
			We will show that $G''$ is $(2\alpha', \mu, f', d',c,C)$-advantageous. Note that $1/n'\ll \alpha',\mu\ll1/f\ll 1/d$ holds because $\alpha<\mu^{1.1}$.
			
			Let $A\subs V(G'')$ be a vertex set of size $a\in[d',d'+f']$ and type $k\in[d']_0$. Let $s\in[a,f'+d']$ and $j\in[k,d']$. We define $E_{s,j}(A)$ to be the number of vertex sets $S\subs V(G'')$ of size $s$ and type $j$ containing $A$ such that $\binom S{d'}\backslash \binom A{d'}\subs G''^{(d')}$. By \Cref{lem: Gskew has two properties} \ref{enum: Every possible set can be extended}, we get
			
			\[\E[E_{s,j}(A)]=\big(1\pm\alpha'\big)\frac{(1-\mu)^{j-k}}{(j-k)!}\cdot\frac{\mu^{s-j-a+k}}{(s-j-a+k)!}\cdot \prod_{t=0}^{d'} w_t^{-\binom{k}{t}\binom{a-k}{d'-t}}\prod_{t=0}^{d'}w_t^{\binom{j}{t}\binom{s-j}{d'-t}}n'^{s-a}.\]
			Similarly, if $A$ is even an edge in $\bigcap_{b\in B}\Gskew(b)$, we can define $E'_{s,j}(A)$ to be the number of edges $S\in G''^{(s)}$ of type $j$ containing~$A$. 
			Then we get with the help of property \Cref{lem: Gskew has two properties} \ref{enum: Every possible edge can be extended}
			\[\E\!\left[E'_{s,j}(A)|A\in G''\right]=\big(1\pm\alpha'\big)\frac{(1-\mu)^{j-k}}{(j-k)!}\cdot\frac{\mu^{s-j-a+k}}{(s-j-a+k)!}\cdot \prod_{t=0}^{d'} w_t^{-\binom{k}{t}\binom{a-k}{d'-t}}\prod_{t=0}^{d'}w_t^{\binom{j}{t}\binom{s-j}{d'-t}}n'^{s-a}.\]
			
			Therefore and by Chernoff (\Cref{theo: Chernoff}), with high probability, the following hold for all $h\in[d]_0$, $B\subs G'^{(h)}$ with $1\le\abs B\le 2^h$ and $G''\coleq\bigcap_{b\in B}G'(b)$ simultaneously\footnote{In the following equations, we frequently use the identity $\sum_{j=0}^dj\cdot\binom{i}{j}\binom{f-i}{d-j}=i\cdot\binom{f-1}{d-1}$ which holds since both count the number of ways to pick a committee with $d$ members out of $i$ women and $f-i$ men and then choose a woman from the committee as president.}:
			
			\setlength{\FrameSep}{0pt}
			\noindent\begin{minipage}{0.13\linewidth}
				\begin{framed}
					\begin{align*}
						a&=d'\\
						k&=i\\
						s&=f'+d'\\
						j&=\ell
					\end{align*}
				\end{framed}
			\end{minipage}%
            \begin{minipage}{0.86\linewidth}
				\begin{align*}
					\abs{G''^{(f'+d')}_\ell(e)}&=\big(1\pm2\alpha'\big)\frac{(1-\mu)^{\ell-i}}{(\ell-i)!}\cdot\frac{\mu^{f'-\ell+i}}{(f'-\ell+i)!}\cdot w_i^{-1}\prod_{t=0}^{d'}w_t^{\binom{\ell}{t}\binom{f'+d'-\ell}{d'-t}}\cdot n'^{f'}\nonumber\\&=\Theta_{d',f'}\!\left(\mu^{f'-\ell+i}\cdot\mu^{-i}\cdot\mu^{\sum_{t=0}^{d'}t\binom{\ell}{t}\binom{f+d'-\ell}{d'-t}}\cdot n'^{f'}\right)\nonumber\\&=\Theta_{d',f'}\!\left(\mu^{f'-\ell+\ell\binom{f'+d'-1}{d'-1}}\cdot n'^{f'}\right)
				\end{align*}
			\end{minipage}
			
			\noindent for all $i\in[d']_0$, $e\in G''^{(d')}_i$ and $\ell\in[i,d']$;
			
			\noindent\begin{minipage}{0.13\linewidth}
				\begin{framed}
					\begin{align*}
						a&=f'\\
						k&=i\\
						s&=f'+d'\\
						j&=\ell
					\end{align*}
				\end{framed}
			\end{minipage}\begin{minipage}{0.86\linewidth}
				\begin{align*}
					\abs{G''^{(f'+d')}_{\ell}(Q)}&=\big(1\pm2\alpha'\big)\frac{(1-\mu)^{\ell-i}}{(\ell-i)!}\cdot\frac{\mu^{f'-\ell+i}}{(d'-\ell+i)!}\prod_{t=0}^{d'}w_t^{-\binom{i}{t}\binom{f-i}{d'-t}}\prod_{t=0}^{d'}w_t^{\binom{\ell}{t}\binom{f'+d'-\ell}{d'-t}}n'^{d'}\nonumber\\
					&=\Theta_{d',f'}\!\left(\mu^{f'-\ell+i}\cdot\mu^{-\sum_{t=0}^{d'}t\binom{i}{t}\binom{f'-i}{d'-t}}\cdot\mu^{\sum_{t=0}^{d'}t\cdot\binom{\ell}{t}\binom{f'+d'-\ell}{d'-t}}\cdot n'^{d'}\right)\nonumber\\
					&=\Theta_{d',f'}\!\left(\mu^{f'-\ell+i-i\cdot\binom{f'-1}{d'-1}+\ell\cdot\binom{f'+d'-1}{d'-1}}\cdot n'^{d'}\right)
				\end{align*}
			\end{minipage}
			
			\noindent for all $i\in [d']_0$, $Q\in G''^{(f')}_i$ and $\ell\in[i,d']$;
			
			\noindent \begin{minipage}{0.1\linewidth}
				\begin{framed}
					\begin{align*}
						a&=d'\\
						k&=i\\
						s&=f'\\
						j&=\ell
					\end{align*}
				\end{framed}
			\end{minipage}\begin{minipage}{0.89\linewidth}
				\begin{align*}
					\abs{G''^{(f')}_\ell(e)}&=\big(1\pm2\alpha'\big)\frac{(1-\mu)^{\ell-i}}{(\ell-i)!}\cdot\frac{\mu^{f'-\ell-d'+i}}{(f'-\ell-d'+i)!}\cdot w_i^{-1}\prod_{t=0}^{d'}w_t^{\binom{\ell}{t}\binom{f'-\ell}{d'-t}}\cdot n'^{f'-d'}
				\end{align*}
			\end{minipage}
			
			\noindent for all $i\in[d']_0$, $e\in G''^{(d')}_i$ and $\ell\in[i,d']$;
			
			\noindent \begin{minipage}{0.1\linewidth}
				\begin{framed}
					\begin{align*}
						a&=d'\\
						k&=i\\
						s&=f'\\
						j&=\ell
					\end{align*}
				\end{framed}
			\end{minipage}\begin{minipage}{0.89\linewidth}
				\begin{align*}
					\abs{E_{f',\ell}(e)}&=\big(1\pm2\alpha'\big)\frac{(1-\mu)^{\ell-i}}{(\ell-i)!}\cdot\frac{\mu ^{f'-\ell-d'+i}}{(f'-\ell-d'+i)!}\cdot w_i^{-1}\prod_{t=0}^{d'}w_t^{\binom{\ell}{t}\binom{f'-\ell}{d'-t}}\cdot n'^{f'-d'}\nonumber\\
					&=\Theta_{d',f'}\!\left(\mu^{f'-\ell-d'+i}\cdot\mu^{-i}\cdot\mu^{\sum_{t=0}^{d'} t\binom{\ell}{t}\binom{f'-\ell}{d'-t}}\cdot n'^{f'-d'}\right)\nonumber\\
					&=\Theta_{d',f'}\!\left(\mu^{f'-\ell-d'+\ell\cdot\binom{f'-1}{d'-1}}\cdot n'^{f'-d'}\right)
				\end{align*}
			\end{minipage}
			
			\noindent for all $i\in[d']_0$, $e\in \binom{V(G'')}{d'}$ of type $i$, and $\ell\in[i,d']$.
		\end{proof}
		All that is left to do now, is to show that the fact that all the $\bigcap_{b\in B}G'(b)$ are advantageous implies that they all are $(\eps,\xi,f,d)$-complexes for certain $\eps$ and~$\xi$.
		In that proof, we use the following two results:
		
		\begin{proposition}[\cite{glock2023existence}]\label{prop: gadget}
			Let $f>d\ge 2$ and let $e$ and $J$ be disjoint sets with $\abs e=d$ and $\abs J=f$. Let $G$ be the complete complex on $e\cup J$. There exists a function $\psi\colon \Gto f\to \mathbb R$ such that
			\begin{enumerate}
				\item for all $e'\in \Gto d$, $\sum_{Q\in\Gto f(e')}\psi(Q\cup e')=\begin{cases}
					1,&e'=e\\0,&e'\ne e;
				\end{cases}$
				\item for all $Q\in\Gto f$, $\abs{\psi(Q)}\le \frac{2^{d-j}(d-j)!}{\binom{f-d+j}{j}}$, where $j\coleq\abs{e\cap Q}$.
			\end{enumerate}
		\end{proposition}
		
		\begin{lemma}\label{lem:linear algebra}
			Consider the upper-triangular $n\times n$-matrix
			\[A\coleq\begin{pmatrix}a_{11}&a_{12}&\dots&a_{1n}\\
				0&a_{22}&\dots&a_{2n}\\
				\vdots&\vdots&\ddots&\vdots\\
				0&0&\dots&a_{nn}\end{pmatrix}\]
			with non-negative entries, positive diagonal entries and such that $\frac{a_{ij}}{a_{jj}}\le\frac{1}{2n}$ for all $i<j$.
			Then, \[A\begin{pmatrix}
				x_1\\x_2\\\vdots\\x_n
			\end{pmatrix}=\begin{pmatrix}
				1\\1\\\vdots\\1
			\end{pmatrix}\]
			has a solution with $\frac{1}{2a_{ii}}\le x_i\le \frac{1}{a_{ii}}$ for all $i\in[n]$.
		\end{lemma}
		\begin{proof}
			We prove the statement by induction from $i=n$ down to 1. For $x_n$, this holds by the last equation.
			
			Suppose, it has already been shown for $x_n,\dots,x_{i+1}$. Then, the $i$-th equation yields
			\[x_i=\frac{1-\sum_{k=i+1}^{n}a_{ik}x_k}{a_{ii}}.\]
			By assumption and induction hypothesis, $a_{ik}x_k$ is non-negative, whence $x_i\le\frac{1}{a_{ii}}$. Furthermore, we get 
			\[x_i=\frac{1-\sum_{k=i+1}^{n}a_{ik}x_k}{a_{ii}}\ge \frac{1-\sum_{k=i+1}^{n}\frac{a_{ik}}{a_{kk}}}{a_{ii}}\ge\frac{1-\frac{n-i}{2n}}{a_{ii}}\ge \frac{1}{2a_{ii}}.\qedhere\]
		\end{proof}
		
		We are now ready to show that the $G'$ from \Cref{lem: Gskew has a subcomplex full of advantages} is indeed an $(\eps,\xi,f,d)$-supercomplex.
		
		\begin{lemma}\label{lem:advantageous implies complex}
			Let $1/n\ll\eps\ll \xi\ll  \alpha , \mu \ll c,1/C\ll 1/f \ll 1/d\le 1/2$. Let $G$ be a complex on $n$ vertices which is $(\alpha,\mu,f,d,c,C)$-advantageous. Then $G$ is an $(\eps,\xi,f,d)$-complex.
		\end{lemma}
		\begin{proof}
			Recall that $w_j=\big(\frac{\mu}{3fd}\big)^j$ for all $j\in[d]_0$ by the definition of $(\alpha,\mu,f,d,c,C)$-advantageous. Consider the $(d+1)\times (d+1)$-matrix $A=(a_{i\ell})_{i,\ell\in[d]_0}$ with
			\[ a_{i\ell}\coleq\begin{cases}
				\frac{(1-\mu)^{\ell-i}}{(\ell-i)!}\cdot\frac{\mu^{f-\ell-d+i}}{(f-\ell-d+i)!}\cdot w_i^{-1}\prod_{j=0}^dw_j^{\binom{\ell}{j}\binom{f-\ell}{d-j}}&\text{if }i\le \ell\\
				0&\text{if }\ell< i.
			\end{cases}\]
			This matrix is upper-triangular with non-negative entries. For $i\in[d]_0$, we get 
			\begin{align*}
				a_{ii}&=\frac{\mu^{f-d}}{(f-d)!}\cdot w_i^{-1}\prod_{j=0}^dw_j^{\binom{i}{j}\binom{f-i}{d-j}}=\frac{\mu^{f-d}}{(f-d)!}\cdot \!\left(\frac{3fd}{\mu}\right)^i\cdot\!\left(\frac{\mu}{3fd}\right)^{\sum_{j=0}^dj\cdot\binom{i}{j}\binom{f-i}{d-j}}\\
				&=\frac{\mu^{f-d}}{(f-d)!}\cdot \!\left(\frac{3fd}{\mu}\right)^i\cdot\!\left(\frac{\mu}{3fd}\right)^{i\cdot\binom{f-1}{d-1}}.
			\end{align*}
			We can conclude
			\begin{equation}\label{eq:bounds on a_ii}
				a_{ii}=\Theta_{f,d}\!\left(\mu^{f-d-i+i\cdot\binom{f-1}{d-1}}\right).
			\end{equation}
			Furthermore,
			\begin{align*}
				\frac{a_{i\ell}}{a_{\ell\ell}}&=\frac{(1-\mu)^{\ell-i}}{(\ell-i)!}\cdot\mu^{i-\ell}\cdot\frac{(f-d)!}{(f-d+i-\ell)!}\cdot\frac{w_\ell}{w_i}\le\mu^{i-\ell}\cdot f^{\ell-i}\cdot\frac{w_\ell}{w_i}=(3d)^{i-\ell}\le\frac{1}{2(d+1)}
			\end{align*}
			for all $i<\ell$. By \Cref{lem:linear algebra}, the linear equations
			\[A\begin{pmatrix}
				g_0'\\g_1'\\\vdots\\g_{d}'
			\end{pmatrix}=\begin{pmatrix}
				1\\1\\\vdots\\1
			\end{pmatrix}\]
			have a solution with $\frac{1}{2a_{ii}}\le g_i'\le \frac{1}{a_{ii}}$ for all $i\in[d]_0$. By \eqref{eq:bounds on a_ii}, $g_i'=\Theta_{f,d}\!\left(\mu^{-f+d+i-i\cdot\binom{f-1}{d-1}}\right)$. Setting $g\coleq1/g_d'$ and $g_i\coleq gg'_i$, we get
			\begin{equation}\label{eq:equations of g_i}A\begin{pmatrix}
					g_0\\g_1\\\vdots\\g_{d}
				\end{pmatrix}=\begin{pmatrix}
					g\\g\\\vdots\\g
			\end{pmatrix}\end{equation}
			with $g_d=1$ and \begin{align}\label{eq:bounds on g}g&=\Theta_{f,d}\!\left(\mu^{f-2d+d\cdot\binom{f-1}{d-1}}\right)\\\label{eq:bounds on g_i} g_i&=\Theta_{f,d}\!\left(\mu^{-d+i+(d-i)\binom{f-1}{d-1}}\right)\end{align} for all $i\in[d-1]_0$.
			\paragraph{}
		
			\renewcommand{\Gto}[1]{G^{(#1)}} 
			
			We want to find $f$-subsets $Y$ such that $G[Y]$ is a full $(\eps,\xi,f,d)$-complex. Assume that there is a function $\psi\colon \Gto f\to [0,1]$ such that for every $e\in\Gto d$, we have
			\[\sum_{Q'\in\Gto{f}(e)}\psi(Q'\cup e)=(g/2)n^{f-d}\]
			and $g_\ell/4\le \psi(Q)\le 1$ for all $Q\in\Gto f_\ell$. We can then take every $Q\in\Gto f$ independently into $Y$ with probability $\psi(Q)$. We will show now that in that case, $G[Y]$ is a full $(\eps,\xi,f,d)$-complex. 
			
			Indeed, we have $\E\!\left[\abs{G[Y]^{(f)}(e)}\right]=(g/2)n^{f-d}$. Thus, \Cref{theo: Chernoff} implies
			\[\P\!\left[\abs{G[Y]^{(f)}(e)}\ne \!\left(1\pm n^{-(f-d)/2.01}\right)(g/2)n^{f-d}\right]\le2\exp\!\left(-\frac{n^{-\frac{2}{2.01}(f-d)}}{3}(g/2)n^{f-d}\right)\le e^{-n^{0.004}}\]
			This shows that with high probability, $G[Y]$ is $\big(n^{-(f-d)/2.01}, g/2,f,d\big)$-regular and, hence, $(\eps,g/2,f,d)$-regular. Note that $g/2>\xi$ by~\eqref{eq:bounds on g}.
			
			For the density, let $e\in \Gto d$ and $Q\in\Gto{f+d}(e)$. Then
			\[\P\!\left[Q\in G[Y]^{(f+d)}(e)\right]=\prod_{Q'\in\binom{Q\cup e}{f}}\psi(Q')\ge(g_0/4)^{\binom{f+d}{f}}\overset{\eqref{eq:bounds on g_i}}\ge\mu^{d\binom{f-1}{d-1}\binom{f+d}{d}}.\]
			Since $\abs{\Gto{f+d}(e)}\ge \abs{\Gto{f+d}_d(e)}\overset{\eqref{eq: size of G^(f+d)(e)}}\ge\mu^{f+d\cdot\binom{f+d-1}{d-1}}\cdot n^f$, \[\E\!\left[G[Y]^{(f+d)}(e)\right]\ge \mu^{f+d\binom{f-1}{d-1}\binom{f+d}{d}+d\cdot\binom{f+d-1}{d-1}}n^f\ge \mu^{1.5d\binom{f-1}{d-1}\binom{f+d}{d}}n^f.\]
			Here, we want to apply McDiarmid's Inequality (\Cref{lem: McDiarmid}). For each $Q\in G^{(f)}$, let $X_Q$ be the indicator variable that $Q$ is in~$Y$. If $\abs{Q\cap e}=\ell$, then changing the value of $X_Q$ changes $\abs{G[Y]^{(f+d)}(e)}$ by at most~$n^{\ell}$. Since there are at most $n^{f-\ell}$ many such $f$-sets $Q$, McDiarmid's Inequality implies
			\[\resizebox{\linewidth}{!}{$\P\!\left[\abs{G[Y]^{(f+d)}(e)}\le\mu^{2d\binom{f-1}{d-1}\binom{f+d}{d}}n^f\right]\le2\exp\!\left(-\E\!\left[\abs{G[Y]^{(f+d)}(e)}\right]^2/\sum_{\ell=0}^dn^{2\ell}\cdot n^{f-\ell}\right)\le \exp\!\left(-n^{f-d-0.01}\right).$}\]
			Therefore, with high probability, $G[Y]$ is $(\xi, f+d,d)$-dense.
			
			Finally, we have to check extendability. Fix any set $X\subs V(G)$ of size $\xi n$. Let $e\in \binom{X}{d}$. By \eqref{eq: size of E_f,ell(e) imprecise}, the number of $(f-d)$-sets $Q\subs V(G)\backslash e$ such that $\binom{Q\cup e}{d}\backslash\lbrace e\rbrace\subs\Gto d\subs G[Y]^{(d)}$ is at least $\mu^{f+d\cdot\binom{f-1}{d-1}}\cdot n^{f-d}\ge \xi n^{f-d}$. Therefore $G[Y]$ is $(\xi, f,d)$-extendable which concludes the proof that $G$ is an $(\eps,\xi, f,d)$-complex.
			
			It remains to show that such a $\psi$ exists. By \Cref{prop: gadget}, for every $e\in\Gto d$ and $J\in\Gto{f+d}(e)$, there exists a function $\psi_{e,J}\colon\Gto f\to\mathbb R$ such that
			\begin{enumerate}
				\item\label{enum: psi_e,J=0 outside} $\psi_{e,J}(Q)=0$ for all $Q\not\subs e\cup J$;
				\item\label{enum: sum of psi_e,J either 0 or 1} for all $e'\in\Gto d$, $\sum_{Q'\in\Gto f(e')}\psi_{e,J}(Q'\cup e')=\begin{cases}
					1,&e'=e\\
					0,&e'\ne e;
				\end{cases}$
				\item\label{enum: psi_e,J is small} for all $Q\in\Gto f$, $\abs{\psi_{e,J}(Q)}\le\frac{2^{d-j(d-j)!}}{\binom{f-d+j}{j}}$, where $j\coleq\abs{e\cap Q}$.
			\end{enumerate}
			We want to define a function $\psi\colon \Gto f\to [0,1]$ such that for each $e\in \Gto d$, we get
			\[\sum_{Q'\in \Gto f(e)}\psi(Q'\cup e)=(g/2)n^{f-d}.\]
			Let $\psi'\colon\Gto f\to[0,1]$ be the function which maps an $f$-set of type $i$ to $g_i/2$.
			
			For every $e\in \Gto d_i$, we define
			\[c_e\coleq \frac{(g/2)n^{f-d}-\sum_{j=i}^d(g_j/2)\abs{\Gto f_j(e)}}{\abs{\Gto{f+d}_{i}(e)}}.\]
			Note that by \eqref{eq: size of G^(f)(e)} and \eqref{eq:equations of g_i}, we have $\sum_{j=i}^dg_j\abs{\Gto f_j(e)}=(1\pm \alpha)\sum_{j=i}^dg_ja_{ij}n^{f-d}=\big(1\pm\alpha\big)gn^{f-d}$ whence \[\abs{c_e}\le \frac{\alpha g n^{f-d}}{2\abs{\Gto{f+d}_i(e)}}\overset{\eqref{eq: size of G^(f+d)(e)}}=\O_{c}\!\left(\alpha g\mu^{-f+i-i\binom{f+d-1}{d-1}}\cdot n^{-d}\right).\]
			Define $\psi_i\colon\Gto f\to[0,1]$ as \[\psi_i\coleq\sum_{e\in\Gto d_i}c_e\sum_{J\in \Gto{f+d}_i(e)}\psi_{e,J}\] for every $i\in[d]_0$. Note that by \ref{enum: psi_e,J=0 outside}, $\psi_i(Q)=0$ for all $f$-sets $Q$ of type larger than~$i$. Finally let \[\psi\coleq\psi'+\sum_{i=0}^d\psi_i.\]
			
			For every $e\in\Gto d_i$, we have
			\begin{align*}
				\sum_{Q'\in \Gto f(e)}\psi(Q'\cup e)&=\sum_{j=i}^d(g_j)/2\abs{\Gto f_j(e)}+\sum_{i'=0}^d\sum_{e'\in\Gto d_{i'}}c_{e'}\sum_{J\in \Gto{f+d}_{i'}(e')}\sum_{Q'\in\Gto f(e)}\psi_{e',J}(Q'\cup e)\\
				&\overset{\ref{enum: sum of psi_e,J either 0 or 1}}{=}\sum_{j=i}^d(g_j)/2\abs{\Gto f_j(e)}+c_e\abs{\Gto{f+d}_i(e)}=(g/2)n^{f-d},
			\end{align*}
			as desired. The last thing to check is that $\psi(Q)$ is between $g_\ell/4$ and 1 for all $Q\in \Gto f_\ell$. Let $Q\in\Gto f_\ell$. Then
			\begin{align*}
				\abs{\psi(Q)-g_\ell/2}&=\abs{\sum_{i=\ell}^d\psi_i(Q)}\overset{\ref{enum: psi_e,J=0 outside}}{\le}\sum_{i=\ell}^d\sum_{\substack{e\in\Gto d_i,J\in\Gto{f+d}_i(e)\\Q\subs e\cup J}}\abs{c_e}\abs{\psi_{e,J}(Q)}\\
				&\overset{\ref{enum: psi_e,J is small}}\le\sum_{i=\ell}^d\abs{\Gto{f+d}_i(Q)}\cdot\O_{c}\!\left(\alpha g\mu^{-f+i-i\binom{f+d-1}{d-1}}\cdot n^{-d}\right)\\
				&\overset{\eqref{eq: size of G^(f+d)(Q)}}{\le}\sum_{i=\ell}^d\O_{C}\!\left(\mu^{f-i+\ell-\ell\cdot\binom{f-1}{d-1}+i\cdot\binom{f+d-1}{d-1}}\cdot n^d\right)\cdot\O_{c}\!\left(\alpha g\mu^{-f+i-i\binom{f+d-1}{d-1}}\cdot n^{-d}\right)\\&=\sum_{i=\ell}^d\O_{c,C}\!\left(\alpha g\mu^{\ell-\ell\cdot\binom{f-1}{d-1}}\right)\overset{\eqref{eq:bounds on g}}{=}\O_{c,C}\!\left(\alpha \cdot \mu^{f-2d+d\cdot\binom{f-1}{d-1}+\ell-\ell\cdot\binom{f-1}{d-1}}\right)\\&=\O_{c,C}\!\left(\alpha \cdot\mu^{f-d}\cdot \mu^{-d+\ell+(d-\ell)\cdot\binom{f-1}{d-1}}\right)\overset{\eqref{eq:bounds on g_i}}{=}\O_{c,C}\!\left(\alpha \cdot\mu^{f-d}\cdot g_\ell\right)
			\end{align*}
			If $\alpha$ and $\mu$ are small enough, this is smaller than $g_\ell/4$.
		\end{proof}
		
		\begin{proof}[Proof of \Cref{lem:skew_super_complex}]
			Let $G'\subs \Gskew$ be the subcomplex given by \Cref{lem: Gskew has a subcomplex full of advantages}. We claim that $G'$ is an $(\eps,\xi,f,d)$-supercomplex. Let $h\in[d]_0$, $B\subs G'^{(h)}$ with $1\le\abs B\le 2^h$ and $G''\coleq\bigcap_{b\in B}G'(b)$. We have to show that $G''$ is an $(\eps,\xi,f-h,d-h)$-complex. By \Cref{lem: Gskew has a subcomplex full of advantages}, we know that $G''$ is $(\alpha',\mu,f',d',c,C)$-advantageous where $d'\coleq d-h$, $f'\coleq f-h$, $\alpha'\coleq f\mu^{-1}(2^{d+1}\cdot d)^d\alpha$, $n'\coleq n-\abs{\bigcup_{b\in B}b}$. By \Cref{lem:advantageous implies complex}, this implies that $G''$ is an $(\eps,\xi,f',d')$-complex. Thus, $G'$ is an $(\eps,\xi,f,d)$-supercomplex which is $(D,\rho)$-rich by \Cref{lem: Gskew has a subcomplex full of advantages}.
			
			Applying \Cref{cor: supercomplex sparsification} on $G'$ yields the sparse supercomplex we are looking for.
		\end{proof}
		
		\section{Concluding Remarks} \label{sec: Conclusion}
		
		\paragraph{Refined absorption.}
		In the proof of \Cref{theo: large min degree implies extra-tight trail} one could potentially replace the use of the iterative absorption method in \Cref{sec: Cover Down} with ``refined'' absorbers, developed recently by Delcourt and Postle~\cite{delcourt2024refined}.  
		Very roughly speaking, they prove the existence of an absorber $A$ together with a collection of cliques such that for any divisible leftover $L$, one can decompose $A\cup L$ using only cliques from the specified collection. Crucially, this collection of cliques can be chosen so sparse that every edge lies only in a constant number of them. 
		Unfortunately for our application one would need an appropriate refined absorption statement with {\em extra-tight cycles} instead of cliques. While this sort of result quite possibly holds, it is not currently available in the literature. 
		Furthermore, even with such an alternative approach, the key novel contributions of our paper (the reduction of the diameter problem to a hypergraph decomposition result via turns, the construction of the switchers, and the approximate decomposition lemma with boosted parameters) would still be necessary.
		
		\paragraph{Inching towards geometry.}
		It is well known (cf.~\cite{santosdefinesH}) that for the purposes of the maximum diameter of a $d$-dimensional polytope with $n$ facets it is sufficient to consider those which are  {\em simple}, that is every vertex is contained in exactly $d$ of the facets.  
		Each vertex of a simple polytope $P$ can then be encoded as a $d$-subset of its $n$ facets, giving rise to a $d$-graph $\cF = \cF (P)$ on vertex set $[n]$. The dual graph $G(\cF)$ is exactly the vertex/edge graph of  $P$. Indeed, two vertices of $P$ are connected by an edge of $P$ if and only if the two corresponding $d$-sets have an intersection of size $d-1$ (corresponding to the $d-1$ facets of $P$ containing the edge). 
		Since the vertex/edge graph of $P$ is connected, so is $ G(\cF)$, and hence $H_s(n,d-1)$ is an upper bound on the diameter of polytopes.  
		The simplicial $d$-complexes $\langle \cF (P) \rangle$ coming from polytopes have in fact a much richer topological structure, beyond the connectivity of their dual graph. 
		As a first step, the diameter of pseudomanifolds without boundary were considered by Criado and Santos~\cite{criado2017maximum}. The simplicial $d$-complex generated by a $(d+1)$-graph $\cF$ is called a {\em pseudomanifold without boundary} if every member of $\partial \cF$ is contained in exactly two members of $\cF$. 
		After improvements of \cite{criado2021randomized}, Bohman and Newman~\cite{bohman2022complexes} determined the correct asymptotics of the maximum diameter $H_{pm}(n,d)$ of $d$-dimensional pseudomanifolds without boundary on $n$ vertices. 
		We can determine the precise value for $d=2$ and every large enough~$n$. A precise bound for $d\geq 3$ and every large enough $n$ seems to be quite a challenge.
		
		In another direction, the nature of our construction requires the lower bound $n_0$ on the number of vertices to be significantly large in terms of $d$. For the connection of the Polynomial Hirsch Conjecture to the Simplex Method, the particular range when $n$ is linear in $d$ is of quite a relevance, as many of the worst case examples are in there. Santos' product construction~\cite{santosdefinesH} does give an exponential lower bound in this range, but more precise, possibly exact results would be of great interest. 
		
		\paragraph{The optimal minimum degree condition for the existence of extra-tight tours.} 
		In relation to Theorem~\ref{cor: extra-tight tour covering everything} it would be very interesting to determine, or at least estimate, the supremum of those $\alpha >0$ for which every $d$-graph on 
		$n > n_0 (d,\alpha)$ vertices with $\delta (G) > (1-\alpha)n$, satisfying that every vertex has degree divisible by $d^2$ admits an extra-tight Euler tour. The answer is not even known for $d=2$.
		
		\section*{Acknowledgements}
		Stefan Glock is funded by the Deutsche Forschungsgemeinschaft (DFG, German Research Foundation) – 542321564.
		Silas Rathke is funded by the Deutsche Forschungsgemeinschaft (DFG, German Research Foundation) under Germany's Excellence Strategy – The Berlin Mathematics Research Center MATH+ (EXC-2046/1, project ID: 390685689).
		
		\bibliographystyle{amsplain}
		\bibliography{bib}

@misc{2-dimensional,
  author = {Olaf Parczyk and Silas Rathke and Tibor Szab\'{o}},
  title  = {The maximum diameter of $2$-dimensional simplicial complexes},
  howpublished = {arXiv:2511.10144},
  year   = {(2025)}
}

@article {ehard2020pseudorandom,
    AUTHOR = {Ehard, Stefan and Glock, Stefan and Joos, Felix},
     TITLE = {Pseudorandom hypergraph matchings},
   JOURNAL = {Combin. Probab. Comput.},
  FJOURNAL = {Combinatorics, Probability and Computing},
    VOLUME = {29},
      YEAR = {2020},
    NUMBER = {6},
     PAGES = {868--885},
      ISSN = {0963-5483,1469-2163},
   MRCLASS = {05C65 (05C15 05C70 05D15 05D40)},
  MRNUMBER = {4173135},
MRREVIEWER = {Ioan\ Tomescu},
       DOI = {10.1017/s0963548320000280},
       URL = {https://doi.org/10.1017/s0963548320000280},
}

@article {glock2020euler,
    AUTHOR = {Glock, Stefan and Joos, Felix and K\"uhn, Daniela and Osthus,
              Deryk},
     TITLE = {Euler tours in hypergraphs},
   JOURNAL = {Combinatorica},
  FJOURNAL = {Combinatorica. An International Journal on Combinatorics and
              the Theory of Computing},
    VOLUME = {40},
      YEAR = {2020},
    NUMBER = {5},
     PAGES = {679--690},
      ISSN = {0209-9683,1439-6912},
   MRCLASS = {05C65 (05B40 05C45 05C81)},
  MRNUMBER = {4181762},
MRREVIEWER = {Tomoki\ Yamashita},
       DOI = {10.1007/s00493-020-4046-8},
       URL = {https://doi.org/10.1007/s00493-020-4046-8},
}

@article {HirschCounterexample,
    AUTHOR = {Santos, Francisco},
     TITLE = {A counterexample to the {H}irsch conjecture},
   JOURNAL = {Ann. of Math. (2)},
  FJOURNAL = {Annals of Mathematics. Second Series},
    VOLUME = {176},
      YEAR = {2012},
    NUMBER = {1},
     PAGES = {383--412},
      ISSN = {0003-486X,1939-8980},
   MRCLASS = {52A37 (51M20 52A40 52B05 52B12)},
  MRNUMBER = {2925387},
MRREVIEWER = {J.\ B\"ohm},
       DOI = {10.4007/annals.2012.176.1.7},
       URL = {https://doi.org/10.4007/annals.2012.176.1.7},
}

@article {KalaiKleitman,
    AUTHOR = {Kalai, Gil and Kleitman, Daniel J.},
     TITLE = {A quasi-polynomial bound for the diameter of graphs of
              polyhedra},
   JOURNAL = {Bull. Amer. Math. Soc. (N.S.)},
  FJOURNAL = {American Mathematical Society. Bulletin. New Series},
    VOLUME = {26},
      YEAR = {1992},
    NUMBER = {2},
     PAGES = {315--316},
      ISSN = {0273-0979,1088-9485},
   MRCLASS = {52B55},
  MRNUMBER = {1130448},
MRREVIEWER = {W.\ J.\ Firey},
       DOI = {10.1090/S0273-0979-1992-00285-9},
       URL = {https://doi.org/10.1090/S0273-0979-1992-00285-9},
}

@article {dkebski2017harmonious,
    AUTHOR = {D{\k{e}}bski, Micha{\l} and Lonc, Zbigniew and Rz{\k{a}}{\.z}ewski, Pawe{\l}},
     TITLE = {Harmonious and achromatic colorings of fragmentable
              hypergraphs},
   JOURNAL = {European J. Combin.},
  FJOURNAL = {European Journal of Combinatorics},
    VOLUME = {66},
      YEAR = {2017},
     PAGES = {60--80},
      ISSN = {0195-6698,1095-9971},
   MRCLASS = {05C15 (05C65)},
  MRNUMBER = {3692137},
MRREVIEWER = {Christina\ Zarb},
       DOI = {10.1016/j.ejc.2017.06.013},
       URL = {https://doi.org/10.1016/j.ejc.2017.06.013},
}

@article {CHUNG199243,
    AUTHOR = {Chung, Fan and Diaconis, Persi and Graham, Ron},
     TITLE = {Universal cycles for combinatorial structures},
   JOURNAL = {Discrete Math.},
  FJOURNAL = {Discrete Mathematics},
    VOLUME = {110},
      YEAR = {1992},
    NUMBER = {1-3},
     PAGES = {43--59},
      ISSN = {0012-365X,1872-681X},
   MRCLASS = {05A99 (05C20)},
  MRNUMBER = {1197444},
MRREVIEWER = {Giuseppe\ Pellegrino},
       DOI = {10.1016/0012-365X(92)90699-G},
       URL = {https://doi.org/10.1016/0012-365X(92)90699-G},
}

@book {janson2011random,
    AUTHOR = {Janson, Svante and {\L}uczak, Tomasz and Rucinski, Andrzej},
     TITLE = {Random graphs},
    SERIES = {Wiley-Interscience Series in Discrete Mathematics and
              Optimization},
 PUBLISHER = {Wiley-Interscience, New York},
      YEAR = {2000},
     PAGES = {xii+333},
      ISBN = {0-471-17541-2},
   MRCLASS = {05C80 (60C05 82B41)},
  MRNUMBER = {1782847},
MRREVIEWER = {Mark\ R.\ Jerrum},
       DOI = {10.1002/9781118032718},
       URL = {https://doi.org/10.1002/9781118032718},
}

@article {glock2023existence,
    AUTHOR = {Glock, Stefan and K\"uhn, Daniela and Lo, Allan and Osthus,
              Deryk},
     TITLE = {The existence of designs via iterative absorption: hypergraph
              {$F$}-designs for arbitrary {$F$}},
   JOURNAL = {Mem. Amer. Math. Soc.},
  FJOURNAL = {Memoirs of the American Mathematical Society},
    VOLUME = {284},
      YEAR = {2023},
    NUMBER = {1406},
     PAGES = {v+131},
      ISSN = {0065-9266,1947-6221},
      ISBN = {978-1-4704-6024-2; 978-1-4704-7444-7},
   MRCLASS = {05-02 (05Bxx 05C65 05Dxx)},
  MRNUMBER = {4572074},
MRREVIEWER = {Junling\ Zhou},
       DOI = {10.1090/memo/1406},
       URL = {https://doi.org/10.1090/memo/1406},
}

@misc{bohman2022complexes,
  title={Complexes of nearly maximum diameter},
  author={Bohman, Tom and Newman, Andrew},
  howpublished={arXiv:2204.11932},
  year={(2022)}
}

@article {criado2021randomized,
    AUTHOR = {Criado, Francisco and Newman, Andrew},
     TITLE = {Randomized construction of complexes with large diameter},
   JOURNAL = {Discrete Comput. Geom.},
  FJOURNAL = {Discrete \& Computational Geometry. An International Journal
              of Mathematics and Computer Science},
    VOLUME = {66},
      YEAR = {2021},
    NUMBER = {2},
     PAGES = {687--700},
      ISSN = {0179-5376,1432-0444},
   MRCLASS = {05E45 (05C12 05D40)},
  MRNUMBER = {4292759},
MRREVIEWER = {Oana\ Stefania\ Olteanu},
       DOI = {10.1007/s00454-020-00248-2},
       URL = {https://doi.org/10.1007/s00454-020-00248-2},
}

@article {criado2017maximum,
    AUTHOR = {Criado, Francisco and Santos, Francisco},
     TITLE = {The maximum diameter of pure simplicial complexes and
              pseudo-manifolds},
   JOURNAL = {Discrete Comput. Geom.},
  FJOURNAL = {Discrete \& Computational Geometry. An International Journal
              of Mathematics and Computer Science},
    VOLUME = {58},
      YEAR = {2017},
    NUMBER = {3},
     PAGES = {643--649},
      ISSN = {0179-5376,1432-0444},
   MRCLASS = {55U10 (52B05 57Q05 90C05 90C60)},
  MRNUMBER = {3690665},
MRREVIEWER = {Nicholas\ A.\ Scoville},
       DOI = {10.1007/s00454-017-9888-5},
       URL = {https://doi.org/10.1007/s00454-017-9888-5},
}

@article {gishboliner2023tight,
    AUTHOR = {Gishboliner, Lior and Glock, Stefan and Sgueglia, Amedeo},
     TITLE = {Tight {H}amilton cycles with high discrepancy},
   JOURNAL = {Combin. Probab. Comput.},
  FJOURNAL = {Combinatorics, Probability and Computing},
    VOLUME = {34},
      YEAR = {2025},
    NUMBER = {4},
     PAGES = {565--584},
      ISSN = {0963-5483,1469-2163},
   MRCLASS = {05C35 (05C45 05C65)},
  MRNUMBER = {4929049},
       DOI = {10.1017/s0963548325000057},
       URL = {https://doi.org/10.1017/s0963548325000057},
}

@misc{Keevash,
  title={The existence of designs},
  author={Keevash, Peter},
  howpublished={arXiv:1401.3665},
  year={(2014)}
}

@article {barber2020minimalist,
    AUTHOR = {Barber, Ben and Glock, Stefan and K\"uhn, Daniela and Lo,
              Allan and Montgomery, Richard and Osthus, Deryk},
     TITLE = {Minimalist designs},
   JOURNAL = {Random Structures Algorithms},
  FJOURNAL = {Random Structures \& Algorithms},
    VOLUME = {57},
      YEAR = {2020},
    NUMBER = {1},
     PAGES = {47--63},
      ISSN = {1042-9832,1098-2418},
   MRCLASS = {05C51 (05B05 05C35 05C80)},
  MRNUMBER = {4120592},
MRREVIEWER = {J\'ozsef\ Balogh},
       DOI = {10.1002/rsa.20915},
       URL = {https://doi.org/10.1002/rsa.20915},
}

@article {santosdefinesH,
    AUTHOR = {Santos, Francisco},
     TITLE = {Recent progress on the combinatorial diameter
              of polytopes and simplicial complexes},
   JOURNAL = {TOP},
  FJOURNAL = {TOP. Transactions in Operations Research},
    VOLUME = {21},
      YEAR = {2013},
    NUMBER = {3},
     PAGES = {482--484},
      ISSN = {1134-5764,1863-8279},
   MRCLASS = {52B05 (90C05 90C60)},
  MRNUMBER = {3112027},
       DOI = {10.1007/s11750-013-0296-6},
       URL = {https://doi.org/10.1007/s11750-013-0296-6},
}

@article {BARBER2017148,
    AUTHOR = {Barber, Ben and K\"uhn, Daniela and Lo, Allan and Montgomery,
              Richard and Osthus, Deryk},
     TITLE = {Fractional clique decompositions of dense graphs and
              hypergraphs},
   JOURNAL = {J. Combin. Theory Ser. B},
  FJOURNAL = {Journal of Combinatorial Theory. Series B},
    VOLUME = {127},
      YEAR = {2017},
     PAGES = {148--186},
      ISSN = {0095-8956,1096-0902},
   MRCLASS = {05C65 (05C51 05C69 05C72)},
  MRNUMBER = {3704659},
MRREVIEWER = {Wei\ Gao},
       DOI = {10.1016/j.jctb.2017.05.005},
       URL = {https://doi.org/10.1016/j.jctb.2017.05.005},
}

@misc{polymath3,
 author = {Gil Kalai},
 title = {Polymath 3: Polynomial {H}irsch Conjecture}, 
 note = {\url{http://gilkalai.wordpress.com/2010/09/29/polymath-3-polynomial-hirsch-conjecture}}, 
 year = {2010}
}

@misc{gould2025advancing,
  title={Advancing the {R}ödl Nibble: New bounds on matchings and the list chromatic index of hypergraphs},
  author={Gould, Stephen and Kelly, Tom},
  howpublished={arXiv:2511.11375},
  year={(2025)}
}

@misc{delcourt2024refined,
  title={Refined absorption: A new proof of the existence conjecture},
  author={Delcourt, Michelle and Postle, Luke},
  howpublished={arXiv:2402.17855},
  year={(2024)}
}

@article {Todd2024,
    AUTHOR = {Todd, Michael J.},
     TITLE = {An improved {K}alai-{K}leitman bound for the diameter of a
              polyhedron},
   JOURNAL = {SIAM J. Discrete Math.},
  FJOURNAL = {SIAM Journal on Discrete Mathematics},
    VOLUME = {28},
      YEAR = {2014},
    NUMBER = {4},
     PAGES = {1944--1947},
      ISSN = {0895-4801,1095-7146},
   MRCLASS = {52B05 (05C12 90C05)},
  MRNUMBER = {3278836},
       DOI = {10.1137/140962310},
       URL = {https://doi.org/10.1137/140962310},
}

@article {Sukegawa2019,
    AUTHOR = {Sukegawa, Noriyoshi},
     TITLE = {An asymptotically improved upper bound on the diameter of
              polyhedra},
   JOURNAL = {Discrete Comput. Geom.},
  FJOURNAL = {Discrete \& Computational Geometry. An International Journal
              of Mathematics and Computer Science},
    VOLUME = {62},
      YEAR = {2019},
    NUMBER = {3},
     PAGES = {690--699},
      ISSN = {0179-5376,1432-0444},
   MRCLASS = {52B05 (52B11 90C05)},
  MRNUMBER = {3996942},
MRREVIEWER = {Lionel\ Pournin},
       DOI = {10.1007/s00454-018-0016-y},
       URL = {https://doi.org/10.1007/s00454-018-0016-y},
}

@article {Eisenbrand2010,
    AUTHOR = {Eisenbrand, Friedrich and H\"ahnle, Nicolai and Razborov,
              Alexander and Rothvo\ss, Thomas},
     TITLE = {Diameter of polyhedra: limits of abstraction},
   JOURNAL = {Math. Oper. Res.},
  FJOURNAL = {Mathematics of Operations Research},
    VOLUME = {35},
      YEAR = {2010},
    NUMBER = {4},
     PAGES = {786--794},
      ISSN = {0364-765X,1526-5471},
   MRCLASS = {52B05 (05B40 05C12 05C90)},
  MRNUMBER = {2777514},
MRREVIEWER = {Ren\ Ding},
       DOI = {10.1287/moor.1100.0470},
       URL = {https://doi.org/10.1287/moor.1100.0470},
}

@article {MR3152073,
    AUTHOR = {Kim, Edward D.},
     TITLE = {Polyhedral graph abstractions and an approach to the linear
              {H}irsch conjecture},
   JOURNAL = {Math. Program.},
  FJOURNAL = {Mathematical Programming},
    VOLUME = {143},
      YEAR = {2014},
    NUMBER = {1-2},
     PAGES = {357--370},
      ISSN = {0025-5610,1436-4646},
   MRCLASS = {05C12 (52B05 52B40 90C05)},
  MRNUMBER = {3152073},
MRREVIEWER = {Mathieu\ Dutour Sikiri\'c},
       DOI = {10.1007/s10107-012-0611-2},
       URL = {https://doi.org/10.1007/s10107-012-0611-2},
}
		
		\towriteornottowrite{
			\section{Proof of \Cref{lem: rooted embedding with U}}
			
			The following proof is almost verbatim the proof of~\cite[Lemma 5.20]{glock2023existence}.
			The differences are the following: 
			\begin{itemize}
				\item The letter $r$ is replaced by~$d$.
				\item Since we do not have the concept of a hull anymore, $K_j$ is no longer the hull of $(\phi_j(T_j),\Ima(\Lambda_j))$, but just the edge set of $\phi_j(T_j)$. 
				\item In the inductive argument, it is no longer assumed that $\phi_1,\dots,\phi_j$ obey (ii) of~\cite[Lemma 5.20]{glock2023existence}. Instead, they have to obey \ref{enum: rooted embedding only uses U}. 
				\item The letter $U$ in the old proof is replaced by $W$ since the letter $U$ is already taken. 
				\item When $T_j$ is iteratively embedded, one has to make sure that $\phi_j(v_\ell)$ is defined to be in~$U$. However, this is not a problem since the condition $\abs{\bigcap_{S\in A}G(S)}\ge \xi n$ of \cite[Lemma 5.20]{glock2023existence} is $\abs{\bigcap_{S\in A}G(S)\cap U}\ge \xi n$ in this new version.
				\item Since condition (ii) of \cite[Lemma 5.20]{glock2023existence} is gone, there is no need anymore to define $O_{r+1}$.
			\end{itemize}
			\begin{proof}
				For $j\in[m]$ and a set $S\subs V(G)$ with $\abs S\in[d-1]$, let \[\troot(S,j):=\abs{\{j'\in[j]:\Lambda_{j'} \mbox{ roots }S\}}.\] We will define $\phi_1,\dots,\phi_m$ successively. Once $\phi_j$ is defined, we let $K_j$ denote the edge set of $\phi_j(T_j)$.
				
				Suppose that for some $j\in[m]$, we have already defined $\phi_1,\dots,\phi_{j-1}$ such that $K_1,\dots,K_{j-1}$ are edge-disjoint, $\phi_{j'}(v)\in U$ for all $j'\in[j-1]$ and all $v\in V(T_{j'})\backslash X_{j'}$, and the following holds for $G_{j}:=\bigcup_{j'\in[j-1]}K_{j'}$, all $i\in[d-1]$ and all $S\in \binom{V(G)}{i}$:
				\begin{align}
					\abs{G_{j}(S)}\le \alpha \gamma^{(2^{-i})}n^{d-i}+ \big(\troot(S,j-1)+1\big)2^t.\label{embedding degree bound}
				\end{align}
				Note that \eqref{embedding degree bound} together with \eqref{eq: few roots with U} implies that for all $i\in[d-1]$ and all $S\in \binom{V(G)}{i}$, we have
				\begin{align}
					\abs{G_{j}(S)}\le 2\alpha\gamma^{(2^{-i})}n^{d-i}.\label{embedding degree bound new}
				\end{align}
				
				We will now define a $\Lambda_j$-faithful embedding $\phi_j$ of $(T_j,X_j)$ into $G$ such that $K_j$ is edge-disjoint from $G_j$, \ref{enum: rooted embedding only uses U} holds for $j$, and \eqref{embedding degree bound} holds with $j$ replaced by $j+1$. For $i\in[d-1]$, define $\BAD_i\coleq \{S\in\binom{V(G)}{i}:\abs{G_j(S)}\ge \alpha\gamma^{(2^{-i})}n^{d-i}\}$.
				We view $\BAD_i$ as an $i$-graph. We claim that for all $i\in[d-1]$,
				\begin{align}
					\Delta(\BAD_i)\le \gamma^{(2^{-d})}n.\label{degeneracy embedding max degree}
				\end{align}
				Consider $i\in[d-1]$ and suppose that there exists some $S\in\binom{V(G)}{i-1}$ such that $\abs{\BAD_i(S)}>\gamma^{(2^{-d})}n$. We then have that
				\begin{align*}
					\abs{G_j(S)}&=\frac{1}{d-i+1}\sum_{v\in V(G)\backslash S}\abs{G_j(S\cup \{v\})}\ge d^{-1} \sum_{v\in \BAD_i(S)}\abs{G_j(S\cup \{v\})}\\
					&\ge d^{-1}\abs{\BAD_i(S)} \alpha\gamma^{(2^{-i})}n^{d-i}  \ge d^{-1}\gamma^{(2^{-d})}n\alpha  \gamma^{(2^{-i})}n^{d-i}=d^{-1}\alpha \gamma^{(2^{-d}+2^{-i})}n^{d-(i-1)}.
				\end{align*}
				This contradicts \eqref{embedding degree bound new} if $i-1>0$ since $2^{-d}+2^{-i}<2^{-(i-1)}$. If $i=1$, then $S=\emptyset$ and we have $\abs{G_j}\ge d^{-1}\alpha\gamma^{(2^{-d}+2^{-1})}n^{d}$, which is also a contradiction since $\abs{G_{j}}\le m\binom{t}{d}\le \binom{t}{d} \alpha\gamma n^d$ and $2^{-d}+2^{-1}<1$ (as $d\ge 2$ if $i\in [d-1]$). This proves~\eqref{degeneracy embedding max degree}.
				
				We now embed the vertices of $T_j$ such that the obtained embedding $\phi_j$ is $\Lambda_j$-faithful. First, embed every vertex from $X_j$ at its assigned position. Since $T_j$ has degeneracy at most $D$ rooted at $X_j$, there exists an ordering $v_1,\dots,v_{k}$ of the vertices of $V(T_j)\backslash X_j$ such that for every $\ell\in[k]$, we have
				\begin{align}
					\abs{T_j[X_j\cup \{v_1,\dots,v_\ell\}](v_\ell)}\le D.\label{degeneracy ordering}
				\end{align}
				Suppose that for some $\ell\in [k]$, we have already embedded $v_1,\dots,v_{\ell-1}$ into~$U$. We now want to define $\phi_j(v_\ell)\in U$.
				Let $W\coleq\{\phi_j(v):v\in X_j\cup \{v_1,\dots,v_{\ell-1}\}\}$ be the set of vertices which have already been used as images for $\phi_j$.
				Let $A$ contain all $(d-1)$-subsets $S$ of $W$ such that $\phi_j^{-1}(S)\cup \{v_\ell\}\in T_j$. We need to choose $\phi_j(v_\ell)$ from the set $(\bigcap_{S\in A}G(S)\cap U)\backslash W$ in order to complete $\phi_j$ to an injective homomorphism from $T_j$ to~$G$. By~\eqref{degeneracy ordering}, we have $\abs A\le D$. Thus, by assumption, $\abs{\bigcap_{S\in A}G(S)\cap U}\ge \xi n$.
				
				For $i\in[d-1]$, let $O_i$ consist of all vertices $x\in V(G)$ such that there exists some $S\in \binom{W}{i-1}$ such that $S\cup \{x\}\in \BAD_i$ (so $\BAD_1= \binom{O_1}{1}$). We have $$\abs{O_i}\le \binom{\abs W}{i-1}\Delta(\BAD_i)\overset{\eqref{degeneracy embedding max degree}}{\le} \binom{t}{i-1}\gamma^{(2^{-d})}n.$$
				Let $O_d$ consist of all vertices $x\in V(G)$ such that $S\cup \{x\}\in G_j$ for some $S\in \binom{W}{d-1}$. By~\eqref{embedding degree bound new}, we have that $\abs{O_d}\le \binom{\abs{W}}{d-1}\Delta(G_j)\le  \binom{t}{d-1}2\alpha\gamma^{(2^{-(d-1)})}n \le \binom{t}{d-1} \gamma^{(2^{-d})}n$.
				
				Crucially, we have \[\abs{\bigcap_{S\in A}G(S)\cap U}-\abs W-\sum_{i=1}^{d}\abs{O_i} \ge \xi n-t-2^t\gamma^{(2^{-d})}n >0.\] Thus, there exists a vertex $x\in U$ such that $x\notin W\cup O_1\cup \dots \cup O_{d}$ and $S\cup \{x\}\in G$ for all $S\in A$. Define $\phi_j(v_\ell):=x$.
				
				Continuing in this way until $\phi_j$ is defined for every $v\in V(T_j)$ yields an injective homomorphism from $T_j$ to $G$ that satisfies \ref{enum: rooted embedding only uses U}. By definition of $O_d$, $K_j$ is edge-disjoint from $G_j$. It remains to show that \eqref{embedding degree bound} holds with $j$ replaced by $j+1$. Let $i\in[d-1]$ and $S\in \binom{V(G)}{i}$. If $S\not\in \BAD_i$, then we have $\abs{G_{j+1}(S)}\le \abs{G_{j}(S)}+\binom{t-i}{d-i}\le \alpha\gamma^{(2^{-i})}n^{d-i}+2^t$, so \eqref{embedding degree bound} holds. Now, assume that $S\in \BAD_i$. If $S\subs \Ima(\Lambda_{j})$ and $\abs{T_j(\Lambda_{j}^{-1}(S))}>0$, then $\troot(S,j)=\troot(S,j-1)+1$ and thus $\abs{G_{j+1}(S)}\le \abs{G_{j}(S)}+\binom{t-i}{d-i}\le  \alpha\gamma^{(2^{-i})}n^{d-i}+ (\troot(S,j-1)+1)2^t+\binom{t-i}{d-i} \le \gamma^{(2^{-i})}n^{d-i}+ (\troot(S,j)+1)2^t$ and \eqref{embedding degree bound} holds. Suppose next that $S\not\subs \Ima(\Lambda_{j})$. We claim that $S\not\subs V(\phi_j(T_j))$. Suppose, for a contradiction, that $S\subs V(\phi_j(T_j))$. Let $\ell\coleq\max\{\ell'\in[k]:\phi_j(v_{\ell'})\in S\}$. (Note that the maximum exists since $(S\cap V(\phi_j(T_j)))\backslash \Ima(\Lambda_{j})$ is not empty.) Hence, $x:=\phi_j(v_\ell)\in S$. Recall that when we defined $\phi_j(v_\ell)$, $\phi_j(v)$ had already been defined for all $v\in X_j\cup \{v_1,\dots,v_{\ell-1}\}$ and hence $S\backslash\{x\}\subs W$. But since $S\in \BAD_i$, we have $x\in O_i$, in contradiction to $x=\phi_j(v_\ell)$. Thus, $S\not\subs V(\phi_j(T_j))=V(K_j)$, which clearly implies that $\abs{G_{j+1}(S)}=\abs{G_j(S)}$ and \eqref{embedding degree bound} holds. The last remaining case is if $S\subs \Ima(\Lambda_{j})$ but $\abs{T_j(\Lambda_{j}^{-1}(S))}=0$. But then $\abs{K_j(S)}=0$ and therefore $\abs{G_{j+1}(S)}=\abs{G_j(S)}$ as well.
				
				Finally, if $j=m$, then the fact that \eqref{embedding degree bound} holds with $j$ replaced by $j+1$ together with \eqref{eq: few roots with U} implies that $\Delta(\bigcup_{j\in[m]}\phi_j(T_j))\le 2 \alpha\gamma^{(2^{-(d-1)})}n \le \alpha\gamma^{(2^{-d})} n$.
			\end{proof}
			\section{Proof of \Cref{lem: make divisible}}
			\newcommand{\Set}[1]{\{#1\}}
			\def\sm{$\backslash$}
			\def\In{\subseteq}
			
			The following is taken almost verbatim from section 11 of \cite{glock2023existence}. The differences are the following:
			\begin{itemize}
				\item The letter $r$ is replaced by $d$.
				\item We only state the things from section 11 of \cite{glock2023existence} that are necessary to understand the proof of \Cref{lem: make divisible}. In particular, all proofs of the preceding lemmas are omitted.
				\item Instead of \cite[Lemma 5.20]{glock2023existence}, we use \Cref{lem: rooted embedding with U} to embed the shifters $T_k$ into $G$ such that each element in $V(T_k)\backslash X_k$ is mapped into $U\subs V(G)$. That way, we can ensure in the end that $H$ has an $F$-decomposition where each copy of $F$ has at most $d$ vertices outside of $U$. 
				\item We introduce the letters $\gamma'$ and $\gamma''$ since we no longer care about the exponents of $\gamma$.
				\item Claim 2 is removed because our statement does not state that the $F$-decomposition must be 1-well separated.
				\item Since \Cref{lem: make divisible} is less general than the original, we do the last part only for $G-D$ and not for general $H$.
			\end{itemize}
			\begin{definition}[Definition 11.1 in \cite{glock2023existence}]
				Let $1\le k<d$ and let $F,F^\ast$ be $d$-graphs.
				Given a $d$-graph $T_k$ and distinct vertices $x^0_1, \dots, x^0_k  , x^1_{1},\dots, x^1_{k}$ of $T_k$, we say that $T_k$ is an \emph{$( x^0_1, \dots, x^0_k ,  x^1_{1},\dots, x^1_{k} )$-shifter with respect to $F,F^\ast$} if the following hold:
				\begin{enumerate}[label=\rm{(SH\arabic*)}]
					\item $T_k$ has a $1$-well separated $F$-decomposition $\cF$ such that for all $F'\in \cF$ and all $i\in[k]$, $|V(F')\cap\Set{x_i^0,x_i^1}|\le 1$;\label{shifter decomposable}
					\item $|T_k(S)|\equiv 0\mod{\Deg(F^\ast)_{|S|}}$ for all $S\In V(T_k)$ with $|S|<k$;\label{shifter low degrees}
					\item for all $S\in \binom{V(T_k)}{k}$, \begin{align*}
						|T_k(S)| \equiv
						\begin{cases}
							(-1)^{\sum_{i \in [k]}z_i } \Deg(F)_{k} \mod{ \Deg(F^\ast)_{k} }& \text{if $S = \{ x_i^{z_i} : i \in [k] \}$, }\\
							0 \mod{\Deg(F^\ast)_{k} }& \text{otherwise.}
						\end{cases}
					\end{align*}\label{shifter degrees}
				\end{enumerate}
			\end{definition}
			\begin{lemma}[Lemma 11.2 in \cite{glock2023existence}\label{lem:shifters exist}]
				Let $1\le k<d$, let $F,F^\ast$ be $d$-graphs and suppose that $F^\ast$ has a $1$-well separated $F$-decomposition $\cF$. Let $f^\ast:=|V(F^\ast)|$. There exists an $(  x^0_1, \dots, x^0_k  ,  x^1_{1},\dots, x^1_{k}  )$-shifter $T_k$ with respect to $F,F^\ast$ such that $T_k[X]$ is empty and $T_k$ has degeneracy at most $\binom{f^\ast-1}{d-1}$ rooted at $X$, where $X:=\Set{x^0_1, \dots, x^0_k ,x^1_{1},\dots, x^1_{k}}$.
			\end{lemma}
			Let $\phi: \binom{V}{d} \rightarrow \mathbb{Z}$. (Think of $\phi$ as the multiplicity function of a multi-$d$-graph.)
			We extend $\phi$ to $\phi : \bigcup_{ k \in [d]_0 } \binom{V}{k} \rightarrow \mathbb{Z}$ by defining for all $S \subseteq V$ with $|S| = k \le d$,
			\begin{align}
				\phi (S) := \sum_{S' \in \binom{V}{d} : S \subseteq S'} \phi (S').\label{function extension to lower sets}
			\end{align}
			Thus for all $0\le i \le k \le d$ and all $S \in \binom{V}{i}$,
			\begin{align}
				\binom{d-i  }{k-i} \phi (S) = \sum_{S' \in \binom{V}{k} : S \subseteq S'} \phi (S').\label{handshaking for functions}
			\end{align}

			For $k \in [d-1]_0$ and $b_0,\dots,b_k \in \mathbb{N}$, we say that $\phi$ is \emph{$(b_0, \dots, b_k)$-divisible} if $b_{|S|} \mid \phi(S)$ for all $S\In V$ with $|S|\le k$.
			
			If $G$ is a $d$-graph with $V(G)\In V$, we define $\mathds{1}_G\colon \binom{V}{d}\to \Z$ as
			$$\mathds{1}_G(S):=\begin{cases} 1 & \mbox{if }S\in G;\\ 0 & \mbox{if }S\notin G.\end{cases}$$
			and extend $\mathds{1}_G$ to $\bigcup_{ k \in [d]_0 } \binom{V}{k}$ as in \eqref{function extension to lower sets}.
			Hence, for a set $S\In V$ with $|S|<d$, we have $\mathds{1}_G(S)=|G(S)|$. Thus, \eqref{handshaking for functions} corresponds to the handshaking lemma for $d$-graphs.
			
			Clearly, if $G$ and $G'$ are edge-disjoint, then we have $\mathds{1}_G+\mathds{1}_{G'}=\mathds{1}_{G\cup G'}$.
			Moreover, for a $d$-graph $F$, $G$ is $F$-divisible if and only if $\mathds{1}_G$ is $(\Deg(F)_0,\dots,\Deg(F)_{d-1})$-divisible.
			
			\begin{definition}[Definition 11.6 in \cite{glock2023existence}]\label{def:adapter}
				Let $V$ be a vertex set and $k,d,b_0,\dots,b_k,h_k\in \N$ be such that $k<d$ and $h_k\mid b_k$.
				For distinct vertices $x^0_1, \dots, x^0_{k}, x^1_{1},\dots, x^1_{k} $ in $V$, we say that $\tau\colon \binom{V}{d}\to \Z$ is an \emph{$( x^0_1, \dots, x^0_k  ,  x^1_{1},\dots, x^1_{k} )$-adapter with respect to $(b_0,\dots,b_k;h_k)$} if $\tau$ is $(b_0,\dots,b_{k-1})$-divisible and for all $S \in \binom{V}{k}$,
				\begin{align*}
					\tau(S) \equiv
					\begin{cases}
						(-1)^{\sum_{i \in [k]}z_i }h_k \mod{ b_k }& \text{if $S = \{ x_i^{z_i} : i \in [k] \}$, }\\
						0 \mod{ b_k }& \text{otherwise.}
					\end{cases}
				\end{align*}
			\end{definition}
			
			\begin{fact}[Fact 11.7 in \cite{glock2023existence}]\label{fact:shifter is adapter}
				If $T$ is an $\mathbf{x}$-shifter with respect to $F,F^*$, then $\mathds{1}_T$ is an $\mathbf{x}$-adapter with respect to $(\Deg(F^\ast)_0,\dots,\Deg(F^\ast)_k;\Deg(F)_k)$.
			\end{fact}
			
			\begin{definition}[Definition 11.8 in \cite{glock2023existence}]\label{def:balancer}
				Let $d,k,b_0,\dots,b_k\in \N$ with $k<d$ and let $U,V$ be sets with $U\In V$. Let $\Omega_{k}$ be a multiset containing ordered tuples $\mathbf{x}=(x_1,\dots,x_{2k})$, where $x_1,\dots,x_{2k}\in U$ are distinct. We say that $\Omega_k$ is a \emph{$(b_0, \dots, b_{k})$-balancer for $V$ with uniformity $d$ acting on $U$} if for any $h_k\in \N$ with $h_k\mid b_k$, the following holds: let $\phi\colon \binom{V}{d} \rightarrow \mathbb{Z}$ be any $(b_0, \dots, b_{k-1},h_k)$-divisible function such that $S\In U$ whenever $S\in \binom{V}{k}$ and $\phi(S)\not\equiv 0\mod{b_k}$. There exists a multisubset $\Omega'$ of $\Omega_{k}$ such that $\phi + \tau_{\Omega'}$ is $(b_0, \dots, b_k)$-divisible, where $\tau_{\Omega'} := \sum_{\mathbf{x} \in \Omega'} \tau_{\mathbf{x}}$ and $\tau_{\mathbf{x}}$ is any $\mathbf{x}$-adapter with respect to $(b_0,\dots,b_k;h_k)$.

				For a set $S \in \binom{V}{k}$, let $\deg_{\Omega_{k}}(S)$ be the number of $\mathbf{x} = ( x_1, \dots, x_{2k} ) \in \Omega_k$ such that $|S \cap \{x_i,x_{i+k}\}| =1$ for all $i \in [k]$.
				Furthermore, we denote $\Delta (\Omega_{k})$ to be the maximum value of $\deg_{\Omega_{k}}(S)$ over all $S \in \binom{V}{k}$.
			\end{definition}
			
			\begin{lemma}[Lemma 11.9 in \cite{glock2023existence}]
				\label{lem:balancer}
				Let $1 \le k < d $.
				Let $b_0, \dots, b_k \in \mathbb{N}$ be such that $\binom{d-s}{k-s} b_s \equiv 0 \mod{b_{k}}$ for all $s \in [k]_0$.
				Let $U$ be a set of $n \ge 2k$ vertices and $U\In V$.
				Then there exists a $(b_0, \dots, b_{k})$-balancer~$\Omega_k$ for~$V$ with uniformity $d$ acting on $U$ such that $\Delta( \Omega_k) \le 2^{k}  (k!)^2b_k$.
			\end{lemma}
			
			\begin{fact}[Fact 11.10 in \cite{glock2023existence}]\label{fact:gcd connection}
				Let $F$ be a $d$-graph. Then for all $0\le i \le k<d$, we have $\binom{d-i}{k-i}\Deg(F)_i\equiv 0\mod{\Deg(F)_k}$.
			\end{fact}
			
			\begin{proof}[Proof of \Cref{lem: make divisible}]
				Let $x_1^0,\dots,x_{d-1}^0,x_1^1,\dots,x_{d-1}^1$ be distinct vertices (not in $V(G)$). For $k\in[d-1]$, let $X_k:=\Set{x_1^0,\dots,x_{k}^0,x_1^1,\dots,x_{k}^1}$.
				By Lemma~\ref{lem:shifters exist}, for every $k\in[d-1]$, there exists an $( x^0_1, \dots, x^0_k , x^1_{1},\dots, x^1_{k} )$-shifter $T_k$ with respect to $F,F^\ast$ such that $T_k[X_k]$ is empty and $T_k$ has degeneracy at most $\binom{f^\ast-1}{d-1}$ rooted at $X_k$. Note that \ref{shifter decomposable} implies that
				\begin{align}
					|T_k(\Set{x_i^0,x_i^1})|=0\mbox{ for all }i\in[k].\label{no degenerate root}
				\end{align}
				
				We may assume that there exist $t\ge \max_{k\in[d-1]}|V(T_k)|$, $\gamma'$, and $\gamma''$ such that $1/n\ll \gamma\ll \gamma'\ll\gamma'' \ll 1/t \ll \xi,1/f^\ast$.
				Let $\Deg(F) = (h_0, h_1, \dots, h_{d-1})$ and let $\Deg(F^\ast) = (b_0, b_1, \dots, b_{d-1})$. Since $F^\ast$ is $F$-decomposable and thus $F$-divisible, we have $h_k\mid b_k$ for all $k\in[d-1]_0$.
				
				By Fact~\ref{fact:gcd connection}, we have $\binom{d - i}{k - i} b_{i} \equiv 0 \mod{ b_{k} }$ for all $0 \le i \le k < d$. For each $k\in[d-1]$ with $h_k<b_k$, we apply Lemma~\ref{lem:balancer} to obtain a $(b_0, \dots, b_{k})$-balancer~$\Omega_k$ for~$V(G)$ with uniformity $d$ acting on $V(G)$ such that $\Delta( \Omega_k) \le 2^{k}(k!)^2 b_k$. For values of $k$ for which we have $h_k=b_k$, we let $\Omega_k:=\emptyset$.
				For every $k\in[d-1]$ and every $\mathbf{v}=(v_1,\dots,v_{2k})\in \Omega_k$, define the labelling $\Lambda_{\mathbf{v}}\colon X_k\to V(G)$ by setting $\Lambda_{\mathbf{v}}(x_i^0):=v_i$ and $\Lambda_{\mathbf{v}}(x_i^1):=v_{i+k}$ for all $i\in[k]$.
				
				For technical reasons, let $T_0$ be a copy of $F$ and let $X_0:=\emptyset$. Let $\Omega_0$ be the multiset containing $b_0/h_0$ copies of $\emptyset$, and for every $\mathbf{v}\in \Omega_0$, let $\Lambda_{\mathbf{v}}\colon X_0 \to V(G)$ be the trivial $G$-labelling of $(T_0,X_0)$. Note that $T_0$ has degeneracy at most $\binom{f^\ast-1}{d-1}$ rooted at $X_0$. Note also that $\Lambda_{\mathbf{v}}$ does not root any set $S\In V(G)$ with $|S|\in[d-1]$.
				
				We will apply Lemma~\ref{lem: rooted embedding with U} in order to find faithful embeddings of the $T_k$ into $G$ where each element in $V(T_k)\backslash X_k$ is mapped into $U\subs V(G)$. Let $\Omega:=\bigcup_{k=0}^{d-1}\Omega_k$. Let $\alpha:=\gamma^{-2}/n$.
				
				\begin{NoHyper}
					\begin{claim}\label{claim:rooted ok}
						For every $k\in[d-1]$ and every $S\In V(G)$ with $|S|\in[d-1]$, we have $|\set{\mathbf{v}\in \Omega_k}{\Lambda_{\mathbf{v}}\mbox{ roots }S}|\le d^{-1}\alpha \gamma n^{d-|S|}$. Moreover, $|\Omega_k|\le d^{-1}\alpha \gamma n^{d}$.
					\end{claim}
				\end{NoHyper}
				
				\claimproof
				Let $k\in[d-1]$ and $S\In V(G)$ with $|S|\in[d-1]$. Consider any $\mathbf{v}=(v_1,\dots,v_{2k})\in \Omega_k$ and suppose that $\Lambda_{\mathbf{v}}$ roots $S$, i.e.~$S\In \Set{v_1,\dots,v_{2k}}$ and $|T_k(\Lambda_{\mathbf{v}}^{-1}(S))|>0$. Note that if we had $\Set{x_i^0,x_i^1}\In \Lambda_{\mathbf{v}}^{-1}(S)$ for some $i\in[k]$ then $|T_k(\Lambda_{\mathbf{v}}^{-1}(S))|=0$ by \eqref{no degenerate root}, a contradiction. We deduce that $|S\cap \Set{v_i,v_{i+k}}|\le 1$ for all $i\in[k]$, in particular $|S|\le k$. Thus there exists $S'\supseteq S$ with $|S'|=k$ and such that $|S'\cap \Set{v_i,v_{i+k}}|=1$ for all $i\in[k]$. However, there are at most $n^{k-|S|}$ sets $S'$ with $|S'|=k$ and $S'\supseteq S$, and for each such $S'$, the number of $\mathbf{v}=(v_1,\dots,v_{2k})\in \Omega_k$ with $|S'\cap \Set{v_i,v_{i+k}}|=1$ for all $i\in[k]$ is at most $\Delta(\Omega_k)$. Thus, $|\set{\mathbf{v}\in \Omega_k}{\Lambda_{\mathbf{v}}\mbox{ roots }S}|\le n^{k-|S|}\Delta(\Omega_k)\le n^{d-1-|S|} 2^{k}(k!)^2 b_k \le d^{-1}\alpha \gamma n^{d-|S|}$. Similarly, we have $|\Omega_k|\le n^k\Delta(\Omega_k)\le d^{-1}\alpha \gamma n^{d}$.
				\endclaimproof
				
				Claim~\ref*{claim:rooted ok} implies that for every $S\In V(G)$ with $|S|\in [d-1]$, we have
				$$|\set{\mathbf{v}\in \Omega}{\Lambda_{\mathbf{v}}\mbox{ roots }S}|\le \alpha \gamma n^{d-|S|}-1,$$
				and we have $|\Omega|\le b_0/h_0+\sum_{k=1}^{d-1}|\Omega_k|\le \alpha \gamma n^d$.
				Therefore, by Lemma~\ref{lem: rooted embedding with U}, for every $k\in[d-1]_0$ and every $\mathbf{v}\in \Omega_k$, there exists a $\Lambda_{\mathbf{v}}$-faithful embedding $\phi_{\mathbf{v}}$ of $(T_k,X_k)$ into $G$, such that, letting $T_{\mathbf{v}}:=\phi_{\mathbf{v}}(T_k)$, the following hold:
				
				\begin{enumerate}[label=\rm{(\alph*)}]
					\item for all distinct $\mathbf{v}_1,\mathbf{v}_2\in \Omega$, $T_{\mathbf{v}_1}$ and $T_{\mathbf{v}_2}$ are edge-disjoint;\label{rooted embedding:disjoint new}
					\item $\Delta(\bigcup_{\mathbf{v}\in \Omega}T_{\mathbf{v}})\le \alpha \gamma' n$;\label{rooted embedding:maxdeg new}
					\item for all $\mathbf v\in\Omega$, $\phi_{\mathbf v}(v)\in U$ for all $v\in V(T_k)$.
				\end{enumerate}
				
				Let $$D:=\bigcup_{\mathbf{v}\in \Omega}T_{\mathbf{v}}.$$ By \ref{rooted embedding:maxdeg new}, we have $\Delta(D)\le \gamma^{-2}\gamma'\le\gamma''$. 
				
				For every $k\in[d-1]$ and $\mathbf{v}\in \Omega_k$, we have that $T_{\mathbf{v}}$ is a $\mathbf{v}$-shifter with respect to $F,F^\ast$ by definition of $\Lambda_{\mathbf{v}}$ and since $\phi_{\mathbf{v}}$ is $\Lambda_{\mathbf{v}}$-faithful. Thus, by Fact~\ref{fact:shifter is adapter},
				\begin{align}
					\mathds{1}_{T_{\mathbf{v}}}\mbox{ is a }\mathbf{v}\mbox{-adapter with respect to }(b_0,\dots,b_k;h_k).\label{shifter is adapter}
				\end{align}
				
				Clearly, for every $\mathbf{v}\in \Omega_0$, $T_{\mathbf{v}}$ is a copy of $F$ and thus has an $F$-decomposition $\cF_{\mathbf{v}}=\Set{T_{\mathbf{v}}}$. Moreover, for each $k\in[d-1]$ and all $\mathbf{v}=(v_1,\dots,v_{2k})\in \Omega_k$, $T_{\mathbf{v}}$ has an $F$-decomposition $\cF_{\mathbf{v}}$ by \ref{shifter decomposable}. Therefore,
				\begin{equation}\label{eq: all shifter dec}\text{for every $\Omega'\In \Omega$, $\bigcup_{\mathbf{v}\in \Omega'}T_{\mathbf{v}}$ has an separated $F$-decomposition.}
				\end{equation}
				
				In particular, $D$ is $F$-divisible and as $G$ is $F$-divisble, $G-D$ is $F$-divisible as well.
				
				We will now show that there exists a subgraph $D^\ast\In D$ such that $(G-D) \cup D^\ast$ is $F^\ast$-divisible and $D-D^\ast$ has an $F$-decomposition.
				
				We will inductively prove that the following holds for all $k\in[d-1]_0$:
				\begin{itemize}
					\item[SHIFT$_{k}$] there exists $\Omega_k^\ast\In \Omega_0\cup \dots \cup \Omega_k$ such that $\mathds{1}_{(G-D)\cup D_k^\ast}$ is $(b_0,\dots,b_k)$-divisible, where $D_k^\ast:=\bigcup_{\mathbf{v}\in \Omega_k^\ast}T_{\mathbf{v}}$.
				\end{itemize}
				We first establish SHIFT$_0$. Since $G-D$ is $F$-divisible, we have $|G-D|\equiv 0\mod{h_0}$. Since $h_0\mid b_0$, there exists some $0\le a< b_0/h_0$ such that $|G-D|\equiv ah_0 \mod{b_0}$.
				Let $\Omega_0^\ast$ be the multisubset of $\Omega_0$ consisting of $b_0/h_0-a$ copies of $\emptyset$. Let $D_0^\ast:=\bigcup_{\mathbf{v}\in \Omega_0^\ast}T_{\mathbf{v}}$. Hence, $D_0^\ast$ is the edge-disjoint union of $b_0/h_0-a$ copies of $F$. We thus have $|(G-D)\cup D_0^\ast|\equiv ah_0+|F|(b_0/h_0-a)\equiv ah_0+b_0-ah_0\equiv 0\mod{b_0}$. Therefore, $\mathds{1}_{(G-D)\cup D_0^\ast}$ is $(b_0)$-divisible, as required.
				
				Suppose now that SHIFT$_{k-1}$ holds for some $k \in [d-1]$, that is, there is $\Omega_{k-1}^\ast\In \Omega_0\cup \dots \cup \Omega_{k-1}$ such that $\mathds{1}_{(G-D)\cup D_{k-1}^\ast}$ is $(b_0,\dots,b_{k-1})$-divisible, where $D_{k-1}^\ast:=\bigcup_{\mathbf{v}\in \Omega_{k-1}^\ast}T_{\mathbf{v}}$.
				Note that $D_{k-1}^\ast$ is $F$-divisible by \eqref{eq: all shifter dec}. Thus, since both $G-D$ and $D_{k-1}^\ast$ are $F$-divisible, we have $\mathds{1}_{(G-D)\cup D_{k-1}^\ast}(S)=|((G-D)\cup D_{k-1}^\ast)(S)|\equiv 0\mod{h_k}$ for all $S\in \binom{V(G)}{k}$. Hence, $\mathds{1}_{(G-D)\cup D_{k-1}^\ast}$ is in fact $(b_0,\dots,b_{k-1},h_k)$-divisible. Thus, if $h_k=b_k$, then $\mathds{1}_{(G-D)\cup D_{k-1}^\ast}$ is $(b_0,\dots,b_k)$-divisible and we let $\Omega_k':=\emptyset$. Now, assume that $h_k<b_k$.
				Recall that $\Omega_k$ is a $(b_0, \dots, b_k)$-balancer and that $h_k\mid b_k$.
				Thus, there exists a multisubset $\Omega_k'$ of $\Omega_k$ such that the function $\mathds{1}_{(G-D)\cup D_{k-1}^\ast} + \sum_{\mathbf{v} \in \Omega_k'} \tau_{\textbf{v}}$ is $( b_0, \dots, b_k)$-divisible, where $\tau_{\textbf{v}}$ is any $\mathbf{v}$-adapter with respect to $(b_0, \dots, b_{k};h_k)$.
				Recall that by \eqref{shifter is adapter} we can take $\tau_{\textbf{v}}=\mathds{1}_{T_{\mathbf{v}}}$.
				In both cases, let
				\begin{align*}
					\Omega_k^\ast&:=\Omega_{k-1}^\ast \cup \Omega_k'\In \Omega_0\cup \dots \cup \Omega_k;\\
					D_k'&:=\bigcup_{\mathbf{v}\in \Omega_k'}T_{\mathbf{v}};\\
					D_k^\ast&:=\bigcup_{\mathbf{v}\in \Omega_k^\ast}T_{\mathbf{v}}=D_{k-1}^\ast \cup D_k'.
				\end{align*}
				Thus, $\sum_{\mathbf{v} \in \Omega_k'} \tau_{\textbf{v}} = \mathds{1}_{D_k'}$ and hence $\mathds{1}_{(G-D)\cup D_{k}^\ast}=\mathds{1}_{(G-D)\cup D_{k-1}^\ast}+ \mathds{1}_{D_k'}$ is $( b_0, \dots, b_k)$-divisible, as required.
				
				Finally, SHIFT$_{d-1}$ implies that there exists $\Omega_{d-1}^\ast\In \Omega$ such that $\mathds{1}_{(G-D)\cup D^\ast}$ is $(b_0,\dots,b_{d-1})$-divisible, where $D^\ast:=\bigcup_{\mathbf{v}\in \Omega_{d-1}^\ast}T_{\mathbf{v}}$. Clearly, $D^\ast\In D$, and we have that $(G-D) \cup D^\ast$ is $F^\ast$-divisible. Finally, by~\eqref{eq: all shifter dec}, $$D-D^\ast=\bigcup_{\mathbf{v}\in \Omega \sm \Omega_{d-1}^\ast}T_{\mathbf{v}}$$ has an $F$-decomposition $\cF$. By setting $H\coleq D-D^\ast$, we complete the proof.
			\end{proof}
		}{}
	\end{document}